\documentclass[11pt]{report}
\usepackage{latexsym}
\usepackage{amssymb}
\usepackage{amsmath}
\usepackage{amscd}
\usepackage{amsthm}
\RequirePackage{HensLaTexPack3}
\usepackage{mathptmx}
\usepackage{accents}
\usepackage{stackengine}
\usepackage{upgreek}
 \usepackage[all,cmtip]{xy}
 \usepackage{setspace,lipsum}
 \usepackage{graphicx}
\usepackage{caption}
\usepackage{subcaption}
\usepackage{wrapfig}
\usepackage[greek,english]{babel}
\usepackage{seqsplit}
\usepackage{afterpage}
\usepackage[letterspace=50]{microtype}
\usepackage{kantlipsum}

\newcommand{\singlespaced}{\renewcommand{\baselinestretch}{1}\normalfont}
\newtheorem{theorem}{Theorem}[section]
\newtheorem{corollary}[theorem]{Corollary}
\newtheorem{lemma}[theorem]{Lemma}
\newtheorem{proposition}[theorem]{Proposition}
\theoremstyle{definition}
\newtheorem{definition}[theorem]{Definition}
\theoremstyle{remark}
\newtheorem{remark}[theorem]{Remark}
\newtheorem{example}[theorem]{Example}
\newtheorem{exercise}[theorem]{Exercise}

\numberwithin{equation}{section}

\newcommand{\sslash}{\mathbin{/\mkern-6mu/}}

\renewcommand{\R}{{\tb R}}
\renewcommand{\C}{{\tb C}}

\newcommand{\pc}{{\mathbf P}}
\renewcommand{\bc}{{\mathbf B}}
\newcommand{\cc}{{\mathbf C}}
\renewcommand{\mc}{{\mathbf M}}
\newcommand{\nc}{{\mathbf N}}
\newcommand{\fc}{{\mathbf F}}
\newcommand{\gc}{{\mathbf G}}
\newcommand{\ic}{{\mathbf I}}

\newcommand{\kc}{{\mathbf K}}
\renewcommand{\sc}{{\mathbf S}}
\newcommand{\tc}{{\mathbf T}}

\newcommand{\uc}{{\mathbf U}}
\newcommand{\vc}{{\mathbf V}}
\newcommand{\wc}{{\mathbf W}}
\newcommand{\xc}{{\mathbf X}}
\newcommand{\zc}{{\mathbf Z}}
\newcommand{\hc}{{\mathbf H}}
\newcommand{\field}{{\tb k}}

\newcommand{\cl}{{\textrm{cl}}}
\newcommand{\orb}{\tm{Orb}}
\newcommand{\op}{^{op}}

\renewcommand{\d}{\mathrm{d}}
\renewcommand{\Com}{\mathrm{Com}}
\newcommand{\D}{\mathrm{D}}
\newcommand{\I}{\mathrm{I}}
\newcommand{\K}{\mathrm{{K}}}
\renewcommand{\H}{\mathrm{H}}
\newcommand{\Ker}{\mathrm {{K}}}
\renewcommand{\Im}{\mathrm{{I}}}

\newcommand{\Supp}{\mathrm{Supp}}
\def\thetitle{MATROIDS AND CANONICAL FORMS: THEORY AND APPLICATIONS}
\def\theauthor{Gregory F Henselman-Petrusek}
\def\theadvisor{Robert W Ghrist}

\def\theyear{2017}

\begin{document}

\pagenumbering{roman}
%\doublespaced
\large\newlength{\oldparskip}\setlength\oldparskip{\parskip}\parskip=.33in
\thispagestyle{empty}
\begin{center}
%\vspace*{\fill}
\thetitle

\theauthor

A DISSERTATION

in

Electrical and Systems Engineering
\end{center}

\noindent\singlespaced\large
Presented to the Faculties of the University of Pennsylvania in Partial
Fulfillment of the Requirements for the Degree of Doctor of Philosophy

%\doublespaced
%\large
\begin{center}
\theyear
\end{center}

\noindent Supervisor of Dissertation

\bigskip
\noindent\makebox[-0.06in][l]{\rule[2ex]{3in}{.3mm}}
\singlespaced
Robert W Ghrist, Andrea Mitchell University Professor of Mathematics and Electrical and Systems Engineering\\

\noindent Graduate Group Chairperson 

\bigskip
\noindent\makebox[-0.06in][l]{\rule[2ex]{3in}{.3mm}}
\singlespaced
Alejandro Ribeiro, Rosenbluth Associate Professor  of Electrical and Systems Engineering\\
\vspace{-.25cm}

\noindent
\singlespaced
Dissertation Committee:\\
Chair: Alejandro Ribeiro, Rosenbluth Associate Professor  of Electrical and Systems Engineering \\
Supervisor: Robert W Ghrist, Andrea Mitchell University Professor of Mathematics and Electrical and Systems Engineering\\
Member: Rakesh Vohra, George A. Weiss and Lydia Bravo Weiss University Professor of Economics and Electrical and Systems Engineering
\vspace*{\fill}

\normalsize\parskip=\oldparskip

\newpage
%\doublespaced

\chapter*{}

\pagenumbering{gobble}

\noindent MATROIDS AND CANONICAL FORMS: THEORY AND APPLICATIONS
\vspace{1cm}

\noindent COPYRIGHT 

\noindent 2017

\noindent Gregory F Henselman

\newpage
\pagenumbering{roman}
\setcounter{page}{3}

\begin{center}
\end{center}
\begin{center}
%DEDICATION
\end{center}

\vfill
\vfill

\begin{center}
\emph{For Linda}
\end{center}

\vfill
\vfill

\newpage

\begin{center}
%DEDICATION%: \emph{To my family}
\end{center}

\begin{center}
\vfill
\end{center}
%\begin{changemargin}{4.34cm}{4.34cm}
%\begin{changemargin}{-3.5cm}{-2.5cm}
\vspace{-.5cm}
\begin{center}
%{FOR LINDA}\vspace{1cm}
\vfill

\begin{tabular}{p{6.75cm}}
%\begin{tabular}{p{6.32cm}}

%To  Linda, John, Steve, Irene, Sarah, Jack,  Caitlin, and Jane
%\vspace{1.5cm}

\textgreek{\seqsplit{\doublespacing \lsstyle GIATHGIAGIAMOUPOUMOUEDWSETOMUALOGIATONPAPPOUMOUPOUMOUEDWSETOPNEUMATOUGIATHMHTERAMOUPOUMOUEDWSETHAGAPHTHSGIATONPATERAMOUPOUMOUEDWSETHNELPIDATOUGIATHNADELFHMOUPOUKRATAEITONDROMOMASKAIGIATONJAKOOPOIOSFWTIZEITAMATIAMASGIAKAITLINTOUOPOIOUHFILIAEINAIHKALUTERHMOUSUMBOULHGIAJANETHSOPOIASTOERGOEINAIAUTO}}
\vspace{.8cm}

%\hfill{Thank you.}
\vspace{5cm}

\end{tabular}

\end{center}

%\noindent \textgreek{\seqsplit{\doublespacing GIATHGIAGIAMOUPOUMOUEDWSETOMUALOGIATONPAPPOUMOUPOUMOUEDWSETOPNEUMATOUGIATHMHTERAMOUPOUMOUEDWSETHAGAPHTHSGIATONPATERAMOUPOUMOUEDWSETHNELPIDATOUGIATHNADELFHMOUPOUKRATAEITONDROMOMASKAIGIATONJAKOOPOIOSFWTIZEITAMATIAMASGIAKAITLINTOUOPOIOUHFILIAEINAIHKALUTERHMOUSUMBOULHGIAJANETHSOPOIASTOERGOEINAIAUTO}}
%\end{changemargin}

%Translation 1
%For Linda, who gave her love.
%For Irene, who gave her mind.
%For Steve, who gave his spirit.
%For John, who taught me hope.
%For Sarah, who keeps our way,
%and Jack who lightens our eyes.
%For Caitlin, whose counsel is love.
%For Frank, Jane, Pete, and Anna.
%For my mother.
%
%Translation 2
%For my mother, who gave her heart.
%For my grandmother, who gave her mind.
%For my grandfather, who gave his spirit.
%For my father, who taught me hope.
%For my sister, who keeps our way, and for Jack, who lightens our eyes.
%For Caitlin, whose counsel is love.
%For Linda.

\newpage
\begin{center}
ACKNOWLEDGEMENTS
\end{center}

I would like to thank my advisor, Robert Ghrist, whose encouragement and support made this work possible.   His passion for this field is inescapable, and his mentorship  gave purpose to  a host of burdens.

A tremendous debt of gratitude is owed to Chad Giusti, whose sound judgement has been a guiding light for the past nine years, and whose mathematical insights first led me to consider research in computation.

A similar debt is owed to the colleagues  who shared in the challenges of the doctoral program: Jaree Hudson, Kevin Donahue, Nathan Perlmutter, Adam Burns, Eusebio Gardella, Justin Hilburn, Justin Curry, Elaine So, Tyler Kelly, Ryan Rogers, Yiqing Cai, Shiying Dong, Sarah Costrell, Iris Yoon, and Weiyu Huang.   Thank you for your friendship and support  at the high points and the low.

This work has benefitted enormously from interaction with mathematicians who took it upon themselves to help my endeavors.  To Brendan Fong, David Lipsky, Michael Robinson, Matthew Wright, Radmila Sazdanovic, Sanjeevi Krishnan, Pawe\l{} Dlotko, Michael Lesnick, Sara Kalisnik, Amit Patel, and Vidit Nanda, thank you.  

Many faculty have made significant contributions to the content of this text, whether directly or indirectly.  Invaluable help came from Ortwin Knorr, whose instruction is a constant presence in my writing.    To Sergey Yuzvinsky, Charles Curtis, Hal Sadofsky, Christopher Phillips, Dev Sinha, Nicholas Proudfoot, Boris Botvinnik, David Levin, Peng Lu, Alexander Kleshchev, Yuan Xu, James Isenberg, Christopher Sinclair, Peter Gilkey, Santosh Venkatesh, Matthew Kahle, Mikael Vejdemo-Johansson, Ulrich Bauer, Michael Kerber, Heather Harrington, Nina Otter, Vladimir Itskov, Carina Curto, Jonathan Cohen, Randall Kamien,  Robert MacPherson, and Mark Goresky, thank you.  Special thanks to Klaus Truemper, whose text  opened the world of matroid decomposition to my imagination.

Special thanks are due, also, to my dissertation committee, whose technical insights continue to excite new ideas for the possibilities of topology in complex systems, and whose coursework led directly to my work in combinatorics.

Thank you to Mary Bachvarova, whose advice was decisive in my early graduate education, and whose friendship offered shelter from many storms.

Finally, this research was made possible by the encouragement, some ten years ago, of my  advisor Erin McNicholas, and of my friend and mentor Inga Johnson.  Thank you for bringing mathematics to life, and for showing me the best of what it could be.  Your knowledge put my feet on this path, and your faith was  the  reason I imagined following it.  

\vfill
\vfill
\vfill

%\clearpage\thispagestyle{empty}\mbox{}\clearpage

\newpage
\begin{center}
  ABSTRACT\\
\thetitle\\
\vspace{.5in}

  \theauthor\\
  \theadvisor
\end{center}
This thesis proposes a combinatorial generalization of a nilpotent operator on a vector space.  The resulting object is highly natural, with basic connections to a variety of fields in pure mathematics, engineering, and the sciences.  For the purpose of exposition we focus the discussion of applications on homological algebra and computation, with additional remarks in lattice theory, linear algebra, and abelian categories.  For motivation, we recall that the methods of algebraic topology have driven remarkable progress in the qualitative study of large, noisy bodies of data over the past 15 years.   A primary tool in \emph{Topological Data Analysis [TDA]} is the \emph{homological persistence module}, which leverages  categorical structure  to compare algebraic shape descriptors across multiple scales of measurement.   Our principle application to computation is  a novel algorithm to calculate persistent homology which, in certain cases, improves the state of the art by several orders of magnitude.  Included are novel results in discrete, spectral, and algebraic {Morse theory}, and on the strong maps of matroid theory.   The defining theme throughout is interplay between the \emph{combinatorial} theory matroids and the \emph{algebraic} theory of categories.   The nature of these interactions is remarkably simple, but their consequences in homological algebra, quiver theory, and combinatorial optimization represent new and widely open fields for  interaction between the disciplines.

%\doublespaced
\noindent

% Put the content of the abstract here

\vspace*{\fill}

\newpage

\tableofcontents

\listoftables

\listoffigures

\newpage
%\draftspaced
\pagenumbering{arabic}
\include{introdept}
\include{back}
\include{finitedept}
\include{infinitedept}

%\clearpage\thispagestyle{empty}\mbox{}\clearpage

\setcounter{page}{1}
\chapter{Introduction}
\label{ch:intro}

\subsubsection{Motivation: computational homology}
In the past fifteen years, there has been a steady advance in the use of techniques and principles from algebraic topology to address problems in the data sciences. This new subfield of \emph{Topological Data Analysis} [TDA] seeks to extract robust qualitative  features from large, noisy data sets. At the simplest and most elementary level, one has \emph{clustering}, which returns something akin to connected components. There are, however, many higher-order notions of global features in connected components of higher-dimensional data sets that are not describable in terms of clustering phenomena. Such ``holes'' in data are quantified and collated by the classical tools of algebraic topology: \emph{homology}, and the more recent, data-adapted, parameterized version, \emph{persistent homology} \cite{CarlssonTopology09,EHPersistent08,GhristBarcodes08}. Homology and persistent homology will be described in Chapter \ref{ch:efficienthomologycomputation}. For the moment, the reader may think of homology as an enumeration of ``holes'' in a data set, outfitted with the structure of a sequence of vector spaces whose bases identify and enumerate the ``essential'' holes in a data set.

This work has as its motivation and primary application, the efficient computation of homology, persistent homology, and higher variants (such as cellular sheaf cohomology \cite{Curry:2016:DMT:2979741.2979751}) for application to TDA. Computational homology is an intensely active area of research with a rich literature \cite{KMMComputational04,ELZTopological02,EHComputational10}. For this introductory summary, it suffices to outline a little of the terminology without delving into detail. Homology takes as its input a sequence of vector spaces $C=(C_k)$ and linear transformations $\partial_k\colon C_k\to C_{k-1}$, collectively known as a \emph{chain complex} that, roughly, describes how simple pieces of a space are assembled.

\begin{equation}
\label{eq:ChainComplex}
\displaystyle
\xymatrix{
C \, =  \, \,
\cdots \ar[r] & C_k \ar[r]^{\partial_k} & C_{k-1} \ar[r]^{\partial_{k-1}} &
\cdots \ar[r]^{\partial_2} & C_1 \ar[r]^{\partial_1} & C_0 \ar[r]^{\partial_0} & 0
} .
\end{equation}
\vspace{0cm}

The chain complex is the primal object in homological algebra, best seen as the higher-dimensional analogue of a graph together with its  adjacency matrix. This ``algebraic signal'' is compressed to a homological core through the standard operations of linear algebra: kernels and images of the boundary maps $\partial$.

The standard algorithm to compute homology of a chain complex is to compute the \emph{Smith normal form} of the aggregated boundary map $\partial:C\to C$, where one concatenates the individual terms of (\ref{eq:ChainComplex}) into one large vector space. This graded boundary map has a block structure with zero blocks on the block-diagonal (since $\partial_k\colon C_k\to C_{k-1}$) and is nonzero on the superdiagonal blocks. The algorithm for computing Smith normal form is a slight variant of the ubiquitous Gaussian elimination, with reduction to the normal form via elementary row and column operations. For binary field coefficients this reduction is easily seen to be of time-complexity $O(n^3)$ in the size of the matrix, with an expected run time of $O(n^2)$.

This is not encouraging, given the typical sizes seen in applications. One especially compelling motivation for computing homology comes from a recent set of breakthrough applications in neuroscience by Giusti et al. \cite{GPC+Clique15}, which uses homology to extract network structure from a collection of neurons based solely on a correlation matrix.  For a typical experimental rig of 250 neurons, the resulting 250-by-250 matrix leads to chain complexes whose net dimension (the size of the resulting boundary matrix) is in the hundreds of billions. Such values of $n$ frustrates the usual algorithms: see \cite{OPT+roadmap} for benchmarks, which as of January 2017 state that with best available software over a 1728-core Sandybridge cluster, the largest complex tested has net dimension $n\approx 3\times 10^9$. The motivation for and direct outcome of this thesis is an algorithm for the efficient computation of homology in larger systems, immediately useful in TDA for Neuroscience and more.

\subsubsection{Approach: three ingredients}
In order to achieve a breakthrough in computational speed and memory management, this thesis turns to increased abstraction as the means of ascent. There are three ingredients that, though all classical and well-known in certain sectors, are synthesized in a novel way here.  These ingredients are as follows.

\begin{enumerate}
\item {\bf Matrix Factorization:} This first ingredient is the most familiar and least surprising. As homology computation in field coefficients is little more than Gaussian reduction, one expects the full corpus of matrix factorization methods to weigh heavily in any approach. One novelty of this thesis is the reconciliation of matrix factorization with the more abstract (category-theoretic) approaches to homology, as well as to the preprocessing/reduction methods of \cite{KMMComputational04}.

\item {\bf Matroid Theory:} The hero of this story is matroid theory, the mathematical fusion of combinatorial geometry and optimization theory \cite{orientedmatroids93,OxleyMatroid11,truemperMatroidDecomposition}. Matroids have a rich history in combinatorics and combinatorial optimization, but are largely absent in the literature on computational homology. This thesis introduces the language and methods of matroids to computational algebraic topology, using this to relate combinatorial optimization to homology computation. As a result, a novel and wholly combinatorial approach to homology is derived, using filtrations (and bifiltrations) of matroids as an extension of the notion of a chain complex (\ref{eq:ChainComplex}) above.

\item {\bf Homological Algebra:} This core branch of Mathematics is simply the algebra of diagrams \cite{GMMethods03}. In its simplest emanation, one works with, first, sequences of vector spaces and linear transformations [chain complexes], generalizing to more intricate diagrams. The key functional tools of linear algebra --- kernels, images, cokernels, coimages, and the like --- yield inferential engines on diagrams, the simplest of which are homologies. One quickly sees that, as with combinatorial structures leading to matroids, diagrams of vector spaces rely little on the actual details of linear algebra: only the core constructs count. This prompts the usual ascension to \emph{abelian categories} with vector-space-like objects, and transformations with suitable notions of kernels and cokernels. This thesis synthesizes homological algebra with matroid and matrix factorization methods.
\end{enumerate}

As a simple example of the intersection of these three subfields, consider the classical, essential operation of \emph{pivoting} in Gaussian reduction. The correspondences alluded to above yield a reinterpretation of pivoting as (1) a manifestation of the \emph{Exchange Lemma} in a particular matroid, and (2) an instance of the \emph{Splitting Lemma} in the homological algebra of short exact sequences. Though not in itself a deep result, this observation points the way to deeper generalizations of matrix methods in homological and combinatorial worlds.

The Gaussian pivot, the exact splitting, and the matroid basis exchange, are a single event. Iteration of this step is the natural progression and, in these domains, points to the common thread connecting them with each other and with computational homology. This mutual intersection is \emph{Discrete Morse Theory} [DMT], a fairly novel generalization \cite{FormanMorse98,KozlovDiscrete05} of the classical Morse Theory on manifolds \cite{MilnorMorse63}. Discrete Morse Theory has been the basis for some of the most effective compression schemes for simplicial and cellular complexes, leading to novel algorithms \cite{MNMorse13} and software. This thesis unearths a previously unknown connection between Discrete Morse Theory and (1) the Schur complement in matrix factorization; (2) minimal basis representations and greedy optimization in matroid theory; and (3) exact seqeunces in homological algebra.

The technical details of how these subjects merge and react to yield dramatic improvements in homology computation are not difficult. Indeed, they are presaged in a canonical result seen by every undergraduate student of Mathematics that acts as a microcosm, combining elements of combinatorics, matrix factorization, and (hidden) homological algebra. This key foreshadowing is the \emph{Jordan Form}.

\subsubsection{Primal example: Jordan bases}
Consider the (abelian) category of finite-dimensional $\C$-linear vector spaces and $\C$-linear maps. The Jordan bases of a complex operator are well-known. Less well-known is the implicit matroid-theoretic structure of Jordan bases: this has not, as far as the author knows, appeared in published literature. On a formal level the connection is fundamental:  matroid theory is built on the study of \emph{minimal bases} (read: bases subordinate to a flag of vector spaces) and the study of Jordan forms centers on bases for flags stabilized by a linear operator.  That this connection is deep can be demonstrated by an application to the problem of computing Jordan bases using minimal bases, which  is simple enough to describe with no matroid-theoretic language at all, though some notation will be required.

The story begins with a standard reduction: to describe the Jordan bases of an arbitrary complex operator, it suffices to describe those of a nilpotent one, since to every operator corresponds a nilpotent with identical Jordan bases.  Therefore fix a $T$ such that $T^n = 0$, for some $n$.  The approach will hinge on the canonical projection map,  $q$, from the base space of $T$ to its quotient by the image of $T$.   A subset of the base space is \emph{$q$-independent} if it maps to an independent set under $q$.  Likewise, sets that map to bases are \emph{$q$-bases}.   These definitions are slightly nonstandard, but the reader who continues through the background section will see where they fit into ordinary combinatorial terminology.

For every real-valued function $\kc$ on the domain of $T$ and every finite subset $E$,  define the \emph{$\kc$-weight} of $E$  to be $\sum_{e \in E} \kc(e)$.     A $q$-basis $B$ has \emph{minimum weight}  if its weight is minimal among $q$-bases.   The specific weight function that will occupy our attention is the one uniquely defined by the condition  $\tm{Ker}(T^m) = \{\kc \le m\}$, for all nonnegative integers $m$.

For convenience, let the \emph{orbit} of $E$  be the set  of all nonzero vectors that may be expressed in form $T^m e$, for some $e$ in $E$, and the \emph{support} of a linear combination $\sum_{e \in E} \ak_e e$  be the set of all $e$ such that $\ak_e$ is nonzero.  We may now describe the relationship between $q$-bases and Jordan bases precisely.

\begin{proposition}
\label{prop:intro1}
The Jordan bases of a nilpotent operator $T$ are exactly the orbits of minimum-weight $q$-bases.
\end{proposition}
\newcommand{\orbb}{E}
\begin{proof}
Let $B$ be a $q$-basis of minimum weight, and  let $E$ be its orbit.   We will first show that $E$ is independent.

%To see that $\tm{Orb}(B)$ is independent, the best proof is an example.  Note that every linear combination in $\tm{Orb}(B)$ can be expressed in a form akin to
%$$
%T^m (u + Tv +T^2 w)
%$$
%i.e., in the form of some power of $T$ applied to a  linear combination in $\tm{Orb}(B)$, the terms of which include a nonzero scalar multiple of some basis vector, $u$.  Without loss of generality, none of the terms inside the parentheses vanish under $T^m$.  Let us write $\sk$ for this linear combination.

If some nontrivial linear combination in $E$ evaluates to zero then, by factoring out as many powers of $T$ as possible, the same can be expressed as $T^m \sk$, where $\sk$ is a linear combination in $\orbb$ whose support contains at least one element of $B$.  Let $S$ denote the support of $\sk$, and assume, without loss of generality, that no element of $S$ vanishes under $T^m$.
%  If $S$ is the set of all elements of $B$ contained in the support of $\sk$, we may assume that .
% Suppose, for a contradiction, that zero may be expressed as a nontrivial linear combination in $\orbb$, and express it in this form.  If $S$ is the intersection of $B$ with the support of $\sk$,  then $S$ is nonempty, and we may assume without loss of generality that none of its elements vanishes under $T^m$.

Since $q(\sk)$ lies in the span of $q(S \cap B)$, there is an $s$ in $S \cap B$ for which $q(S \cap B)$ and $q(S \cap B - \{s\} \cup \{\sk\})$ have equal spans.  Evidently, $B - \{s\} \cup \{\sk\}$ is a $q$-basis.  Since $T^m$ vanishes on $\sk$ and not on $s$, however, this new basis has weight strictly less than $B$.  This is impossible, given our starting hypothesis, so $E$ must be independent.

To see that $E$ has full rank, let $U$ denote the quotient of the base space by the span of $\orbb$.  Our operator  induces a nilpotent map on $U$, and the cokernel of that map is trivial if and only if $U$ is trivial.  By the First Isomorphism Theorem, the same cokernel may be identified with the quotient of the base space by the span of  $\orbb \cup \tm{Im}(T)$.  The latter vanishes, and so does $U$.  Thus $\orbb$ is a basis.  Evidently,  it is a Jordan basis.

This establishes that the orbit of every minimum-weight $q$-basis is a Jordan basis.  For the converse,  consider a bijection between Jordan bases and certain of the $q$-bases.  In the forward direction, this map sends $J$ to $J - \tm{Im}(T)$.  In reverse, it sends $J - \tm{Im}(T)$ to its orbit.  Thanks to  uniqueness of the Jordan decomposition, every $q$-basis in the image of this map will have equal weight.  At least one such will have minimum weight -- for we have shown that at least one Jordan basis can be expressed as the orbit of a minimum-weight $q$-basis  -- and therefore all do.  Thus every Jordan basis is the orbit of a minimum-weight $q$-basis.
%
%that every Jordan basis $J$ may be expressed as the orbit of $B_J = J - \tm{Im}(T)$.
%
%To prove the converse, note that $B = E - \tm{Im}(T)$.  Evidently, any minimum-weight $\ic$-basis may be recovered from its orbit this way.  Since $J-\tm{Im}(T)$ is an $\ic$-basis for every every Jordan basis $J$, and the weight of $J-\tm{Im}(T)$ does not depend on the Jordan basis selected, it follows that every $J - \tm{Im}(T)$ has minimum weight.
%
%To prove the converse, note that $J - \tm{Im}(T)$ is an $\ic$-basis for any Jordan basis $J$.  If that Jordan basis is the orbit of an $\ic$-basis, then, vacuously, removing the image of $T$ returns the the same.
%
%Every $\ic$-basis of this form will  evidently have equal weight.  At least of one Jordan basis is the orbit of a minimum of these comes from the orbit of a minimum-weight $\ic$-basis, and therefore has minimum weight.   Thus all do.
% note that at least one minimum weight $\ic$-basis exists, since $\fc$ takes finitely many values.  Every set that can be expressed  $J - \tm{Im}(T)$ for some Jordan basis $J$ will have the same $\fc$-weight, so $J - \tm{Im}(T)$ has minimal weight, for every Jordan basis.
\end{proof}

Conveniently,  minimum weight $q$-bases have a simple description.  Let $K_m$ denote the kernel of $T^m$, and put $B_m = B \cap K_m$

\begin{proposition}
\label{prop:intro2}
A $q$-basis $B$ has minimum weight  if and only if $q(B_m)$ spans $q(K_m$), for all $m$.
\end{proposition}

%\begin{proposition}  A $\ic$-basis $B$ has minimum weight if and only if $q(B_m)$ is a linear basis for $q(K_m)$, for all $m$.
%\end{proposition}
\begin{proof}
Fix any $q$-basis $B$, and suppose $q(B_m)$ does not span $q(K_m)$, for some $m$.  Then there exists $\sk \in K_m$ such that $q(\sk)$ lies outside the span of $q(B_m)$, and one can construct a new $q$-basis by replacing an element of $B - B_m$ with $\sk$.   The new basis will weigh strictly less than $B$, so $B$ is not minimal.
\end{proof}

The following corollary is an immediate consequence.  To save unnecessary verbiage, let us say that a set \emph{represents} a basis in a quotient space if it maps to one under the canonical projection map.

%\begin{corollary}  A set $I$ extends to an $\ic$-basis of minimum weight if and only if $I \cap (K_m - K_{m-1})$ represents an independent set in $q(K_m) / q(K_{m-1})$, for all $m$.
%\end{corollary}

%\begin{corollary}  A set $B$ is an $\ic$-basis of minimum weight if and only if $B_m - B_{m-1}$ represents a basis in $q(K_m)/q(K_{m-1})$, for all $m$.
%\end{corollary}
%\begin{proof}  Any $\ic$-basis that meets the criterion of Proposition X clearly satisfies this one.  The converse is a simple case of reverse-engineering.
%\end{proof}

\begin{corollary}
\label{cor:intro3}
A set $B$ is a minimum-weight $q$-basis if and only if there exist  $I_m \su K_m$ such that
\[
B = I_1 \cup \cd \cup I_n
\]
and $I_m$ represents a basis in $q(K_m)/q(K_{m-1})$, for all $m$.
\end{corollary}
\begin{proof}
If $B$ is a minimum-weight basis then we may take $I_m = B_m - B_{m-1}$, by Proposition \ref{prop:intro2}.  The converse is simple reversal.
\end{proof}

%How are these observations useful?  One way to say this is by comparison with the two standard modes of proving the existence of Jordan bases.  The first, generally taught to undergraduates, reduces to the nilpotent case and there gives an explicit constructive proof.  The second, typically taught to graduate students, notes that Jordan decomposition is an elementary consequence of the far more general theorem that classifying finitely generated modules over a PID.  Crucially, however, neither  argument gives an independent characterization of what a Jordan basis actually \emph{is}.  This is basic contribution of the combinatorial approach.

An immediate application of these results is the ease with which one can now compute Jordan bases:  For each $m$,  collect enough vectors from $K_m$ to represent a basis in $q(K_m)/q(K_{m-1})$.  Then, take their union.

A second application regards the ``constructive'' proof of the Jordan form in linear algebra texts.  A noted source of discomfort with this approach is that its construction procedure can be shown to work, but there is no clear sense \emph{why} it works.  Here, too, discrete optimization can shed some light.
As examples we take three clean constructive arguments.
\begin{enumerate}
\item The first, by Tao \cite{taoblog07}, begins with an arbitrary basis of the base space.   In general, the orbit of this set will span the space but  fail to be linearly independent.  It is shown that each linear dependence relation reveals how to augment one of the original basis vectors so as to shorten its orbit, in essence replacing $\{v, Tv, \ld, T^n v\}$ with $\{v + u, T(v + u), \ld, T^k(v + u)\}$, for some  $k$ less than $n$.  Since the union of orbits grows strictly smaller after each iteration, the process  terminates.  The set that remains is a linearly independent union of orbits that spans the  space.  That is, a Jordan basis.

\item The second, by Wildon \cite{wildonshortproofjcf}, inducts on the dimension of the base space.  The inductive hypothesis provides a Jordan basis for the restriction of $T$ to its image, which for a niltpotent operator lies properly inside the domain.  Since this basis lies in the image, every orbit $\{v, Tv, \ld, T^n v\}$ is contained in that of some vector $u$ such that $v = Tu$.  Moreover, some vectors in the basis vanish under $T$, and we may extend these to form a basis for the kernel.  It is shown via algebraic operations that the set composed of the orbits of the $u$, \emph{plus} all the kernel elements, is linearly independent.  Dimension counting then shows the set to be a basis.

\item The third and perhaps the most direct comes from Baker \cite{BakerBlog}, who argues inductively that the linear span of every maximum-length orbit has a $T$-invariant complement.  Splitting off maximal-length orbits from successively smaller complements gives the desired decomposition.
\end{enumerate}

What do these approaches have in common?  With the benefit of hindsight, each computes a minimum-weight $q$-basis.   In fact, each implements one of two classical algorithms in combinatorial optimization.  In reference to $q$-bases, the last of these algorithms begins with an empty set, $B_0$, and so long as $B_k$ is properly contained in a $q$-independent set of cardinality $|B_k|+1$, chooses from among these a 1-element extension of  minimum weight, assigning this the label $B_{k+1}$.  We will show in later section that, because $q$-bases realize the structure of a matroid, this process returns a minimum-weight basis.  The common name for this procedure is the \emph{greedy algorithm} for matroid optimization.  The first algorithm begins with a complete $q$-basis $B$.  So long as some element $s$ in $B$ can be exchanged for some element $t$ to form a new $q$-basis of lesser weight, the algorithm does so.  Again, because $q$-bases realize the structure of a matroid -- and, more generally, of an $M$-convex set -- the output is an optimal basis.  This is called \emph{gradient descent}.

Clearly, the algorithm of Tao implements gradient descent.  That of Baker is a formal dual to the matroid greedy algorithm we have just described.  In this formulation the weight function of interest is the length of a Jordan chains, hence the focus on maximal cycles.  That of Wildon implements a quotient construction which we will discuss in later sections.  Pleasingly, it takes much less work to prove and understand these observations than it did to prove our first sample proposition.  Indeed, the only reason we have worked so hard in the fist place was to avoid the use of formal language.

%A third application, which we mention only in passing, regards a family of research problems introduced in the mid-1980's. At that time, [?], [?], and [?] posed a simple question: when does a Jordan basis for a $T$-invariant subspace extend to a Jordan basis for the entire base space? The answer, as explained in several publications, was a pair of conditions on Jordan chains relating to extendibility.  In later sections we will see how the machinery we have already developed can clarify and extend this answer beyond the case of complex operators to maps in an arbitrary abelian category.

\subsubsection{Outline and Contributions of the Thesis}

The principal arc of our story is a dramatic simplification of the example described above regarding Jordan bases.   The idea is to de-clutter the anecdote by rising in abstraction and removing excess structure.
\begin{itemize}
 \item In place of $q$-bases, we turn to simpler matroid bases.
 \item In place of linear operators, we generalize to morphisms in an Abelian category.
\end{itemize}

%The prologue to this story is a short account of the relation between matroids and abelian categories, and the post script is a collection of applications.
This lift in level of abstraction yields increased generality. This is leveraged into the following specific contributions:
%This work begins with a novel synthesis of matroid theory and abelian categories. Its end consists of applications in a number of directions, roughly organized as follows.

%  THIS IS REALLY WEAK -- YOU WILL NEED TO UPDATE THIS WITH THE SPECIFIC CONTRIBUTIONS.  NOT THE PROCESS BY HOW YOU GET THERE, BUT THE BIG OUTCOMES.
\begin{enumerate}
\item The principal outcome of the thesis is a novel algorithm for computing homology and persistent homology of a complex. This is incarnated in the \emph{Eirene} software platform, which is benchmarked against leading competitors and shown to give roughly two order-of-magnitude improvement in speed (elapsed time) and memory (max heap) demands on the largest complexes computed in the literature. This can be found in \S\ref{sec:ops}.
\item In \S\ref{sec:gradcanform} we give a novel definition of \emph{combinatorial homology} in terms of matroid bifiltrations. This generalizes homology of chain complexes and persistent homology and permits the merger of greedy algorithms with more classical homological algebra.
\item A novel relationship between the Schur complement (in linear algebra), discrete Morse Theory (in computational homology), and minimal bases (in combinatorial optimization) is developed in \S\ref{ch:efficienthomologycomputation} and \S\ref{ch:morsetheory}; this is used as a key step in building the algorithm for \emph{Eirene}.
\end{enumerate}

%{\bf Combinatorial homological algebra:} Conceptually, the content can be arranged into three tiers.  On the first we relate the axioms that define matroids with those that define abelian categories.  The basic objects on this level are matrices, kernels, independent sets, and biproducts.  On the second we relate some elementary structures built on the axioms of matroids and abelian categories, such as filtrations, weighted bases, and greedy algorithms.  Taken together, these tiers provide an rudimentary lexicon to translate algebraic objects and operations to combinatorial ones, and vice versa.   On the third tier we apply this lexicon to study some algebraic problems combinatorially, and to introduce a combinatorial object motivated by homological algebra.  The objects on this level are Laplace and differential operators.

%{\bf Nilpotent operators and efficient computational homology:} Outside this stepped system, we relate the combinatorics of nilpotent operators to some basic problems in computation with particular emphasis on topological data analysis.  The principal objects of this application are the circuits of cellular matroids.  We use these cycles as a means to transfer geometric intuition to computer algebra and algorithm design.  The outcome, as implemented in the software library \emph{Eirene}, is an orders-of-magnitude improvement in the time and memory cost of persistent homology computations, over current state-of-the-art.

These outcomes are the product of careful distilling of the notion of minimal bases from the foreshadowed Jordan form story above. The steps are as follows:
% THESE NEED TO BE FILLED IN WITH REFERENCE TO THE SPECIFIC OUTCOMES ABOVE.  I AM NOT YET READY TO SUMMARIZE THESE WELL.
\begin{itemize}
\item Chapter \ref{ch:matroids} is a self-contained introduction to the tools from matroid theory here needed, with an emphasis on \emph{exchange} as the key concept.
\item Chapter \ref{ch:modularity} reviews modularity in matroids as the precursor for generating minimal bases.
\item Chapter \ref{ch:canonicalforms} introduces a formal notion of a nilpotent operator on a matroid, and classifies its canonical forms.  As special cases, we derive combinatorial generalizations of homology and persistent homlogy.
\item Chapter \ref{ch:exchangeformulae} provides formulae relating the combinatorial operation of exchange with the algebraic operation of matrix factorization.   The main technical insight is an lifting of the \emph{M\"obius Inversion Formula} to the natural setting of homological algebra, abelian categories.
\item Chapter \ref{ch:elementaryexchange} classifies the LU and Jordan decompositions of classical matrix algebra combinatorially.   We describe combinatorial algorithms to obtain such decompositions by greedy optimization.   The key idea in this formalism is the matroid theoretic application of \emph{elementary exchange}.
\item Chapter \ref{ch:efficienthomologycomputation}  applies the algorithms of Chapter \ref{ch:elementaryexchange} to the problem of efficient homology computation.   Our main observation is a three-way translation between the topology, algebra, and combinatorics of a cellular space.
\item Chapter \ref{ch:morsetheory} posits a new foundation for spectral and algebraic Morse theory.  This approach is simpler and more general than the published results in either subject of which we are aware.   The main idea is to lift the notion of a Morse complex to that of a Jordan complement.
\end{itemize}

%\afterpage{\blankpage}

%\begin{enumerate}

\chapter{Notation}
Although a linear-algebraic sensibility is the default mode of this thesis, some conventions are category-theoretic in nature.  For instance, we employ a superscript symbol $op$ to emphasize the symmetry between certain pairs of maps or objects.

To each function $f$ we associate a domain $\D(f)$, an image $\I(f)$, and, optionally, a \emph{codomain} $\D\op(f)$.   We will write $f: A \to B$ to declare that $A = \D(f)$ and  $B = \D\op(f)$.  The \emph{identity} function on $A$ is the unique map $1_A:A \to A$ so that $1_A(a) = a$.   A map \emph{into $W$} is a map with codomain $W$, and a map \emph{out of $W$} is a map with domain $W$.  We write
\begin{align*}
f(S) = \{f(s) : s \in S\} && f \inv(T) = \{a \in A: f(a) \in T\}
\end{align*}
for any $S  \su A$ and $T \su B$.

The terms \emph{collection} and \emph{family} will be used interchangeably with the term \emph{set}.    An  \emph{$A$-indexed family in $B$} is a set function $f:A \to B$.       We will sometimes write $f_a$ for $ f(a)$, and either $(f_a)_{a \in A}$ or $(f_a)$ for the function $f$.      A \emph{sequence in $B$} is an $I$-indexed family in $B$, where $I=\{n \in \Z : a \le n \le b\}$ for some  $a, b \in \Z \cup \{-\infty, +\infty\}$.    Given an indexed family $f$ and a collection $\ic$ of unindexed sets, we write $f \in \ic$ when $f$ is injective and $\I(f) \in \ic$.

Several  mathematical operations accept indexed families as inputs, e.g.\ sum, product,  union, and intersection.   By convention
\begin{align*}
\sum_{a \in \emptyset} f_a = 0 && \prod_{a \in \emptyset} f_a = 1 && \bigcup_{a \in \emptyset} f_a = \emptyset && \bigcap_{a \in \emptyset } f_a = E,
\end{align*}
where $E = \cup_{a \in A} f_a$.  Arguments will be dropped from the expression $\square_{s \in S} f_s$ where context leaves no room for confusion, e.g.\ $\sum f = \sum_{a \in A} f_a$.   

We will frequently write $A - B$ for the set difference $\{a \in A : a \notin b\}$.  The left-right rule is observed for order of set operations, so that
$$
A - B \cup C = (A - B) \cup C,
$$
though such expressions will generally be reserved  for cases where $B$ and $C$ have trivial intersection, hence where $A - B \cup C = A \cup C - B$.

An \emph{unindexed family in $B$} is a subset  $S \su B$.  To every unindexed family corresponds a canonical indexed family, $1_S$.  By abuse of notation, we will use $S$ and $1_S$ interchangeably as inputs to the standard operations, for instance writing $\sum S$ for $\sum_{s \in S} s$.
%Where context leaves no room for confusion, we will write $\square_{S} f(a)$ for $\square_{a \in S}f(a)$.  In the special case where $S = A$, we write $\square f$, e.g.\ $\sum f$ for $\sum_{a \in A} f(a)$.

 Given  any relation $\sim$ on $B$, we write $f_{\sim b}$ for the set $\{a \in A : f(a) \sim b\}$.  If $f$ is real-valued, for example,  then $f_{\le t} = \{a \in A: f(a) \le t\}$.  Similarly, given any $C \su A$, we will write $C_{f \sim b}$ for $\{c \in C: f(c) \sim b\}$.  By extension we set 
 \begin{align*}
f_{\sim a} = f_{\sim f(a)} && and  && C_{\sim a} = C_{\sim f(a)}.
 \end{align*}

The \emph{characteristic function} of a subset $J \su I$ is the zero-one set function $\chi_J$ with domain $I$ defined by
  \begin{align*}
  \chi_J(i) =
  \begin{cases}
  1 & i \in J  \\
  0 & else.
  \end{cases}
  \end{align*}

Symbols $\R$ and $\C$ denote the fields of real and complex numbers, respectively.  Given a coefficient field $\field$, we write $\field^I$ for the set of all functions $I \to \field$, regarded as $\field$-linear vector space under the usual addition and scalar multiplication.   We write $\K(T)$ for the null space of a linear map $T$, and $\Hom(U,V)$ for the space of  linear maps $U \to V$.   If $V$ is a vector space and $I$ is a set, then the \emph{support} of  $f:I \to V$ is
  $$
  \Supp(f) = \{i \in I : f(i) \neq 0\}.
  $$

The \emph{product} of  an indexed family of $\field$-linear spaces $(V_i)_{i \in I}$ is the space of all maps $f:I \to \cup V$ such that $f_i \in V_i$ for all $i$, equipped with the usual addition and scalar multiplication 
\begin{align*}
(f + g)_i = f_i + g_i && (\ak \cdot f)_i = \ak \cdot  f_i.
\end{align*}
The \emph{coproduct} of $(V_i)_{i \in I}$ is the subspace $\{f \in \times_i V_i : |\Supp(f)| < \infty\}.$  The product and coproduct constructions will coincide in most cases of interest.

\begin{example}  By  convention $\R^n = \R^{\{1, \ld, n\}}$.   If $V_1 =  \cd = V_n = \R$, then $\times V = \R^n$.   Since every element of $\R^n$ has finite support,  $\oplus V = \times V$.
\end{example}

Given any family of maps $f: I \to \Hom(W, V_i)$, we write $\times f$ for the linear transformation $W \to \times V$ that assigns to each $w \in W$ the function $(\times f)(w): I \to \times V$ such that 
$$(\times f)(w)_i = f_i(w).
$$   Dually, given any family $f: I \to \Hom(V_i, W)$, we write $\oplus f$ for the map $\oplus V \to W$ that assigns to each $v \in \oplus V$ the sum $\sum_i f_i(v)$.  

\begin{example}  Suppose $W = V_i  = \R$ for $i \in I = \{1,2\}$, so that $\times V  = \oplus V = \R^2$.  If $f_1$ and $f_2$ are  linear maps  $\R \to \R$, then 
\begin{align*}
\times f: \R \to \R^2 && (\times f)(x) = ( f_1(x), f_2(x))
\end{align*}
and
\begin{align*}
\oplus f : \R^2 \to \R && (\oplus f)(x,y) = f_1(x) + f_2(y).
\end{align*}
\end{example}

%\afterpage{\blankpage}

\part{Canonical Forms}

\chapter{Background: Matroids}
\label{ch:matroids}

This chapter gives a comprehensive introduction to the elements of matroid theory used in this text.   It is suggested that the reader skim  \S\ref{sec:indrankclose}-\ref{sec:modularity} on a first pass, returning as needed for the small number of applications that require greater detail.  Thorough treatments of any of the subjects introduced in this summary may be found in any one of several excellent texts, e.g. \cite{orientedmatroids93,truemperMatroidDecomposition}.

\section{Independence, rank, and closure}  \label{sec:indrankclose}

An \emph{independence system} is a mathematical object determined by the data of a set $E$, called the \emph{ground set}, and a family of subsets of $E$  that is closed under inclusion.   The subsets of $E$ that belong to this family are called \emph{independent}, and those that do not are called \emph{dependent}.  One requires the empty set to be independent.
% is a pair $(E, \ic)$, where $E$ is a set and $\ic$ is a family of subsets of $E$ closed under inclusion.  One requires that this family be nonempty, so that, in particular, it contains the empty set.  The sets that belong to this family are called \emph{independent}, and those that do not are called \emph{dependent}.  The set $E$ itself is called the \emph{ground set} of the system.
%A family of sets is \emph{closed under inclusion} if it contains the subsets of each of its elements.   family closed under inclusion is an \emph{independence system} if it contains the empty set and  its elements are subsets of some fixed set $E$, called a \emph{ground set}.  By extension, elements of an independence system are called \emph{independent sets}.
An independence system satisfies the \emph{Steinitz Exchange Axiom} if for every pair of independent sets $I$ and $J$ such that $|I|<|J|$, there is at least one $j \in J$ such that  $I \cup \{j\}$ is independent.  The systems that satisfy this property have a special name.

\begin{definition}
\label{def:matroid}
A \emph{matroid} is an independence system that satisfies the Steinitz Exchange Axiom.
\end{definition}

In keeping with standard convention, we  write $|\mc|$ for the ground set of a matroid $\mc = (E, \ic)$, and $\ic(\mc)$ for the associated family of independent sets.
%We will generally denote matroids by calligraphic script, e.g. $\mc = (E, \ic)$.  We write $| \mc|$ for the ground set of $\mc$, and $\ic(\mc)$ for the associated family of independent sets.

The pair $(E, \ic)$, where $E$ is a subset of a vector space and $\ic$ is the family of all linearly independent subsets of  $E$, is an archetypal  matroid.  The fact that $\ic$ satisfies the Steinitz Exchange Axiom is the content of the eponymous Steinitz exchange lemma.  A closely related example is the family of all sets that \emph{map} to independent sets under an index function $r: E \to V$.   Matroids realized in this fashion are called \emph{linear} or \emph{representable}, and the associated index functions are called \emph{linear representations}.   Representations that index the columns of a matrix give rise to the term ``matroid.''

\begin{remark}
Every representation function $r: E \to \field^I$ determines an array $[r] \in I \times E$ such that $[r](i,e) = r(e)_i$.  We call this the \emph{matrix representation} of $r$.
\end{remark}

A proper subset of the linear matroids is the family of {graphical matroids}.    Every (finite) undirected graph $G = (V, E)$, where $V$ is the set of vertices and $E$ the set of edges, determines an independence system $(E, \ic)$, where $\ic$ is the family of forests.  That this system satisfies the Exchange Axiom follows from the fact that a set of edges forms a forest if and only if the corresponding columns of the node-incidence matrix are linearly independent over the two element field.  A matroid is \emph{graphical} if it is isomorphic to one of this form.

Graphical matroids are a field of study in their own right, but also provide helpful intuition at all levels of matroid theory.  At the introductory level, especially, recasting the statement of a general theorem in terms of a small graphical example, e.g.\ the complete graph on four vertices, has a useful tendency to demystify abstract results.  Several terms used generally throughout matroid theory have graph-theoretic etymologies: an element $e$ of the ground set is a \emph{loop} if the singleton $\{e\}$ is dependent.  Two non-loops are \emph{parallel} if the two-element set that contains them both is dependent.  A matroid is \emph{simple} if it has no loops and no parallel pairs, i.e., no dependent sets of cardinality strictly less than three.

Some familiar constructions on vector spaces have natural counterparts in matroid theory.  To begin with dimension, let us say that an independent set is \emph{maximal} if it is properly contained in no other independent set.  We call such sets  \emph{bases}, and denote the family of all bases $\bc(\mc)$.   It is clear from the Exchange Axiom that every basis has equal cardinality, and this common integer is called the \emph{rank} of the matroid.  Since the pair $\mc|_S = (S, \{I \in \ic : I \su S\})$ is itself a matroid for every  $S \su E$, we may extend the notion of rank to that of a \emph{rank function} $\rk$ on the subsets of $E$.

In linear algebra, the  span of a subset $S \su V$, denoted $\spann(S)$, may be characterized as the set of  vectors $v$ such that any maximal independent subset of $S$ is also maximal in $S \cup \{v\}$.  The analog  for matroids is the collection of all $t$ such that $\rk(S) = \rk(S \cup \{t\})$.  We call this  the \emph{closure} of $S$, denoted $\cl(S)$.   A subspace of $V$ is a subset that equals its  linear span, and, by analogy, a \emph{flat} of a matroid is defined to be a subset that equals its  closure.  Flats are also called \emph{closed sets}.

\section{Circuits}
\label{sec:circuits}
A set $\zk$ is called \emph{minimally dependent} if it contains no proper dependent subset, or, equivalently, if every subset of order $|\zk|-1$ is independent.  Minimally dependent sets are called \emph{circuits}.
%Every finite dependent set contains a circuit, since any that did not would give rise to a sequence of dependent subsets, each properly contained in the one that precedes it -- no such sequence can exist, since the empty set is independent, by definition.
While the combinatorial definition of a circuit may seem a bit alien at first glance, circuits themselves are quite familiar to students of graph theory and linear algebra.  It is quickly shown, for example, that a set $C$ in a graphical matroid forms a circuit if and only if the corresponding edges  form a simple cycle in the underlying graph, since any one edge may be dropped from such a set to form a tree.  For a fixed basis $B$ of a finite dimensional vector space $V$, moreover,  every vector may be  expressed as a unique linear combination in  $B$.  The set
$$
\zeta_B(v) = \{v\} \cup \supp(\ak)
$$
where $\ak$ is the indexed family of scalars such that $v = \sum_{b \in B} \ak_b$, is called the \emph{fundamental circuit} of $v$ with respect to $B$.  It is simple to check that this is a bona fide circuit in the associated matroid, since its rank and cardinality differ by one, and each $u \in \zk_B(v)$ lies in the linear span of $\zk_B(v) - \{u\}$.
% To verify that this is a bona fide circuit, let us note a useful alternative definition:  a set $\zk$ is a circuit if and only if it is dependent and every subset of order $|\zk| -1$ is independent.
%The second is a reformulation of the first:  a set is a  circuit if and only if it has rank  $|\zk|-1$, and every subset of this order spans the  set.
%This condition is clearly met by $\zk_B(v)$.

%This latter observation ties the notion of fundamental circuits for vector spaces to that of matroids more generally.
%In particular, it implies that $B - \{u\} \cup \{v\}$ whenever $u \in B$ lies in the fundamental circuit of $v$.
More generally, given any element $i$ of a matroid with basis $B$, one can define the \emph{fundamental circuit of $i$ with respect to $B$}, denoted  $\zk_B(i)$, to be the set containing $v$ and those  $j \in B$  for which $B  - \{j\} \cup \{i\}$ is a basis.  It is simple to check that this, also, is a bona fide circuit, since every subset of cardinality $|\zk_B(i)|-1$ extends to a basis.
%\begin{exercise}  The fundamental circuit of a vector $v$ with respect to a linear basis $B$ is exactly that of the $v$ with respect to the same basis in the corresponding matroid.
%\end{exercise}

\begin{exercise}
\label{ex:forest}
Suppose that $E$ is the family of edges in an undirected graph, and $\ic$  the family of forests.  If $B$ is a spanning forest and $e \in E - B$, then the matroid-theoretic fundamental circuit of $e$ in $B$ agrees with the graph-theoretic fundamental circuit.
\end{exercise}

It will be useful to have a separate notation for the portion of $\zk_B(i)$ that lies in $B$.  We call this  the \emph{support of $i$ with respect to $B$}, denoted  $\supp_B(i)=\zk_B(i) - \{i\}$.  Let us further set $\supp_B(b) = b$ for any $b$ in $B$ and $\supp_B(S) = \cup_{s \in S} \supp_B(s)$.   This notation is not standard, so far as we are aware, but it is highly convenient.

\begin{example}
\label{ex:matrixN}
Let $\field$ be the two-element field and $N$ be the $3 \times 6$ matrix over $\field$ shown below.  Here, for visual clarity, zero elements appear only as blank entries.  Let $\mc$ be the matroid on ground set $\{a_1, a_2, a_3, b_1, b_2, b_3\}$ for which a subset $I$ is independent iff it indexes a linearly independent set of column vectors.  The fundamental circuit of $a_1$ with respect to basis $B  = \{b_1, b_2, b_3\}$ is  $\{a_1, b_2, b_3\}$.  Its combinatorial support  is $\{b_2, b_3\}$.  The fundamental circuit of $a_3$ is $\{a_3, b_1, b_3\}$, and its combinatorial support is $\{b_1, b_3\}$.
\vspace{0.2cm}

\[
\begin{array}{|ccc|ccc|}
\multicolumn{1}{c}{a_1} & a_2 & \multicolumn{1}{c}{a_3} & b_1 & b_2 & \multicolumn{1}{c}{b_3}\\ \cline{1-6}
 &1&1&1&& \\
 1& 1&&&1& \\
 1&1 & 1&&&1 \\ \cline{1-6}
\end{array}
\]
\vspace{0.2cm}

\end{example}

The reader will note that in the matrix $N$ of Example \ref{ex:matrixN}, the matroid-theoretic support of each column coincides with its linear support as a vector in $\field^n$, if one identifies the  unit vectors $b_1, b_2, b_3$ with the rows that support them.

In general, any  representation that carries $B \su E$ bijectively onto a basis of unit vectors in $\field^B$ is  called \emph{$B$-standard}.  Standard representations play a foundational role in the study of linear matroids, due in part to the correspondence between linear and combinatorial supports observed above.   This correspondence, which  we have observed for one standard representation, evidently holds for all.

One elementary operation in matroid theory is the rotation of elements into and out of a basis.   Returning to $N$, let us consider the operation  swapping $a_2$ for $b_2$ in $B$, producing a new basis $C = B - \{b_2\} \cup \{a_2\}$.  We can generate a $C$-standard representation of $M$ by multiplying $N$ on the left with an invertible $3 \times 3$ matrix $Q$, while leaving the column labels unchanged.  This multiplication should map $a_2$ to the second standard unit vector and  leave the remaining unit vectors unchanged.  These conditions determine that $Q$ should have form

\[
\arraycolsep=7pt\def\arraystretch{1}
\begin{array}{|ccc|}
\cline{1-3}
1 & 1 & \\
& 1& \\
& 1 & 1 \\ \cline{1-3}
\end{array}
\]
\vspace{0.2cm}

over the two element field.   The resulting representation, $M$, has form

\[
\begin{array}{|ccc|ccc|}
\multicolumn{1}{c}{a_1} & a_2 & \multicolumn{1}{c}{a_3} & b_1 & b_2 & \multicolumn{1}{c}{b_3}\\ \cline{1-6}
 1&&1&1&1& \\
 1& 1&&&1& \\
 & & 1&&1&1 \\ \cline{1-6}
\end{array}
\]
\vspace{0.2cm}

Note that the fundamental circuit of $b_2$ with respect to $C$ agrees with the fundamental circuit of $a_2$ with respect to $B$.

We define the \emph{support matrix} of a basis $B$, denoted  $\tm{Supp}_B \in \{0,1\}^{B \times A}$ by the condition $\tm{Supp}_B(i,j) = 1$ iff $i \in \supp_B(j)$.   The literature refers to this with the lengthier name \emph{$B$-fundamental circuit incidence matrix}.
% Brylawski, T. & Lucas, T. D. (1976). Uniquely representable combinatorial geometries, in Colloquio Internazionale Sulle Teorie Combinatorie (Roma, 1973), Atti dei Convegni Lincei 17, Tomo I, pp. 83-104. Accad. Naz. Lincei, Rome.
As we have seen, if  $[r] \in \field^{B \times |\mc|}$ is the matrix realization of a $B$-standard representation $r: |\mc| \to \field^B$,  then $\Supp_B(|\mc|)$ is the matrix obtained by switching all nonzero entries of $[r]$ to 1's.

%This is no coincidence.  In general, if a basis element $b$ is swapped out for $a$ in some representation $r$, where $r(b)$ is the $k$th unit vector and $r(a) = (\ak_1, \ld, \ak_k = 1,\ld \ak_m)$, the transformation matrix $Q$ will have form % $[u_1\; | \; \cd \; | \; u_m]$ where $u_k = (-\ak_1, \ld, -\ak_{k-1}, 1, \ld, -\ak_m)$ and $u_i$ is the the $i$th standard basis vector, for $i \neq k$.
%$$
%\begin{array}{|cccccc|}
%\cline{1-6}
%\hspace{.1cm}1\hspace{.4cm} & & & -\ak_1 &&  \\
%% &1  && -\ak_2 &  \\
%&\ddots&& \vdots && \\
%&& &1 && \\
%&&& \vdots & \hspace{.2cm}\ddots& \\
%&&& -\ak_m & & \hspace{.4cm}1 \hspace{.1cm}\\   \cline{1-6}
%\end{array}
%$$
%where all entries off the diagonal and outside the the $k$th column vanish (if $\ak_k \neq 1$ we must scale column $k$ by $1/ \ak_k$).  By inspection, left-multiplication with $Q$ will send the $k$th standard basis vector to one with identical \emph{linear} support to that of $r(a)$.   We may interpret this set combinatorialy almost exactly as we did  $\supp(r(a))$, except that the $k$th coordinate is now  identified with $a$, rather than $b$.  Thus $\supp_C(b) = \supp_B(a)- \{b\} \cup \{a\}$. It follows that $\zk_C(b) = \zk_B(a)$, as we already saw for $b = b_2$ and $a = a_2$.  Lemma X is the general version of this observation.
%
%
%

\section{Minors}

%One of the most important objects associated with matroid a matroid is its family of minors.
A \emph{minor} of a matroid $\mc$ is a matroid obtained by a sequence of two elementary operations, called \emph{deletion} and \emph{contraction}.   A \emph{deletion minor} of $\mc$ is a pair $ (S, \ic|_S)$, where $S$ is a subset of $E$ and $\ic_S = \{I \in \ic : I \su S\}$.  We call this the \emph{restriction} of $\mc$ to $S$ or the minor obtained by \emph{deleting} $E - S$.

The notion of a contraction minor is motivated by the following example.  Suppose are given a vector space $V$, a subspace $W$, and a surjection $r: V \to U$ with kernel $W$.  We may regard $r$ as the index function of a linear matroid on ground set $V$,  declaring $S \su V$ to be independent if and only if $r(S)$ is independent in $U$.  If we follow this action by deleting $S$, the resulting matroid $\nc$  \emph{contraction minor} of $V$.  The general operation of contraction is an elementary generalization of this construction.

Of note, a subset $I$ is independent in $\nc$ if and only if one (respectively, all) of the following three conditions are met: (i) $I \cup J$ is independent, for some basis $J$ of $W$, (ii) $I \cup J$ is independent for \emph{every} basis of $W$, and (iii)   $I \cup J$ is independent in $V$ for every  independent  $J \su W$.

Let us see if we can equate these conditions combinatorially.  Equivalence of (ii) and (iii) is clear from the definition of an independence system.  For (i) and (ii), posit a basis $S$ of $W$ such that  $I \cup S$ is linearly independent in $V$.  Any basis $T$ of $W$ may then be extended to a maximal independent subset of $I \cup S \cup T$, by adding some elements of $I \cup S$.  None of the elements in this extension can come from $S$, since otherwise $T$ would fail to be maximal in $W$.  Thus the extended basis includes into $T \cup I$.  By dimension counting, it is exactly $T \cup I$.  Therefore the union of $I$ with \emph{any} basis of $W$ is independent in $V$, and all three conditions are equivalent.

This argument does not depend on the fact that $W$ is a subspace of $V$, or on any algebraic structure whatsoever.  For any $W \su E$, then, and any  matroid $(E, \ic)$, we may unambiguously define $\ic/W$ to be the family of all $I$ that satisfy one or all of (i), (ii), and (iii).  This family is  closed under inclusion, and since only elements from $I$ could possibly be transferred from a set of form $S \cup I $ to  a set of form $S \cup J $ while preserving independence, satisfaction of the Steinitz exchange axiom by $\ic/W$ follows from $\ic$.  Thus the pair $(E - W, \mc/W)$ is a bona fide matroid, the \emph{contraction minor} of $\mc$ by $E-W$.

\begin{remark}
\label{rem:contractionminor}
We will  make frequent use of a related matroid,   $\mc \sslash W = (E, I /W)$, in later sections.
\end{remark}

It will be useful to record the following two identities for future reference, while the material is fresh.  For any $U$ and $W$, one has
\begin{align}
(\ic /U) /W = \ic / (U \cup W)  && \mathrm{and} && \ic/U = \ic/ \cl(U) .
\end{align}
The first assertion follows easily from the definition of contraction.  One can verify the second is by noting that $I$ is independent in $\ic/U$ if and only if $\rk(U\cup I) = \rk(U) + |I|$.  This suffices, since $\rk(U \cup I) = \rk(\cl(U) \cup I)$.

A matroid obtained by some sequence of deletion and contraction operations is called a \emph{minor}.  It can be shown that distinct sequences will produce the same minor, provided that the each deletes (respectively, contracts) the same set of elements in total.  Thus any minor may be expressed in form $(\mc/T)|_{S}$.  Where context leaves no room for confusion, we will abbreviate this expression to $S/T$.
%The reader may wish to note that this expression is well-posed, in particular, when $T$ is not a subset of $S$.
More generally, we write $S/T$ to connote the minor $(\mc/(T \cap E))|_{(S \cap T \cap E )}$ for sets $S$ and $T$ not contained in $E$.  This will be convenient, for example, when  ``intersecting'' two minors of the same matroid.

%\begin{remark}  It follows immediately from the preceding definitions that a set $I$ is independent in $S/T$ if and only if $I \cup J$ is independent in $\mc$, for every independent $J$ contained in $T$.
%\end{remark}

\section{Modularity}
\label{sec:modularity}

Given any pair of subsets $S, T \su E$, one can in general extend a basis $B$ for $S \cap T$ to either a basis $B \cup I$ of $S$ or a basis $B \cup J$ of $T$.  Since $B \cup I \cup J$ then generates $S \cup T$, it follows that
\begin{align}
\label{eq:submodular}
\rk(S \cup T) \le \rk(S) + \rk(T) - \rk(S \cap T).
\end{align}
A real-valued function on the power set of $E$ that satisfies \eqref{eq:submodular} for any $S$ and $T$ is called \emph{submodular}.

The language of submodularity is motivated by the theory of lattices.  If $V$ is a finite-dimensional vector space and ${\mathcal L}$ the associated lattice of linear subspaces (read:\ matroid flats), then the restriction of $\rk$ to ${\mathcal L}$ coincides exactly with the height function of this lattice.   By definition, a lattice with height function function $h$ is \emph{semimodular} if $r(s \vee t) \le r(s) + r(t) - r(s \wedge t)$ for all $s$ and $t$, and \emph{modular} if strict equality holds.  By way of extension, we  say the pair $(S,T)$ {obeys the modular law (with respect to $\rk$)}  if strict equality holds in \eqref{eq:submodular}.  More generally, given any families of sets ${\mathcal S}$ and ${\mathcal T}$, we  say that $(\mathcal S, \mathcal T)$ is modular if every $(S, T) \in \sc \times \tc$ is modular.

It will be useful to have one alternative interpretation of \eqref{eq:submodular}.  Since $\rk(S  / T) = \rk(S \cup T)  - \rk(S)$ and $\rk(T/ S \cap T) = \rk(T) - \rk(S \cap T)$, this inequality may be rewritten
\begin{align}
\rk(S/T) \le \rk(S / S \cap T).
\end{align}
Strict equality holds if and only if $(S, T)$ is modular.  In fact, with modularity more can be said.

%Returning to the example where $S$ and $T$ are linear subspaces, the author finds it useful to draw a loose connection between Y and the First Isomorphism Theorem of abstract algebra.  Much interesting behavior on the part of matroids can be traced to a failure of strict equality in Y, and much of the good behavior of certain algebraic structures can be explained in terms of the modularity of the underlying matroid, as we will see in later sections.

\begin{lemma}
\label{lem:modularisomorphism}
If $(S, T)$ obeys the modular law, then $S/T = S/(S \cap T)$.
\end{lemma}
\begin{proof}
Let $B$ be a basis for $S \cap T$, and let $I$,  $J$ be bases for  $S/(S \cap T)$ and $T/(S \cap T)$, respectively.  Then $B \cup I$ is a basis for $S$ and $B \cup J$ is a basis for $T$.  The union $B \cup I \cup J$ contains a basis for $S \cup T$, and by modularity $|B \cup I \cup J| = \rk(S \cup T)$.  Thus $B \cup I \cup J$ is a \emph{basis} for $S \cup T$, so $I \in \ic(S/T)$.   It follows that the independent sets of $S/(S \cap T)$ are also independent in $S/T$.  Since the converse holds as well, the desired conclusion follows.
\end{proof}

\section{Basis Exchange}
\label{sec:basisexchange}

A \emph{based} matroid is a pair $(\mc, B)$, where $\mc = (E, \ic)$ is a matroid and $B$ is a basis.  A fundamental operation on based matroids is  \emph{basis} exchange:  the substitution of $B$ in $(\mc, B)$ with a basis of form $C = B - I \cup J$.

A \emph{based representation} (with coefficients in field $\field$) of a based matroid is a linear representation   $r: E \to \field^B$ such that $r(b) = \chi_b$ for all $b \in B$.   We call such  representations  \emph{$B$-standard}.    Basis exchange induces a canonical transformation of based representations,  sending each $r: E \to \field^B$ to the composition $T \com r$, where $T$ is the linear map $\field^B \to \field^C$ such that $(T \com r) (c) = \chi_c$ for each element $c \in C$.    %To return to our example, the transformation induced by substituting $J = \{a_1, a_2\}$ for $I = \{b_1, b_2\}$ is realized by left-multiplication with the matrix $N$ shown in Y.

In general, if we regard $r$ as an element of $\field^{B \times E}$ with block decomposition
\vspace{0.2cm}

\[
\begin{array}{cccc}
&\hspace{.2cm}J\hspace{.2cm} & E - J \\ \cline{2-3}
\multicolumn{1}{c|}{I}&\ak & \multicolumn{1}{c|}{\bk}  \\
\multicolumn{1}{c|}{B-I}&\ck & \multicolumn{1}{c|}{\dk}\\ \cline{2-3}
\end{array}
\]
\vspace{0.2cm}

then $T$ may be understood as an element of $\field^{C \times B}$ with block form
\vspace{0.2cm}

\[
\begin{array}{cccc}
&\hspace{.2cm}I\hspace{.2cm} & B - I \\ \cline{2-3}
\multicolumn{1}{c|}{J}&\ak\inv & \multicolumn{1}{c|}{0}  \\
\multicolumn{1}{c|}{B-I}&-\ck\ak \inv & \multicolumn{1}{c|}{1}\\ \cline{2-3}
\end{array}
\]
\vspace{0.2cm}

where, $1$ denotes an identity submatrix of appropriate size.  Representation $T \com r$ will then have form

\[
\begin{array}{cccc}
&\hspace{.2cm}J\hspace{.2cm} & E - J \\ \cline{2-3}
\multicolumn{1}{c|}{J}&1 & \multicolumn{1}{c|}{\ak\inv \bk}  \\
\multicolumn{1}{c|}{B-I}&0 & \multicolumn{1}{c|}{\sk}\\ \cline{2-3}
\end{array}
\]
\vspace{0.2cm}

where, by definition, $\sk = \dk - \ck \ak\inv \bk$ is the \emph{Schur complement} of $\ak$ in $r$.   In the special case where $I = \{i\}$ and $J = \{j\}$ are singletons, multiplication on the left with $T$ is precisely the clearing operation that transforms column $j$ to the unit vector supported on $i$ (and subsequently relabels this row as $j$).   By analogy we may regard composition with $T$ as the ``block clearing'' operation that transforms all the columns of $J$ to unit vectors (and relabels the corresponding rows).  This perspective will be  developed in later sections.

%This makes the analogy between $T$ and a block-clearing operation a bit more precise: if $(i_1, \ld, i_k)$ and $(j_1, \ld, j_k)$ are linear orderings of the elements of $I$ and $J$ such that $B_p  = B - \{i_1, \ld, i_p\} \cup \{j_1, \ld, j_p\}$ is a basis for all $p$, then sequential exchange operations of type $B_p \mapsto B_{p+1}$, for $p = 1,\ld, {k-1}$, will produce the same representation as $B \mapsto C$.  The orders placed on $I$ and $J$ in this example matter only to the extent that each $B_p$ must be a basis.

\subsubsection{Base, Rank, and Closure}

Based matroids owe their importance in part to the special relationship between the elements of a distinguished basis and those of the ground set generally.  It is very easy, for example, to describe the closure of any $I \su B$: it is the collection of all $e \in E$ for which $I \cup \{e\}$ is dependent.  If a $B$-standard representation is available, then we may alternately characterize this set  as the collection of all $e$ for which $\supp(r(e)) \su I$.  Simultaneously, if we wish to find a basis for $I$ we may take $I$ itself, and to calculate its rank we may take $|I|$.

The situation is more delicate with sets not contained in $B$.  This is one place where elementary exchange offers some help.  If, for example, we wish to know the rank of an arbitrary subset $S$, we may rotate as many elements as possible from $S$ into the standard basis $B$.  This may be done element-by-element via singleton exchange operations, or in sets of elements simultaneously via the ``block-clearing'' operation.  One is forced to stop moving elements of $S$ into $B$ if and only if one of following two equivalent stopping conditions are met (i)  the current basis, $C$, contains a basis for $S$, (ii) the closure of $S \cap C$ contains $S$.  Once either condition holds, $\rk(S) = |S \cap C|$.  This procedure also admits the calculation of $\cl(S) = \cl(S \cap C)$, via the support heuristic used for $\supp(I)$.

\subsubsection{Contraction}

Elementary exchange offers a means to calculate contraction minors, as well.  Recall that  the independent sets of $\mc \sslash S$ are those  $I$ for which $I \cup J$ is independent, for some (equivalently, every) basis $J$ of $S$.  Let us suppose that $I$ is given, together with a basis $B$ containing $I$ and a $B$-standard representation $r$.  Since $r$ carries $I$ to a collection of unit vectors, it is easy to see that $q \com r$ represents $\mc \sslash I$, where $q: \field^B \to \field^{B-I}$ is the canonical deletion operator.

Given $r$, can in general leverage the fact that
\[
\mc \sslash S = \mc \sslash \cl(I) = \mc \sslash I
\]
for   $I \in \bc(S)$ to compute a representation of $\mc \sslash S$, for any $S$.  To do so, simply rotate an  $I \in \bc(S)$ into the standard basis by elementary exchange, and apply the deletion operator $q$.   To obtain a representation for the formal contraction $\mc/ C$, restrict the resulting representation to $E - S$.

\subsubsection{Deletion}

The role of deletion in standard representations is in many ways dual to that of contraction.  The simplest case holds when one wishes to delete a set $I$ that is \emph{disjoint from} the standard basis $B$.  In this case one may simply restrict the representation $r$ to $E - I$.   To delete an arbitrary $S$ requires slightly more care, as deletion of elements from the standard basis may result in a non-standard representation.

To address this issue, first rotate as many elements of $S$ as possible \emph{out} of the standard basis.  There is only one condition under which some elements of $S$ will remain in the resulting basis, $C$, namely that the corresponding representation $r$ has the following block matrix form
\vspace{0.2cm}

\[
\begin{array}{c|ccc|}
\multicolumn{1}{c}{}& S \cap C  & S - C & \multicolumn{1}{c}{E - C - S } \\ \cline{2-4}
S \cap C 	& 1 & * & 0 \\
C - S 	& 0 & * & * \\ \cline{2-4}
\end{array}
\]
\vspace{0.2cm}

where asterisks denote blocks of indeterminate form.  Consequently, if $q: \field^C \to \field^{C-S}$ is the standard deletion operator, then $(q \com r)|_{E - S}$ will represent  $\mc - S$ with respect to the standard basis $C - S$.

\section{Duality}

The \emph{dual} of a matroid $\mc = (E, \ic)$, denoted $\mc^* = (E, \ic^*)$, is the matroid on ground set $E$ for which a subset $S$ is  independent if and only if the complement of $S$ in $E$ contains a basis.   Bases of the dual matroid are called \emph{cobases} of $\mc$.

Duality is integral to the study of matroids general, and to our story in particular.  We will introduce terminology for this structure as needed throughout the text.

%\afterpage{\blankpage}

\chapter{Modularity}
\label{ch:modularity}

\section{Generators}
\label{sec:generators}

Recall from Section \ref{sec:modularity} that the rank function $\rk$ of a matroid $\mc = (E,\ic)$ is \emph{submodular}, meaning
\begin{align} \label{eq:submodular2}
\rk(S \cup T) + \rk(S \cap T) \le \rk(S) + \rk(T)
\end{align}
for any  $S, T \su E$.  We say that a family $\sc$ \emph{obeys the modular law} if strict equality holds in \eqref{eq:submodular2} for all $S, T \in \sc$.  

\begin{remark}  We will often refer to a family that obeys the modular law as a \emph{modular}, for shorthand.  For $|\sc| = 2$ this convention skirts on an abuse of language, since the field of lattice theory assigns a very different meaning to ``modular pair,'' which concerns ordered pairs in a poset.
\end{remark}

\begin{remark} \label{rmrk:modulartransversetame}  Every pair of linear subspaces $S, T \su W$ obeys the modular law in the  matroid associated to $W$.  This pair is defined to be \emph{transverse} if 
$$
\rk(W) = \rk(S) + \rk(T) - \rk(S \cap T).
$$   
The tameness conditions imposed by modularity  on set families is analogous and indeed closely linked to the tameness conditions imposed by transversality in bundle theory and geometric topology\cite{GPDifferential10}.
\end{remark}

A set $I$ \emph{spans} a set $S$ if $S \su \cl(I)$, and \emph{generates} $S$ if $I\cap S$ spans $S$.  If in addition $I$ is independent, we say $I$ \emph{freely generates} $S$.   By extension, $I$ spans (respectively, generates,  freely generates) a family  $\sc$ if $I$ spans (respectively, generates,  freely generates) each $S \in \sc$.     Free generation has a basic relation to modularity.

\begin{lemma}
\label{lem:freeclosed}  
The family of sets freely generated by an independent set $I$ is closed under union and finite intersection.
\end{lemma}
\begin{proof}
Let $I_U$ denote $I \cap U$ for any $U$.  If  $\uc$ is any family of sets generated by $I$ then  $I_{\cup_{U \in \uc} U}$  freely generates $\cup_{U \in \uc} U$.  Thus closure under union.   For finite intersection, suppose that $\uc = \{ S, T \}$.  Then $I_S$ and $I_T$ and $I_{S \cup T}$ freely generate $S$, $T$, and $S \cup T$, whence
\[
\rk(S) + \rk(T) - \rk(S \cup T) = |I_S| + |I_T| - |I_{S \cup T}|  =   |I_{S \cap T}| \le \rk(S \cap T).
\]
In fact strict equality must hold throughout, since submodularity of the rank function demands $\rk(S) + \rk(T) - \rk(S \cup T) \ge \rk(S \cap T)$.  Therefore $ |I_{S \cap T}|  =  \rk(S \cap T)$.  The desired conclusion follows.
\end{proof}

\begin{proposition}  
\label{prop:freemodular}
All freely generated set families are modular.
\end{proposition}
\begin{proof}  Let $\sc$ be any family of sets.  If $\sc$ is freely generated by an independent set $I$, then 
\[
\rk(S) + \rk(T)  = |I_S| + |I_T| =  |I_{S \cup T}|  +  |I_{S \cap T}| =   \rk(S \cup T) + \rk(S \cap T)
\]
for any $S, T \in \sc$ by Lemma \ref{lem:freeclosed}.  The desired conclusion follows.
\end{proof}

%\begin{proposition}
%\label{prop:freemodular}
%The freely generated set families are modular.
%\end{proposition}
%\begin{proof}
%If $S$ and $T$ are any two sets generated by an independent set $I$, and if  $I_U$ denotes $I \cap U$    for arbitrary $U$, then  $|I_S| + |I_T| - |I_{S \cup T}| =  |I_{S \cap T}|$, since $I_{S \cup T} = I_S \cup I_T$ and $I_{S \cap T} = I_S \cap I_T$.   As the sets  on the left hand side are bases of $S$, $T$, and $S \cup T$, respectively, one has
%\[
%\rk(S) + \rk(T) - \rk(S \cup T)  =  |I_{S \cap T}| \le \rk(S \cap T).
%\]
%Submodularity implies that the opposite estimate holds, as well.
%\end{proof}

The converse to Proposition \ref{prop:freemodular} is in general false: if $u$, $v$, and $w$ are points in the Euclidean plane and none is a scalar multiple of another,  then the family $\{ \{u\}, \{v\}, \{w\} \}$ is modular and  cannot be generated by any independent set.

Let us describe some additional conditions necessary for free generation.  Fix any freely generated set family $\sc$  and, to avoid pathologies, assume   $\sc$ is finite and  $\cup_{S \in \sc} S = E$.  By  Lemma \ref{lem:freeclosed}  we may assume that $\sc$ is closed under union and  intersection,  so there exists an independent  $I$ for which $I_{\cap_{T \in \tc} T} - I_{\cup_{U \in \uc} U}$ freely generates
\begin{align}
\label{eq:discreteelement}
(\cap_{T \in \tc} T) / (\cup_{U \in \uc} U)
\end{align}
for arbitrary $\tc, \uc \su \sc$.    For convenience let us define
\begin{align*}
\ov \tc = \bigcap_{T \in \tc} T && \ul \tc = \bigcup_{T \notin  \tc} T,
\end{align*}
put
\begin{align*}
 \mc_\tc = \ov \tc / \ul \tc && E_\tc = |\mc_\tc|
\end{align*}
and set
$$
\mc / \sc = \bigcup_{\tc \su \sc} \mc_\tc.
$$

Evidently $E_{\tc} = \ov \tc - \ul \tc$.    The following is simple to verify by set arithmetic.  

\begin{lemma} \label{lem:setunion} For any $\sc$,
$$
\left | \mc /\sc \right | = |\cup_\sc S|.
$$
That is, the family $\{E_\tc: \tc \su \sc \}$ is a pairwise disjoint partition of $\cup_\sc S$.
\end{lemma} 

Lemma \ref{lem:setunion} notwithstanding, it is in general false that
\begin{align*}
\mc/\sc  = \cup_{\sc}  S.
\end{align*}
It is not even true that
\begin{align}
\rk \left ( \mc / \sc \right ) = \rk(\cup_\sc S).  \label{eq:equalrank}
\end{align}
Under certain circumstances, however, it will be within our horizon to show
\begin{align}
\ic \left ( \mc/ \sc \right )  \su  \ic (\cup_{\sc}  S).  \label{eq:independentinclusion}
\end{align}

What circumstances, in particular?   Arrange the elements of $2^\sc$ into a sequence $(\uc_i)$.  For convenience put $\mc_i = \mc_{\uc_i}$ and $E_i = E_{\uc_i}$.  We say that $(\uc_i)$ is a \emph{linear order} on $\sc$ if $E_i \su \ul \uc_j$ whenever $i \le j$.  

\begin{lemma}  Suppose that $\sc$ is finite.  If it  has a linear order, then \eqref{eq:independentinclusion} holds true.
\end{lemma}

%Let $(\uc_1, \ld, \uc_m)$ be a finite sequence in $2^\sc$, and let  $(I_1, \ld, I_m)$ be any sequence such that  $I_i$ is independent in $\mc_{\uc_i}$ for all $i$.   We begin with the simpler question of wether
%$$
%I_1 \cup \cd \cup I_m 
%$$
%is in independent in $\mc$.  For this the following simple condition will suffice: define a binary relation  $\preceq$ on $2^\sc$ by 
%\begin{align*}
%E_{\uc} \preceq E_{\tc} && \tm{\emph{iff}}  && E_{\uc} \su \ul \tc.
%\end{align*}
%
%\begin{lemma} \label{lem:linordind}  If $\uc_i \preceq \uc_j$ for all $i \le j$, then $I_1 \cup \cd \cup I_m$ is independent.
%\end{lemma}
\begin{proof}
In the special case $m = 2$ the desired result holds by definition, since $I_1$ is an independent subset of the restriction minor $\ul \uc_2$, and $\mc_2 = \ov \uc_2/ \ul \uc_2$.  The general case follows by a simple induction.
\end{proof}

%Since $\mc_\tc = \ov \tc / \ul \tc$,  for any sequence $(\uc_1, \ld, \uc_m)$ such that
%\begin{align*}
%\uc_i \preceq \uc_j  && i \le j
%\end{align*}
%and any sequence $(I_1, \ld, I_m)$ such that $I_i \in \ic(\mc_{\uc_i})$ for  $i = 1, \ld, m$, the union
%$$
%I_1 \cup \cd \cup I_m \in \ic(\mc)
%$$
%is independent in $\mc$.

%For any set family $\sc \su 2^{|\mc|}$, let $\mc_\tc$ denote the matroid minor \eqref{eq:discreteelement} with $\uc = \sc - \tc$, and set 
%\begin{align*}
%E_\tc = |\mc _\tc| && E_\bullet = \{E_\tc : \tc \su \sc\}.
%\end{align*}
%
%Fix $E_\uc, E_\vc \in E_\bullet$, and suppose that $E_\uc \su \cup_{\tc \in (\sc -  \vc)} \tc$.   Then 
%\begin{align*}
%\ic(E_\vc)
%\end{align*}
%
%
%Define a binary relation $\preceq$ on $E_\bullet \times E_\bullet$ by the rule $\uc \preceq \vc$ iff  $E_\vc \su \cup_{U \in \uc } U$.
%
%
%Finally, given any set family $\sc \su 2^{|\mc|}$, let us define a binary relation $\preceq$ on $\sc \times \sc$ by the rule
% $\{|\mc_\tc| : \tc \su \sc \}$

\begin{proposition}  
\label{prop:generatingbases}
Suppose that $\sc$ is freely generated, and $\cup_\sc S = E$.  Then the bases that generate $\sc$ are exactly the bases of $\mc/\sc$.
\end{proposition}
\begin{proof}
Recall that
\begin{align*}
 E_\tc = \ov \tc - \ul \tc,
\end{align*}
so that $\ov \tc$ may be expressed as a disjoint union
$$
\ov \tc = (\ov \tc - \ul \tc) \cup (\ov \tc \cap \ul \tc) = E_\tc \cup (\ov \tc \cap \ul \tc) 
$$
and $E$ as a disjoint union 
$$
E = \ov \tc  \cup (E - \ov \tc) =  E_\tc \cup (\ov \tc \cap \ul \tc) \cup (E - \ov \tc).
$$

For economy let us denote  $A \cap B$ by $A_B $.  By hypothesis there exists an $I \in \bc(\mc)$ such that $I_S \in \bc(S)$ for all $S \in \sc$.  Fixing  any $J \su E_\tc$, let us ask when
\begin{align}
J \cup (I_{E- E_\tc})   \label{eq:questionset}
\end{align}
will be independent in $\mc$.  Set $I$ decomposes as a disjoint union 
$$
I_{E_\tc} \;\; \cup \;\; I_{\ov \tc \cap \ul \tc} \; \; \; \cup \;\; I_{E - \ov \tc} \; \; \; = \; \; \;    I_{\ov \tc } \; \; \; \cup \;\; I_{E - \ov \tc},
$$
and since the family of sets feely generated by a basis is closed under union and finite intersection,  $I_{\ov \tc}$ is a basis for $\ov \tc$.  Thus $I_{E - \ov \tc}$ is a basis for $E/ \ov \tc$,  hence
\begin{align*}
 J \cup  (I_{E- E_\tc}) \; \; \; = \; \; \;  J \;\; \cup \;\; I_{\ov \tc \cap \ul \tc} \; \; \; \cup \;\; I_{E - \ov \tc} 
\end{align*}
is independent in $\mc$ iff 

\begin{align}
J \;\; \cup \;\; I_{\ov \tc \cap \ul \tc}   \label{eq:jcap}
\end{align}
is independent in $\ov \tc$.

Now recall that freely generated families are modular, so that, by Lemma \ref{lem:modularisomorphism},
$$
\ov \tc / \ul \tc = \ov \tc / (\ov \tc \cap \ul \tc).
$$
Since $I_{\ov \tc \cap \ul \tc}$ is a basis for $\ov \tc \cap \ul \tc$, it follows that \eqref{eq:questionset} is independent in $\mc$ (respectively,  \eqref{eq:jcap} is independent in $\ov \tc$) iff $J$ is independent in $\ov \tc / \ul \tc = \mc _\tc$.    In particular, \eqref{eq:questionset} is a basis iff $J$ is a basis in $\mc _\tc$.  It follows  that every set of form 
\begin{align*}
\cup_{\tc \su \sc} J_\tc  && J_\tc \in \bc(\mc_\tc)
\end{align*}
is a basis of $\mc$ that freely generates $\sc$.  That is, every basis of $\mc/\sc$ freely generates $\sc$.  The converse follows easily from the definition of $\mc/\sc$.
\end{proof}

\begin{proposition}
\label{prop:freelygeneratediff}
A finite set family $\sc$ is freely generated iff it satisfies  \eqref{eq:equalrank} and \eqref{eq:independentinclusion}.
%Suppose that $\sc$ has a linear order.  Then  $\sc$ is freely generated if and only if $\rk(\mc)  = \rk(\cup_{\tc \su S} \mc_\tc)$.
%$$
%\bigcup_{\tc \su \sc} (\cap_{T \in \tc} T) / (\cup_{U \in \sc - \tc} U)
%$$
%$$
%\sum_{\tc \su \sc} \rk(\mc_{\tc, \sc - \tc}) = \rk(\mc).
%$$
%has rank equal to that of $\mc$.
\end{proposition}
\begin{proof}
%To each element $e$  of $\mc$ corresponds exactly one $\tc$ such that $e \in \mc_\tc$, namely $\tc = \{T \in \sc : e \in T\}$.  The family of all $E_\tc$, where $E_\tc$ is the ground set of $\mc_\tc$, has several helpful properties.  First, it forms a pairwise disjoint partition of the ground set of $\mc$.  Second, each  $\mc_\tc$ may be expressed as as a minor of form $E_\tc / (\cup_{\uc \neq \tc} E_\uc)$ in $\mc$, so, by  definition of contraction, any union of independent sets drawn from  $\{\mc_\tc: \tc \su \sc\}$ will be independent in $\mc$.  In symbols,
%%
%\begin{align}
%\ic(\cup_{\tc} \mc_\tc) \su \ic(\mc).  \label{eq:independentinclusion}
%\end{align}
%%
%Third and finally, 
For convenience, assume  and without loss of generality that $\cup_\sc S$ is the ground set of $\mc$.

Let us first suppose that $\sc$ is freely generated.  Then  \eqref{eq:equalrank} and \eqref{eq:independentinclusion} follow directly from Proposition \ref{prop:generatingbases}.
Conversely, posit \eqref{eq:equalrank} and \eqref{eq:independentinclusion}, and let $S \in \sc$ be given.  It will suffice to show that
\begin{align}
\rk( [{\mc / \sc}] |_S )  = \rk ( \mc|_S),  \label{eq:eqrk}
\end{align}
for in this case every basis of $\mc/\sc$ will freely generate $\sc$ in $\mc$.  Therefore let us argue the lefthand side is neither greater nor less than the right.

Put $\nc = \mc \sslash S$.    It is simple to check that $\nc_\tc = \mc_\tc$ when $E_\tc \subsetneq S$, and that $\rk(\nc_\tc) = 0$ otherwise.  It follows that 
\begin{align*}
\nc/S = \mc/S && and && [{\nc/ \sc}] / S  = [\mc / \sc] / S.  
\end{align*}
It may be argued  that \eqref{eq:independentinclusion} holds with $\nc$ in place of $\mc$, so
$$
\rk(\nc/\sc) \le \rk(\nc ).
$$  
Moreover, since
$$
 \rk( [{\nc / \sc}]  | _ S )   =  \rk ( \nc |_S)  
$$ 
vanishes, one has
\begin{align*}
 \rk \left ( [{\nc/ \sc}] / S  \right )  \le  \rk( \nc / S)
\end{align*}
and therefore
\begin{align*}
\rk  ({\mc/ \sc}) -  \rk( [{\mc / \sc}] |_S ) = \rk \left ( [{\mc/ \sc}] / S  \right ) \le \rk( \mc / S) = \rk(\mc) - \rk( \mc |_S).  
\end{align*}
Since $\rk(\mc/\sc) = \rk(\cup_\sc S) = \rk(\mc)$ by hypothesis, the lefthand side of \eqref{eq:eqrk} is no less than the right.  The converse holds by \eqref{eq:independentinclusion}, and the desired conclusion follows.
%
%Moreover,
%\begin{align*}
%\nc/S = \mc/S && and && [{\nc/ \sc}] / S  = [\mc / \sc] / S.
%\end{align*}
%
%
%Inclusion \eqref{eq:independentinclusion} implies that  
%$$
%\rk( [{\mc / \sc}] |_S )  \le \rk ( \mc|_S) ,
%$$ 
%so
%\begin{align}
%\rk \left ( [{\mc/ \sc}] / S  \right ) \ge \rk( \mc / S).  \label{eq:uinequality}
%\end{align}
%One the other hand, let $\nc - \mc \sslash S$.  It is simple to check that $\nc_\tc = \mc_\tc$ when $E_\tc \subsetneq S$, and $\rk(\nc_\tc) = 0$ otherwise.  Therefore \eqref{eq:independentinclusion} holds with $\nc$ substituted for $\mc$, and consequently
%$$
%\rk( [{\nc / \sc}] |_S )  = \rk ( \nc|_S).
%$$ 
%Thus
%\begin{align}
%\rk \left ( [{\nc/ \sc}] / S  \right ) = \rk( \nc / S).  \label{eq:uinequality}
%\end{align}
%
%On the other hand, if  $\nc = \mc \sslash S$ then $(\mc/ \sc)/ S  = \nc /\sc$.  Since $\ic(\cup_\tc \nc_\tc) \su \ic (\nc)$, it follows that strict equality holds in \eqref{eq:uinequality}.   Consequently, when $\mc$ and $\mc_\tc$ have equal rank, $S$ has equal rank in both matroids.  In particular, any basis of $\cup_\tc \mc_\tc$ will freely generate $\sc$.   Conversely, if there exists a basis $B \in \bc(\mc)$ that freely generates $\sc$, then, as argued above, $B$ is independent in $\cup_\tc \mc_\tc$.  In fact, $B$ is a basis in $\cup_\tc \mc_\tc$, by \eqref{eq:independentinclusion}.  Thus  the two matroids have equal rank.
\end{proof}

%The matroid union in Proposition \ref{prop:freelygeneratediff} will reappear sufficiently often to warrant special notation:  for any set family $\sc$, we will set 
%$$
%\mc/\sc = \cup_{\tc \su \sc} \mc_\tc.
%$$  
Let us consider the form taken by $\mc/\sc$ for several important types of $\sc$.

\subsubsection{Filtrations}

A (finite) \emph{filtration} is nested sequence of sets  $\sc = \{S_1 \su \cd \su S_m\}$, where $S_m$ is the ground set.   The nonempty minors that make up $\mc/\sc$ are those of form $S_{p} / S_{p-1}$.
Since the rank of $S_p/S_{p-1}$ equals  $\rk(S_{p}) - \rk(S_{p-1})$, the ranks of $\mc$ and $\mc/\sc$ agree for every filtration  $\sc$.  Thus filtrations are both modular and freely generated, though this is easy enough to show directly -- one can always build a basis that generates $\sc$ by first finding a basis for $S_1$, extending to a basis for $S_2$, et cetera.

There is a one to one correspondence between length $m$ filtrations and functions  $\fc: E \to \{1, \ld, m\}$.  In one direction this correspondence sends $\fc$ to the sequence $(\fc_{\le p})_{p = 1}^m$, where
$$
\fc_{\le p}= \{i: \fc(i) \le p\}.
$$  
In the opposite direction the correspondence sends $\sc$  to the unique integer-valued function $\fc$ such that  $S_p = \fc_{\le p}$.  It will be convenient to identify functions with filtrations under this correspondence,  so that one may speak, for example, of the value of a filtration  $\sc$ on an element $e$, and of the matroid $\mc/\fc$ for some function $\tb F$.   \emph{Nota bene:} under this identification, $\fc_p = \fc_{\le p}$.    %It will further be convenient to extend these conventions to functions $\fc$ whose domain properly contains the ground set.  Thus $\fc$  determines a filtration on both $\mc$ and on all of its minors.

\subsubsection{Bifiltrations}

A \emph{bifiltration} is a family of form $\sc = \fc \cup \gc$, where $\fc$ and $\gc$ are filtrations.  The minors that make up $\mc/\sc$ are those of form
\[
(\fc_p \cap \gc_q) / (\fc_{p-1} \cup \gc_{q-1}).
\]
Such unions need not be  freely generated or even modular.  Unlike arbitrary set families, however, they must be either both or neither.

\begin{proposition}
\label{prop:bimodularmeansfree}
A bifiltration is freely generated if and only if it is modular.
\end{proposition}
\begin{proof}
One implication has already been established by Proposition \ref{prop:freemodular}.  For the converse, suppose $\sc$ to be modular.   Recall from the background section on modularity that $\fc_p / \gc_{q-1} = \fc_p / (\fc_p \cap \gc_{q-1})$, so that, in particular,
\[
(\fc_p \cap \gc_q)/ \gc_{q-1} = (\fc_p \cap \gc_{q})/ (\fc_p \cap \gc_{q-1}).
\]
The rank of the righthand side is $\rk(\fc_p \cap \gc_q) - \rk(\fc_p \cap \gc_{q-1})$, so a sum over $q$ telescopes to $\rk(\fc_p)$.  The left hand side is isomorphic to the restriction of  $\nc_q = \gc_q/\gc_{q-1}$ to  $\fc_p \cap |\nc_q|$, so $\fc_p$ has the same rank in $\mc$ that it has in $\mc/\gc = \cup_q \nc_q$.  Any basis that generates $\fc$ in the latter matroid will therefore generate both filtrations in the former.
\end{proof}

\subsubsection{$N$-filtrations}

It is a simple matter to show that a union of $N \ge 2$ filtrations may be modular without being freely generated, even in the special case where each is a linear filtration on a vector space.   Given  three distinct lines $\ell_1, \ell_2, \ell_3$ in the Euclidean plane, for example, the family $\{\ell_1, \ell_2, \ell_3 \}$ is modular, but  cannot be generated by an independent set.

\section{Minimal Bases}
An object of basic interest in the study of matroids is the weighted basis.  Given any real-valued function $\fc$ on $E$, we may define the \emph{$\fc$-weight} of a finite subset $S \su E$ to be the sum $\sum_S \fc(s)$.  A basis is \emph{$\fc$-minimal} (respectively, $\fc$-maximal) if its weight is no greater (respectively, no less) than that of any other basis.  Here and throughout the remainder of this text, we will assume that  weight functions take finitely many values, hence  minimal bases will always exist.

It is an important property of minimal bases that they determine a matroid independence system in their own right.  By extension of the notation introduced in the preceding section, let us identify $\fc$ with the family of sublevel sets 
$$
\{\fc_{\le \ek} : \ek \in \R\}
$$
so that $\mc/\fc = \cup_{\ek \in \R} \; \fc_{\ek} / \fc_{< \ek}$.

\begin{lemma}
\label{lem:genlem}
The $\fc$-minimal bases of $\mc$ are the bases of $\mc / \fc$.  Equivalently, they are the bases that generate $\fc$.
\end{lemma}
\begin{proof}
Fix a basis $B$, and suppose there exists an $\ek$ for which $B_{\fc \le \ek}$ is not a basis of $\fc_{ \le \ek}$.  There exists at least one $e$ in $\fc _{\le \ek}$ such that $B_{\fc \le \ek} \cup \{e\}$ is independent, and this union may be extended to a basis by adding some elements of $B_{\fc > \ek}$.  The new basis will be identical to $B$, except that one element of $B_{\fc > \ek}$ will be replaced by  $e \in \fc_{\le \ek}$.   The old basis evidently outweighs the new, so $B$ is not minimal.   Therefore $B_{\fc \le \ek}$ freely generates $\fc_{ \le \ek}$, whenever $B$ is minimal.  The converse is immediate, since every basis that satisfies this criterion has identical weight.
\end{proof}

The following is a vacuous consequence, but merits expression for completeness.  We say that a basis is $\fc$-$\gc$ minimal if it is minimal with respect to both $\fc$ and $\gc$.

\begin{proposition} Let filtrations $\fc$, $\gc$ be given.   Either
\begin{itemize}
\item $\fc \cup \gc$ is not modular, and 
\item no $\fc$-$\gc$ minimal basis exists
\end{itemize}
or 
\begin{itemize}
\item $\fc \cup \gc$ is modular,
\item   $\fc$-$\gc$ minimal basis exist, 
\item they are exactly the bases of  $\mc/ (\fc \cup \gc)$, and
\item they are exactly the bases that generate $\fc \cup \gc$.
\end{itemize}
\end{proposition}

%
%\begin{corollary} If $\fc \cup \gc$ is modular, then the $\fc$-$\gc$ minimal bases determine a matroid structure on $E$.  In particular, they are the bases of $\mc / (\fc \cup \gc)$.
%\end{corollary}

\begin{proof}
 The bases that generate $\fc \cup \gc$ are the $\fc$-$\gc$ minimal bases, by Lemma \ref{lem:genlem}.   Such bases exist if and only if $\fc \cup \gc$ is modular, by Proposition \ref{prop:bimodularmeansfree}.  When they do exist, they coincide exactly with the bases of  $\mc / (\fc \cup \gc)$, by Proposition \ref{prop:generatingbases}.
\end{proof}

\chapter{Canonical Forms}
\label{ch:canonicalforms}

This chapter  introduces a combinatorial abstraction of the idea of a Jordan basis for a nilpotent operator.  This is not the first combinatorial model of a nilpotent canonical form.  A closely related notion  for complete $\vee$-homomorphisms on  atomistic lattices was introduced by J.\ Szigeti in \cite{SZIGETI2008296}.  The idea of Szigeti was to understand  operators via their action on the lattice of subspaces of the domain.  In this formulation, orbits of vectors are replaced by orbits of one-dimensional subspaces, or more generally \emph{atoms}, and bases are replaced by families of atoms that join irredundantly  to  the maximum element of the lattice.

While the language of posets dramatically increases the reach of traditional canonical forms, several structures of interest break down in this more general context.   Of particular significance, distinct bases in a lattice that lacks the Jordan-H\"older property may have different cardinalities.  Moreover, while a sufficient condition for the existence of Jordan bases is given in \cite{SZIGETI2008296}, no structural characterization  is given.  Our main result states that in any lattice where bases satisfy the Exchange Axiom, an exact condition for existence can be given, and the corresponding Jordan bases may be characterized combinatorially.

\section{Nilpotent Canonical Forms}
\label{sec:nilpotentforms}

Let $\mc$ be a matroid with ground set $E$.  We  say a set function $T: E \to E$   is $\vee$\emph{-complete} if $T  (\cl(S)) \su \cl(TS)$ for every $S \su E$.   

\begin{remark}
If $\mc$ is a simple, then $T$ is $\vee$-complete if and only if the rule 
$$
F \mapsto \cl(TF )
$$ 
determines a $\vee$-complete homomorphism on the associated lattice of flats.
\end{remark}

\begin{remark}  In the field of matroid theory, $\vee$-complete functions are called \emph{strong maps}.  Much is known about these maps and the category they generate;  see for example \cite{white_1986, Heunen2017}.
\end{remark}

Several properties of strong maps follow directly from the definition.  %First, $T$ is $\vee$-complete iff $T^m$ is $\vee$-complete for all nonnegative $m$, since one can ``pass'' sequential copies of $T$ across the closure operator to form an increasing sequence of sets ranging from $T^m \cl(S)$  to $\cl(T^m S)$, for any $S$.
The first is closure under composition: $TU$ is strong whenever $T$ and $U$ are strong.
The second concerns a notion of null space.  Suppose that $\mc$ contains a loop $0$.  Define the \emph{kernel} of $T$ by  $\K(T)= \{e \in E : T(e) \in \cl(0)\}$. This is the natural counterpart to the notion of a kernel introduced in \cite{SZIGETI2008296}.    One may argue that
\begin{align}
\rk(TS) \le \rk(S/\K(T))  \label{eq:TSineq}
\end{align}
for any subset $S$, as follows.  Extend a basis, $J$, of $\K(T)$ to a basis $I \cup J$ of $S \cup \K(T)$.   One then has $\rk(TS) = \rk(T(S \cup \K(T))) = \rk(T I)$.  The righthand side is bounded by $|I| = \rk(S/\K(T))$, hence \eqref{eq:TSineq}.   We will say that $T$ is \emph{complementary} if strict equality holds for every subset $S$.\footnote{In the standard language of strong maps, a morphism is complementary if and only if it has defect zero.}

In the special case where $S$ is the entire ground set, \eqref{eq:TSineq} implies that
\begin{align}
\rk ( \tm{K}(T)) + \rk( \Im({T})) \le \rk(\mc). \label{eq:kernelimcomplementary}
\end{align}
Strict equality holds in \eqref{eq:kernelimcomplementary} when $T$ is complementary, and in fact, the converse holds as well.  Why?  If sets $I$, $I \cup J$, and $I \cup J \cup B$ are bases for $\K(T)$, $\K(T) \cup S$, and $\mc$, respectively, then equality in \eqref{eq:kernelimcomplementary} implies the second identity in
\[
\rk(T(J \cup B)) = \rk(\Im(T)) = \rk(\mc) - \rk(\K(T)) = \rk(\mc) - |I| = |J \cup B|.
\]
Thus $T(J)$ is independent, so $\rk(TS) = |J| = \rk(S/\K(T))$.

\begin{remark}  A simple consequence of the preceding observation may be stated as follows.  Let  $T$ be a complementary strong map $\mc \to \mc$.  Then $T$ factors as
$$
\mc \lr{q} \mc  \sslash \K(T) \lr{r} \mc
$$
where $q$ and $r$ are strong maps and
\begin{itemize}
\item  $q$ is identity (as a set function)
\item  $r$ is an embedding, in the sense that $\rk_\mc(r(S)) = \rk_{\mc \sslash \K(T)}(S)$ for all $S \su |M|$.
\end{itemize}
Thus every complementary strong map may be viewed as a contraction followed by an embedding.
\end{remark}

We will say that $T$ is \emph{nilpotent} if $T^n \su \cl(0)$ for some $n$.  The \emph{orbit} of an element $e$ under a nilpotent operator $T$ is the set (possibly empty) of all nonzero elements that may be expressed in form $T^m e$, for some nonnegative $m$.   A basis is \emph{Jordan} with respect to $T$ if it may be expressed as the disjoint union of some $T$-orbits.

A reasonable point of departure for the study of  Jordan bases is the relation between kernels, images, and nilpotence. Provided $T^n = 0$, one may define an increasing sequence of kernels 
$$
\kc: \; \; \Ker(T^0) \su \cd \su \Ker(T^n)
$$ and one of images  
$$
\ic: \; \; \Im(T^n) \su \cd \su \Im(T^0),
$$ each terminating with the ground set.  We call these the \emph{kernel} and \emph{image} filtrations, respectively, of $T$.  %It will be convenient to  write $\kc$ and $\ic$ for the  integer-valued functions on $\mc$ such that $\kc_{\le m} = \Ker(T^m)$ and $\ic_{\ge m} = \Im(T^m)$.

\begin{remark} Our notation for the image filtration is slightly unfortunate, as it conflicts with the universal convention that $\ic$ should denote a family of independent sets.  The reward of this transgression is the clarity it lends to certain duality results,  c.f.\ Corollary \ref{cor:jordanconditions}. 
\end{remark}

\begin{lemma}
\label{lem:kerimgen}
Every Jordan basis of a nilpotent strong map $T$ freely generates the kernel and image filtrations of $T$.
\end{lemma}
\begin{proof}
Suppose $J$ is a Jordan basis of $T$.  Since $J _{\kc \le m}$ and $T^mJ_{\kc > m}$ are independent subsets of $\Ker(T^m)$ and $\Im(T^m)$, respectively, inequality \eqref{eq:kernelimcomplementary} implies that $J_{\kc \le m}$  generates $\Ker(T^m)$, and $T^mJ_{\kc > m}$  generates $\Im(T^m)$.
\end{proof}

%In fact, strict equality in X implies strict equality in Y, with $T^m$ substituted for $T$.
%
%
%
% The \emph{kernel} of $\lk$ is $K(\lk) = \{e \in E: \lk(e) = 0\}$ and the \emph{image} of $\lk$ is $Im(\lk) = \{\lk(e): e \in E\}$.   To each nilpotent $\lk$ corresponds a unique weight function $\kc$ such that $K(\lk^n) = \{\kc \le n\}$.  By definition, the \emph{rank} of $\lk$ is the rank of its image, denote $\rk(\lk)$, and the \emph{nullity} of $\lk$ is the rank of its kernel, denoted $\nk(\lk)$.
%
%
%We say that $T$ is \emph{nilpotent} if its domain has at least one loop, $0$, and $T^n \su \cl(0)$, for some $n$.
%%
%%Suppose that a matroid $\mc = (E, \ic)$ has a zero element 0.  For convenience we will assume that exactly one such exists, though in general this is not required.  We say that a set function $\lk: E \to E$ is \emph{$\vee$-complete} if $\lk(0) = 0$ and the closure of $\lk(S)$ equals that of $\lk( \cl(S))$, for all $S \su E$.   We say $\lk$ is \emph{nilpotent} if  $\lk^n (E) = \{0\}$ for some $n$.  Without loss generality, $\lk^{n-1} \neq \{0\}$.
%%
%
%
%The \emph{orbit} of an element $e$ under $\lk$ is  the (possibly empty) set of nonzero elements of $E$ that may be expressed in form $\lk^k e$, for some $k$.  A basis of $\mc$ is \emph{Jordan} with respect to $\lk$ if it can be expressed as the disjoint union of some orbits.

\begin{proposition}
\label{prop:complementaryiffjordan}
If $T^m$ is complementary for all nonnegative $m$, then $J$ is a Jordan basis iff it may be expressed as the orbit of a $\kc$-minimal basis of \emph{$\mc/ \Im(T)$}.
\end{proposition}
\begin{proof}
That every Jordan basis  may be expressed as the orbit of a $\kc$-minimal basis of $\mc/\I(T)$ follows  from Lemma \ref{lem:genlem}, which characterizes the minimal bases of a filtration as those that freely generate it, and from Lemma \ref{lem:kerimgen}.  

For the converse, suppose  $J$ to be the orbit  a  $\kc$-minimal basis $B \in \bc(\mc/\I(T))$.  We will show that $J$ is a basis.  To establish independence, assume for a contradiction that $J$ contains a circuit $\zk$.   Fix an integer $m$ and a subset $\omega \su J$ such that $\zk = T^m \omega$ and $ \omega \cap B$ is nonempty.   Since each   $z \in \zk$ lies in the closure of $\zk - \{z\}$, complementarity implies that  $w \in \omega$ lies in the closure of $\omega - \{w\}$ in  $\mc/\kc_{\le m}$.  Since $\kc$ takes values strictly greater than $m$ on $\omega$, it follows \emph{a fortiori} that $w$ lies in the closure of $\omega - \{w\}$ in $\mc/(\kc_{< w} \cup \I(T))$.  However, if we take $w \in \omega \cap B$ to have maximum $\kc$-weight this implies a contradiction, since $w$ includes into a $\kc$-basis in $\mc/\I(T)$.  Thus $J$ is independent.

To see that $J$ spans $\mc$, let $\nc$ be the matroid obtained by introducing a unique zero element into the contraction minor $\mc/ \cl(J)$.  Evidently, $T$ induces a nilpotent map $Q$  on $\nc$ that sends $e$ to $Te$ if $e$ lies outside $\cl(J)$, and to zero otherwise.   Since $\cl_{\nc}(S) = \cl(J \cup S ) - J$ for any $S$, an element $j \in \mc - J$ belongs to $T \cl_\nc(S)$ if and only if it lies in
\[
T(\cl(J \cup S )) \su \cl(T (J \cup S)) \su \cl(J \cup TS).
\]
Inclusion in the righthand side is equivalent to membership in $\cl_\nc(TS)$, for $j$ outside $J$, so $Q\cl_\nc(S) \su \cl_\nc(QS)$.  In particular, $Q$ is $\vee$-complete.

As $Q$ is evidently nilpotent, it follows that either $\rk(Q \nc) < \rk(\nc)$ or $\nc$ has rank zero.  If the former holds, then $\nc/ Q\nc$ will have positive rank.  This is impossible, since the independent sets of $\nc/Q \nc$ are exactly those of $\mc / (J \cup T\mc)$, and the latter has rank zero.  Therefore $\nc$ has rank zero, whence $J$ is a basis.
\end{proof}

Jordan bases admit a natural dual characterization.  Let us say that  $\{e, \ld, T^m e\}$ is a \emph{preorbit} of $T^me$ if there exists   if 
$$
e \in T^m \mc - T^{m+1} \mc.
$$
A preorbit of a set $S$ is a union of form $\cup_{s \in S} J_s$, where  for each $s \in S$ the set $J_s$ is a preorbit of $s$.   The proof of the following observation is entirely analogous to that of Proposition \ref{prop:complementaryiffjordan}.  The details are left as an exercise to the reader.

\begin{proposition}
A nilpotent strong map $T$ has a Jordan basis if and only if $T^m$ is complementary for all nonnegative $m$.
\end{proposition}
\begin{proof}
The ``if'' portion follows from Proposition \ref{prop:complementaryiffjordan} below.  The ``only if''  follows from proof of Lemma \ref{lem:kerimgen}, where we showed that every Jordan basis may be partitioned into two disjoint subsets, one generating the kernel of $T^m$ and the other its image.
\end{proof}

\begin{proposition}  If $T^m$ is complementary for all nonnegative $m$, then the Jordan bases of $T$ are  the preorbits of $\ic$-maximal basis in $\K(T)$.
\end{proposition}

In summary,  we have the following.

\begin{corollary}  \label{cor:jordanconditions}
If $\mc$ is a matroid and
$$
T: \mc \to  \mc,
$$ 
is a nilpotent strong map, then following are equivalent.
\begin{enumerate}
%\item $T$ has a Jordan basis.
\item $T$ has a Jordan basis. 
\item $T^m$ is complementary, for all  $m$.
%\item The orbit of some $\kc$-minimal basis in $\mc/\Im(T)$ is a basis of $\mc$.
%\item The orbit of every $\kc$-minimal basis in $\mc/\Im(T)$ is a basis of $\mc$.
\item The Jordan bases of $T$ are the orbits of  $\kc$-minimal bases in \emph{$\mc\;/\;\Im(T)$}.
\item The Jordan bases of $T$ are the preorbits of $\ic$-maximal bases in \emph{$\mc \; |  \;\Ker(T)$}.
\end{enumerate}
\end{corollary}

The fourth and final characterization is encountered quite often in practice.  The procedure  for finding a nilpotent Jordan basis outlined in \S\ref{sec:jordanalgebra}, for example, may be understood as concrete application of the classical  greedy algorithm for matroid optimization to the problem of finding a $\ic$-maximal basis in $\K(T)$.  The elements of this  argument are not new.  The basic elements were recorded at least as early as 1956 \cite{ptak1956}, and have been revisited frequently over the following decades, e.g.  \cite{BakerBlog,egervary1959,farahat1957,kurepa1967,ptak1962}, though to our knowledge none has recognized that the problem being solved was one of matroid optimization.

Let us say that an orbit $I$ is \emph{maximal with respect to inclusion} if there exists no orbit $J$ such that $I \su J$ and $I\neq J$.

\begin{corollary}[Uniqueness]  Suppose that $(I_1, \ld, I_m)$ and $(J_1, \ld, J_n)$ are pairwise disjoint families of maximal orbits for which
\begin{align*}
I_1 \cup \cd \cup I_m&& and &&  J_1 \cup \cd \cup J_n
\end{align*}
are Jordan bases.  Then  there exists a bijection $\fk: \{1, \ld, m\} \to \{1, \ld, n\}$ such that $|I_p| = |J_{\fk(p)}|$ for all $p$.
\end{corollary}

%\begin{corollary}[Uniqueness]
%Suppose  $J_1  \cup \cd \cup J_m$ and $K_1 \cup \cd \cup, K_m$ are pairwise-disjoint unions of maximal orbits.  If each is a Jordan basis, then  there is a bijection $\fk: \{1, \ld, m\} \to \{1, \ld, n\}$ such that $|J_i| = |K_{\fk(i)}|$ for all $i$.
%\end{corollary}
\begin{proof}
If $\cup_p I_p$ is a Jordan basis then
\begin{align*}
I_p = \orb(\psi(p)) && p \in \{1, \ld, m\}
\end{align*} 
for some $\kc$-minimal basis $B$ in $\mc /\I( T)$ and some  bijection $\psi: \{1,\ld, m\} \to B$.   Thus the number of orbits of given length in each Jordan basis is uniquely determined.
\end{proof}

\section{Graded Canonical Forms}
\label{sec:gradcanform}

A \emph{$\Z$-grading} on a matroid $\mc$ is a function $\hc$ that assigns a flat of $\mc$ to each integer $p$, subject to the condition that
\begin{align*}
\mathrm{cl} \left(\bigcup_p \hc_p \right ) = \mc
&& and &&
\sum_p \rk(\hc_p) = \rk(\mc).
\end{align*}
Recalling that $\cl(0)$ denotes the possibly empty family of \emph{loops} in $\mc$,  we write $\ag$ for the integer-valued function on $\cup_p\hc_p - \cl(0)$ such that  
\begin{align*}
\ag(e) = p &&e \in \hc_p - \cl(0).
\end{align*}

 \begin{example} \label{ex:lineargrades} If $\mc = (V, \ic)$, where $V$ is a $\field$-linear vector space and $\ic$ is the family of $\field$-linearly independent subsets of $V$, then the $\Z$-gradings of $\mc$ are the $\Z$-indexed families of subspaces $\hc$ such that $V$ is the internal direct sum of $\{\hc_p : p \in \Z\}$.
 \end{example} 
 
 A map $T: \mc \to \mc$ is \emph{graded of degree $k$} if 
 $$
 T\hc_p \su \hc_{p+k}
 $$ 
 for all $p$.  Unless otherwise indicated, we will write \emph{graded} for graded of degree one.  Here and throughout the remainder of this discussion we will assume that $\mc$ has finite rank, so that all graded  graded maps on $\mc$ are nilpotent.

A subset $I \su \mc$ is \emph{graded} if $I \su \cup_p \hc_p$.  To every orbit $J_q$ in a graded Jordan basis $J$ corresponds an integer interval, $$
\Supp(J_q) = \{p : J_q \cap \hc_p \neq \emptyset\}.
$$ 
The associated multiset 
$$
\{\Supp(J_q): q = 1, \ld, m\},
$$
where $J_1, \ld, J_m$ are the  orbits that compose $J$,   is  the \emph{barcode} of $J$.  

Proposition \ref{prop:uniquejordan} states that the barcode of a graded Jordan basis for $T$ is uniquely determined by $T$.

\begin{proposition}
\label{prop:uniquejordan}
If $I_1, \ld, I_m$ and $J_1, \ld, J_n$ are the orbits that compose two graded Jordan bases, then there exists a bijection $\fk: \{1, \ld, m\} \to \{1, \ld, n\}$ such that 
$$
\Supp(I_p) = \Supp(J_{\fk(p)})
$$
for each $p$ in $\{1, \ld, m\}$.
\end{proposition}
\begin{proof}
Every Jordan basis is the orbit of a $\kc$-minimal basis of  $\mc/ \I( T )$.  A graded Jordan basis, therefore, is the orbit of a $\kc$-minimal basis in  
$$
(\cup_p \hc_p)/\I(T) = \cup_p (\hc_p / \I(T)).
$$  
The orbits of any two such bases will determine identical barcodes. %Every $\kc$-minimal basis $B$ of $\mc / T \mc$ a disjoint union  of $\kc$-minimal bases of the $\hc_p/ T\mc$.  For given $p$ and $q$, $B_{\kc = q}$ has the same cardinality for any $\kc$-minimal basis of $\hc_p/T \mc$.
%Since the product function $\kc \times \ag : \mc \to \Z^2$ takes the same multiset of values on each $\kc$-minimal basis of $\hc_p/ T \mc$, for all $p$, it necessarily takes the same multiset of values on every $\kc$-minimal basis of $\cup_p \hc_p / \mc$.
\end{proof}

%The simplicity of this argument should be contrasted with the fairly substantial machinery needed to produce the same result algebraically, as formulated in the structure theorem for finitely generated graded modules over $\field[t]$.  The existence of graded Jordan bases is similarly given by Propositions X, applied to the minor $\cup_p \hc_p$.

%An important class of graded maps are the operators induced on a coproduct of vector spaces $\oplus_p \hc_p$  by a sequence of linear maps $T_p: \hc_p \to \hc_{p+1}$.  The data of such a sequence, together with the grading  $\hc_p$, is called a \emph{linear persistence module}.  Graded linear maps are basic to several standard constructions in algebraic topology, and their barcodes are regarded as fundamental statistics in the field of topological data analysis.

The remainder of this section will be devoted to a class of graded nilpotent maps with a particularly simple combinatorial structure.
%  This structure is fundamental to a range of results in homological algebra and computation; just as    spectral decomposition is basic to understanding positive operators and Smith normal form is basic to understanding finitely generated modules over a PID, the modular structure of nested filtrations is basic to understanding persistence modules.
%Motivation for this family comes from the study of nested sequences of topological spaces.  %To every topological space $X$ can be associated a pair of linear spaces $\bc \su \zc$, called the space of singular \emph{cycles} and \emph{boundaries}, respectively.
%To every nested sequence of topological spaces $X_0 \su \cd \su X_n$ gives rise to nested sequences of vector spaces $\bc_0 \su \cd \su \bc_n$ and $\zc_0 \su \cd \su  \zc_n$, with $\bc_p \su \zc_p$ for all $p$.  The natural inclusion map from $\zc_p$ into $\zc_{p+1}$ engenders a sequence of linear maps $\zc_p/\bc_p \to \zc_{p+1}/\bc_{p+1}$.
%To describe the more general class of maps to which these belong, let us drop some excess structure.
Fix filtrations $\zc$, $\bc$ on a finite-rank matroid $\mc$.  Assume that the elements of these filtrations are  closed, and that 
$\bc_{ p} \su \zc_{ p}$ for $p \in \Z.$  %These functions will take the place of our linear filtrations.
Define 
$$
\hc_p = \zc_p \sslash \bc_p,
$$ 
and let $\nc$ be the matroid union $\cup_p \hc_p$.   For each  $e$ in the ground set of $\mc$, write $e\der p$ for the copy of $e$ in $\hc_p \su \nc$.  Finally, let $T: \nc \to \nc$ be the function sending $e\der p $ to $e \der {p+1}$.
%The first question one must ask about a general pair $(\zc, \bc)$ is when
Given this data, it is natural to ask when  $T$ engenders a Jordan basis, or equivalently, when  the powers of $T$ are complementary.

\begin{lemma}
If $m$ is a nonnegative integer, then $T^m$ is complementary if and only if $(\zc_p, \bc_{p+m})$ is modular, for every $p$.
\end{lemma}
\begin{proof}
Let us write $\zc \der p$ and $\bc \der p$ for the filtrations on $\hc_p$ engendered by $\zc$ and $\bc$.  Since the sublevel sets of $\zc$ and $\bc$ are closed, one has
\begin{align*}
\K(T^m)  = \bigcup_p \left (\zc_p \der p \cap \bc_{p + m} \der p \right ) &&
\Im(T^m)  = \bigcup_p \zc_p \der{p + m}.
\end{align*}
The ranks of $\K(T^m)$ and $\Im(T^m)$ in $\nc$ are therefore given by the left and right-hand sums below.
\begin{align*}
\sum_p \rk \left ((\zc_p \cap \bc_{p+m}) / \bc_p \right)  &&
\sum_p \rk \left (\zc_p / \bc_{p+m} \right).
\end{align*}
%The term $\rk((\zc_p \cap \bc_{p+m} )/ \bc_p) $ agrees with $\rk(\zc_p \cap \bc_{p+m}) - \rk(\bc_p)$ because $\bc_p$ is a subset of both $\bc_p$ and $\bc_{p+m}$.  Submodularity implies that $\rk(\zc_p/\bc_{p+m})$ is only bounded above by $\rk(\zc_p) - \rk(
The identity below follows from the inclusion of $\bc_p$ into  the intersection of  $\zc_p$ and $\bc_{p+m}$.  The subsequent estimate is a consequence of submodularity, and holds with strict equality if and only if $(\zc_p, \bc_{p+m})$ is modular.
\begin{align*}
\rk((\zc_p \cap \bc_{p+m} )/ \bc_p) &= \rk(\zc_p \cap \bc_{p+m}) - \rk(\bc_p) \\
\rk(\zc_p / \bc_{p+m}) &\le \rk(\zc_p) - \rk(\zc_p \cap \bc_{p+m}).
\end{align*}
Since the rank of $\nc$ is $\sum_p  (\rk(\zc_p) - \rk(\bc_p))$, complementarity  holds if and only if  strict equality holds in both estimates, for all $p$.
\end{proof}

Since  $(\zc_{p+m}, \bc_p)$ is trivially modular for every nonnegative $m$, we have shown the following.

\begin{proposition}  \label{prop:jordanforses}
Operator $T$ is Jordan if and only if $\zc \cup \bc$ is modular.
\end{proposition}

In light of the preceding observation, it is reasonable to suppose that a basis that generates $\zc$ and $\bc$  may bear some special relation to the Jordan bases of $T$.  For convenience, define the \emph{orbit} of a subset $S \su E$ to be the orbit of $\psi(S)$, where $\psi$ is the map that sends each $e \in \mc$ to $e\der {\zc(e)}$ in $\nc$.

\begin{proposition}
\label{prop:fgminimalquotientjordanbases}
The Jordan bases of $T$ are the orbits in $\nc$ of the $\zc$-$\bc$-minimal bases in $\mc$.
\end{proposition}

\begin{proof}
If $B$ freely generates $\zc$ and $\bc$ in $\mc$, and if $J$ is the orbit of $B$, then $J \cap \hc_p = B \der p_{\zc \le p < \bc}$ freely generates $\hc_p = \zc_{ p} \sslash \bc_{ p}$ for all $p$.  Therefore $J$ is a Jordan basis.

Conversely, suppose that $J$ is a Jordan basis, and let 
$$
B = \bigcup_p  \{e : e\der p \in J \}.
$$  
Since $\nc$ is the matroid union of the $\hc_p$, $J$ freely generates both $\Im(T^m) \cap \hc_p$ and $\Ker(T^m) \cap \hc_p$ for all $m$ and $p$.  Thus $B$ freely generates $\zc$ and $\bc$ on the minor $\zc_p/\bc_p$.  Consequently, if $\zc'$ and $\bc'$ are the restrictions of $\zc$ and $\bc$, respectively, to $\zc_p - \bc_p$, then $B$ freely generates $(\zc'_{q+1} \cap \bc'_{p+1}) / \bc_p$ and $\zc'_{q}/\bc_p$ for all  $q$.  By modularity, it  generates
$(\zc_{q+1}' \cap \bc_{p+1}') / (\zc_q' \cup \bc_p)$.
When $q <  p$ one has $\zc'_{q+1} \cap \bc_{p+1}' = \zc_{q+1} \cap \bc_{p+1}$ and $\zc_q' = \zc_q$, so  this minor agrees with
\begin{align}
(\zc_{q+1} \cap \bc_{p+1}) / (\zc_q \cup \bc_p).  \label{eq:zbmaxmin}
\end{align}
When $p \le q$ the ground set of \eqref{eq:zbmaxmin} contains only elements $e$ for which $\bc(e) = p+1 \le q+1 = \zc(e)$.  Since $e\der p$ is a loop for every such $e$, it follows that $B$ intersects the ground set of \eqref{eq:zbmaxmin} trivially when $p \le q$.

In conclusion, $B$ is the disjoint union of some independent sets in the matroid $\mc/(\zc \cup \bc)$, and may thus be extended to a basis that generates $\zc$ and $\bc$.  The orbit of this set will be a Jordan basis containing $J$.
\end{proof}

In closing, let us return to the linear regime.  Suppose that $\zc_{p}$ and $\bc_{ p}$ are linear filtrations on a vector space $W$, and let $V_p$ denote the linear quotient space $\zc_p/\bc_p$.   The inclusion maps $\zc_p \to \zc_{p+1}$  induce linear maps $V_p \to V_{p+1}$, which collectively determine a graded linear operator on $\oplus_p V_p$.  Let us denote this map by $Q$.  There is a canonical set function $\phi:\cup_p \hc_p \to \cup_p V_p$,  which may be described as the rule that assigns   $e \in \hc_p$ to the equivalence class of $e$ in $V_p$.   For each $p$ there is a commutative square
\vspace{.4cm}
\[
\xymatrix@R=3pc@C=3pc@M=0.5pc{
	\ar[];[rr]^T  \ar[];[d]
	\ar[rr];[drr]  \ar[d];[drr]^Q
	\hc_p  && \hc_{p+1} \\
	V_p && V_{p+1}
}
\vspace{.35cm}
\]
whose vertical maps are given by $\phi$.  Since  $S \su \hc_p$ is independent in $\nc$ iff  $\fk(S)$ is linearly independent in $V_p$, it follows that $J \su \nc$ is a matroid theoretic Jordan basis of $T$ if and only if $\fk(J)$ is a linear Jordan basis of  $Q$.  Thus the following.

\begin{proposition}
\label{prop:gradedquotientbases}
The  graded Jordan bases of $Q$ are the orbits of the $\zc$-$\bc$ minimal bases in $W$.
\end{proposition}

This fact is reflected by a celebrated result in homological persistence, the so-called \emph{Elder Rule}.  If one identifies each element of $\mc/(\zc \cup \bc)$ with its orbit under $T$, then the Elder rule coincides exactly with the matroid greedy algorithm applied to $\mc$ with weight functions $\zc$ and $\bc$.

\begin{corollary}
The Elder rule returns a Krull-Schmidt decomposition of (finitely supported, pointwise finite dimensional) homological persistence module.
\end{corollary}
\begin{proof}
Follows directly from Proposition \ref{prop:gradedquotientbases}.
\end{proof}

\section{Generalized Canonical Forms}
\label{sec:gencanform}

Let $\field[x]$ denote the ring of polynomials over ground field $\field$.  Recall that this object consists of a $\field$-linear vector space freely generated by a basis of formal symbols $\{1, x, x^2, \ld \}$, and  a binary operation $\field[x] \times \field[x] \to \field[x]$ sending $(p,q)$ to $p \cdot q$, where $\cdot$ is the standard multiplication of polynomials.  A polynomial  is  \emph{irreducible} over $\field$ if it cannot be expressed as the product of two polynomials of positive degree.  We will write $(p)$ for the \emph{ideal} generated by $p$, which may be formally expressed 
$
\{q\cdot p : q \in \field[x]\}.
$

\begin{remark}  As $(p)$ is a linear subspace of $\field[x]$ we may form the quotient space $\field[x]/(p)$ by the usual coset construction.  It is typical, in any such construction, to write $[p]$ for the equivalence class of a vector $p$ in the quotient space.  In order to avoid an excessive number of brackets, however, we will bypass this convention, writing $p$ for $[p] \in \field[x]/(p)$.  
\end{remark}

Since left-multiplication with $x$ carries every element of $(p)$ to a polynomial in the same set, the linear map $p \mapsto x \cdot p$ determines an operator $T$ on the quotient space $\field[x]/(q)$.   If $q = a_0 + a_1 x + \cd + a_d x^d$ then $\{1, x, \ld, x^{d-1}\}$ is a basis for $\field[x]/(q)$, as is simple to check.  The matrix representation of $T$ with respect to this basis has form
\vspace{0.2cm}
\[
\Com(q) =
\left (
\arraycolsep=8pt\def\arraystretch{.8}
\begin{array}{ccccc}
	0 & 0 & \cd & 0 & -a_0 \\
	1 & 0 & \cd & 0 & -a_1 \\
	0 & 1 & \cd & 0 & -a_2 \\
	   &    &  \vdots & & \\
	0 & 0 & \cd & 0 & -a_{d-2} \\
	0 & 0 & \cd & 1 & -a_{d-1}
\end{array}
\right)\vspace{0.4cm}
\]
If $q$ is irreducible then this matrix has full rank, since $a_0$ does not vanish.  If $a$ is any positive integer, then 
\[
\{q^{a-1},xq^{a-1}, \ld, x^{d-1}q^{a-1},q^{a-2}, x q^{a-2}, \ld, x^{d-1} q^{a-2}, \ld, 1, x, \ld x^{d-1}\}.
\]
is a basis for $\field[x]/(q^a)$, and in this basis multiplication with $x$ has the following matrix form
\vspace{0.2cm}

\[
J(q,a) =
\left (
\begin{array}{cccccc}
\Com(q) & M & 0 & \cd & 0 & 0 \\
0 & \Com(q) & M & \cd & 0 & 0 \\
&&& \vdots && \\
0 & 0 & 0 & \cd & \Com(q) & M \\
0 & 0 & 0 & \cd & 0 & \Com(q)
\end{array}
\right ) \vspace{0.6cm}
\]
where $M$ is the square matrix with 1 in the top right entry and zeros elsewhere.  We call this array a \emph{generalized Jordan block}.  Since $\Com(x + a_0) = (-a_0)$, the notion of a generalized Jordan block specializes to that of a classical Jordan block, when $q$ is a linear polynomial.    The following is a standard  in linear algebra.

\begin{theorem}[Generalized Jordan Canonical Form] \label{thm:generalizedcanonicalform}
If $T$ is a linear operator on a finite dimensional vector space $V$ over an arbitrary field \emph{$\field$}, then there exists a basis of $V$ with respect to which $T$ has matrix form 
$$
\mathrm{diag}(J(p_1, a_1), \ld, J(p_m, a_m)),
$$
where $p_1, \ld, p_m$ are  polynomials irreducible over \emph{$\field$}, and $a_1,  \ld, a_m$ are positive integers.  Such a presentation is unique up to a permutation of generalized Jordan blocks.
\end{theorem}

A block-diagonal matrix of the form $\tm{diag}(J(p_1, a_1), \ld, J(p_m, a_m))$  is said to have \emph{generalized Jordan Canonical Form}.  The corresponding basis is  a \emph{generalized Jordan basis}.   We will refer to the set of basis vectors that index a generalized Jordan block as a \emph{Jordan orbit}.   As an application of the classification of nilpotent Jordan bases for matroids, let us classify the generalized Jordan bases of $\field$-linear maps.

To begin, let $T$ be any operator on a finite-dimensional $\field$-linear vector space $V$.   By Theorem \ref{thm:generalizedcanonicalform}, there exist irreducible polynomials $p_1, \ld, p_m$, positive integers $a_1, \ld, a_m$ and a linear isomorphism $\Phi$ from $V$ to
$$
U = \bigoplus_{i=1}^m \field[x] /(p_i ^{a^i})
$$ 
such that $\Phi T = X \Phi$, where $X$ is the linear map on $U$ induced by left-multiplication with $x$.  The generalized Jordan bases of $T$ are exactly those of form $\Phi\inv(J)$, where $J$ is a generalized Jordan basis of $X$, so we may deduce most of what we need to know about general operators from the special case $T = X$.

For convenience, let us denote by $p$ the map $U \to U$  induced by left-multiplication with $p$.   Under this convention
\begin{align}
{\mathrm {\tb{K}}}\left (p^{\dim V} \right) = \oplus_{i \in I} \field[x]/ \left (p^{a_i} \right)\label{eq:polyform}
\end{align}
where $I = \{i : p_i = p\}$,  for each $p$.  Every Jordan orbit is contained in a subspace of form $\K(p^{\dim V})$, and every subspace of this form is  invariant under $X$.  Thus every generalized Jordan basis may be expressed as a disjoint union
$\cup_p J_p$,
where $p$ runs over irreducible polynomials and $J_p$ is a generalized Jordan basis for the restriction of $X$ to $\K(p^{\dim V})$. 

 Let us fix an irreducible polynomial $p$ with degree $d$, and assume temporarily that
$$
U =  \K \left (p^{\dim V} \right).
$$  
%the generaized Jordan bases of $X$ are the unions of the generalized Jordan bases for the restrictions $X |_{K(p^{ \dim  V})}$.
%so each generalized Jordan basis may be expressed as the disjoint union of the restrictions of $X$ to
%
%To begin, consider any map realized by multiplication with $x$ on a coproduct of form $V = \oplus_{i=1}^m \field[x] p_i ^{a^i}$, where $p_1, \ld, p_m$ are irreducible polynomials.   If we identify each $p_i$ with the action of left-multiplication by $p_i$ on $V$, then
%\begin{align}
%K(p^{\dim V}) = \oplus_{p_i = p} \field[x]/(p^{a_i}).
%\end{align}
%for each irreducible polynomial $p$.  Since the space on the lefthand side of this identity does not depend on  one's choice of decomposition, it follows that every generalized Jordan basis lies in the union over all irreducible polynomials $p$  of the subspaces $K(p ^{\dim V})$.  We may therefore restrict our attention, \emph{a priori}, to the maps realized by multiplication with $x$ on $\oplus_i \field[x]/(p^{a_i})$, for some fixed irreducible $p$.
If $E_k = \Im(p^k)$, then for each nonnegative integer $k$ the quotient module
$$
U_k = E_k / E_{k+1}%\Im\left(p^{k} \right)/\Im \left(p^{k+1} \right).
$$  
is a direct sum of simple modules isomorphic to $\field[x]/(p)$.  We define a subset $S \su E_k$ to be \emph{independent} if $\pi_k S$ generates a submodule of length $|S|$ in $U_k$, where $\pi_k$ is the quotient map $E_k \to U_k$.  
%Let $\pi_k$ be the qoutient map $E_k \to U_k$, and define $S \su E_k$ to be \emph{independent} if  $\pi_k(S)$ generates a submodule isomorphic to a  direct sum of $|S|$ copies of $\field[x]/(p)$ in $U_k$.   
Let us denote the family of all such sets by $\ic_k$.

\begin{lemma} The pair
$$
\nc_k = (E_k, \ic_k).
$$
is a matroid independence system.
\end{lemma}
\begin{proof}  Both the Steinitz Exchange Axiom and closure under inclusion follow from the Krull-Schmidt theorem for semisimple modules of finite length.  %For instance, if $R$ and $S$ are independent sets with $|R|<|S|$, then there exists an $s \in S$ that may not be expressed in form $qr$ for some $q \in \field[x]$ and some $r \in R$.  The union $\{R\} \cup \{s\}$ is independent, by Krull-Schmidt.  %The independence criterion is equivalent to the condition that $\pi(s)$ be nonzero for all $s \in S$, and that the map $\oplus_S (\pi(s)) \to U_k$ have trivial kernel.    That $\ic_k$ is closed under inclusion follows from the Krull-Schmidt theorem for left $\fiekd[x]$ modules of finite length.  %It is evident that every subset of an independent set is independent.  To see that $\ic_k$ satisfies the Steinitz Exchange axiom,  note that $S$ is independent iff $\pi(S)$ lies in no submodule isomorphic to a direct sum of $|S|-1$ copies of $\field[x]/(p)$.  Consequently, if $R$ and $S$ are independent sets with $|R|<|S|$, then there exists an  $s \in S -  \{q r: q \in \field[x], \; r \in R\}$.  The union $R \cup \{s\}$ is independent.
\end{proof}

If $\nc$ is the matroid union $\cup_k \nc_k$, and $e \der k$ is the copy of polynomial $e$ that lies in the ground set of $\nc_k$, then one may define an operator $P:\nc \to \nc$ by
$$
P \left ( e \der k \right ) = \left( p \cdot  e \right ) \der {k+1}.
$$
 As $P^m$ is complementary for all $m$, it has a Jordan basis.  We claim that the generalized Jordan bases of $X$ are exactly the sets of form
\begin{align}
X^0 J \cup \cd \cup X^{d-1} J  \label{eq:xcup}
\end{align}
where $J$ is a Jordan basis of $P$.

To see that this is the case, let $J$ be such a basis, and for convenience put $\pi_k j = 0$ for all $j \notin E_k$.  A set $I \su E_k$ freely generates $\nc_k$ as a matroid iff $\pi_k I$ freely generates $U_k$ as a module, so the nonzero elements of $\cup_k \pi_k J$ freely generate $\oplus_k U_k$ as a module.  
For each $j$ and $k$ so that $\pi_k j \neq 0$, the submodule generated by $\pi_k j$ in $U_k$ is the linear span of  $\{\pi_k X^0j, \ld,\pi_kX^{d-1}j \}$, so the nonzero elements of
$$
\bigcup_k \bigcup_{m=0}^{d-1} \pi_k X^m  J
$$
freely generate $\oplus_k U_k$ as a $\field$-linear vector space.  It follows easily that \eqref{eq:xcup} is a basis for $U$.   Evidently, it is a generalized Jordan basis.

 %\eqref{eq:xcup} represents a basis of $\oplus_k U_k$, hence also of $U$.  It is evidently a union of Jordan orbits, and therefore a generalized Jordan basis.   Conversely, every generalized Jordan basis may be expressed in form \eqref{eq:xcup} for some $J$, and it is vacuous to show that such $J$ are {\em bona fide} Jordan bases of $P$.

To describe the generalized Jordan bases of an arbitrary operator $T$, one need only synthesize across irreducible polynomials.   For a given operator $T$ on vector space $V$, let us define $U \der p \su V$ to be the module on ground set $p^{\dim(V)}(T)$ on which $x$ acts by
$$
x \cdot v = p(T) \cdot v.
$$
Set $E \der p_k = x^k U\der p = p^k(T) U \der p$, and put
$$
U_k \der p = E \der p _k/ E \der p_{k+1}.
$$ 
Let $\ic_k \der p$ be the family of independent sets on $E \der p_k$ given by the module structure of $U \der p_k$, as described above, and  $\nc_k \der p = (E_k\der p, \ic_k \der p)$ be the corresponding matroid.  Write $P \der p$ for the operator on $\nc \der p = \cup_k \nc_k \der p$ that sends each  $e \in \nc_k \der p$ to the element  $x\cdot e \in \nc_{k+1}\der p$.

Finally, let $\nc  = \cup_{p} \nc \der p$ be the matroid union of $\nc_k \der p$ over all irreducible polynomials $p$, and  $P$ be the operator on $\nc$ determined by the maps $P \der p$.     If the \emph{generalized Jordan orbit} of a vector $i \in U \der p$ is the union $T^0 I \cup \cd \cup T^{\deg (p)-1} I$, where $I$ is the usual orbit of $i$ under  $x$ in $U \der p$, then we have shown the following.

\begin{proposition}
If $T$, $\nc$, and $P$ are  as above, then the generalized Jordan bases of $T$ are  the generalized orbits of the $\kc$-minimal bases of $\nc/ P \nc$, where $\kc$ is the unique integer-valued weight function so that \emph{$\kc_{\le m} = \Ker(P^m)$}.
\end{proposition}

%\afterpage{\blankpage}

\part{Algorithms}

\chapter{Algebraic Foundations}
\label{ch:exchangeformulae}

This chapter lifts several classical ideas from the domain of matrix algebra and linear matroid representations to that of \emph{abelian} and \emph{preadditive categories}.  For readers unfamiliar with the language of category theory, the terms \emph{object} and \emph{morphism} may be replaced  by $\field$\emph{-linear vector space} and $\field$\emph{-linear map}, respectively.  The term \emph{map} is occasionally used in place of  morphism.   A \emph{monomorphism} is an injection and an \emph{epimorphism} is a surjection.   An \emph{endomorphism} on $W$ is a morphism $W \to W$, and an \emph{automorphism} is an invertible endomorphism.  With these substitutions in place, the phrases \emph{in a preadditive category} and \emph{in an abelian category} may be stricken altogether.

\section{Biproducts} 
\label{sec:biproducts}

A \emph{product structure} on an object $W$  is a family $\lk$ of maps $f:W \to \D\op (f)$, with the property  that  $\times \lk  : W \to \times_{f \in \lk} \D\op(f )$  is invertible.  A \emph{coproduct structure}  $W$ is a family $\uk$ of maps  $f: \D(f) \to W$, such that  $\oplus \uk: \oplus_{f \in \uk} \D(f ) \to W$ is an isomorphism. %\footnote{Strictly speaking we should write $\times_{\lk} f: W \to \times_\lk \D\op(f)$ for the map associated to a product, since this construction derives from a separate formalism.  We forgo this notational distinction in the interest of legibility, as the two constructions coincide for all cases examined by this text.}

To every product structure $\lk$ corresponds a unique \emph{dual} coproduct structure $\lk^{\flat}$ and  bijection $\flat: \lk \to \lk^\flat$ such that $f^\flat g = \dk_{f  g}$.    Likewise, to every coproduct structure $\uk$ corresponds a unique dual product structure $\uk^{ \sharp}$ and bijection $\sharp: \uk \to \uk^{\sharp}$ such that $f  g^\sharp = \dk_{f g}$.   We refer to an unordered pair consisting of a product structure and its dual  as  a \emph{biproduct structure}.

A \emph{complement} to a family $\lk$ of morphisms out of $W$ is a map $g$ such that $\lk \cup \{g\}$ is a product structure.  A \emph{complement} to a family $\uk$ of morphisms into $W$ is a morphism $g$  such that $\uk \cup \{g\}$ is a coproduct structure.

%Given any subset $a$ of a (co)product structure $\lk$, we define $a^\flat = \{f^\flat : f \in a\}$ and $a^\sharp$ likewise.  Similarly, given any map $T$, we write $Ta$ for $\{Tf : f \in a\}$. 
%For any (co)product structure $\lk$ and any $a \su \lk$, we write $\lk_a^\flat$ for $\oplus_a f^\flat: \D(f^\flat) \to W$, and  $\lk_a^\sharp$ for $\oplus_a f^\sharp : W \to \oplus_a \D\op(f^\sharp)$.

\begin{remark}  
The following mnemonic may be helpful in recalling the distinction between $\flat$ and $\sharp$.  The former  denotes a gravitational drop in tone,  with $\flat$ connoting maps  directed  {into} $W$.  Conversely, sharp denotes a rise in tone, suggesting upward, outward motion, with $\sharp$ connoting maps leaving $W$.
\end{remark}

%To every $I$-indexed coproduct structure $\lk_i$ corresponds a unique $I$-indexed product $\lk_i^\sharp$ such that $\lk^\sharp_i \lk_j = \dk_{ij}$.  Dually, to every coproduct stru, and vice versa.  We refer to this primal-dual pair as a \emph{biproduct structure}.
\subsubsection{Examples}

\begin{example}  
Let $\uc$ be any finite family of $\field$-linear spaces.  For each $V \in \uc$, let $\pi_V$ denote the projection  $\oplus_\uc U \to V$ that vanishes on $\oplus_{U \neq V} U$, and let $\iota_V$ denote the inclusion $V \to \oplus_\uc U$.  The pair of $\uc$-indexed families $\{\pi, \ik\}$ is the \emph{canonical biproduct structure} on $\oplus_{\uc} U$ generated by $\vc$.
\end{example}

\begin{example}  
Suppose that the set $\uc$ in the preceding example is a complementary family of subspaces in some $\field$-linear space  $W$.  The structure defined by replacing $\oplus_\uc U$ with $W$ in the description of $\pi_V$ and $\iota_V$ is called the \emph{canonical  biproduct structure} on $W$ generated by $\vc$.
\end{example}

\begin{example}[Primal]  
\label{ex:comp} 
An arbitrary family of injective linear maps into $W$ %If $\lk$ is any family of {injective} maps, each with codomain $W$, then $\lk$
is a coproduct structure iff  $\{\I(f) : f \in \lk\}$ is a complementary family of subspaces in $W$.    %Dually, let $\lk$ be any family of \emph{surjective} linear maps each with domain $W$.  If $\cap_\lk K(f) = 0$ and $\cap_\uc K(f) \neq 0$ for every $\uc$ properly contained in $\lk$, then $\lk$ is a product structure.
\end{example}

\begin{example}[Dual]  
If $\lk$ is a family of {surjective} linear maps out of  $W$, then $\lk$ is a product structure iff the family of kernels $\kc = \{\K(f) : f \in \lk\}$ is \emph{cocomplementary}, meaning that  $\cap_\kc L = 0$ and $\cap_\uc L \neq 0$ when $\uc \subsetneq \kc$.
\end{example}

\begin{example}  \label{ex:canonical} Let $\lk$ be the unique index function $\lk : I \to \Hom(\field, \field^I)$ such that
$$
\lk_i(1) = \chi_i
$$
for each $i \in I$.  Then $\lk$ is a coproduct structure on $\field^I$.  The dual product structure is the function  $\lk^\sharp : I \to \Hom(\field^I, \field) $ such that
\begin{align*}
\lk_i^\sharp( w) = w_i
\end{align*}
for  $w \in \R^n$.  We call $\lk$ the \emph{canonical indexed coproduct structure} on $\field^I$.  The set $\{\lk_i : i \in I\}$ is the canonical  \emph{unindexed} coproduct structure.  By a slight abuse of language, we will use the term \emph{canonical coproduct structure} to refer to either of these constructions.
\end{example}

\section{Idempotents}

An endomorphism $e: W \to W$ is \emph{idempotent} if $e^2 = e$.  Equivalently, $e$ is an idempotent iff
\begin{align*}
e|_{\I(e)} = 1_{\I (e)} && e|_{\K(e)} = 0.
\end{align*}
Given complementary subobjects $K$ and $I$, we write $e^K_I$ for the idempotent with kernel $K$ and image $I$.  We refer to this morphism as \emph{projection onto $I$ along $K$}.

  If $\wk$ is a (co)product structure, then to each $f \in \wk$  we associate the endomorphism
  $$
  e_f = f^\flat f ^\sharp.
  $$
  Evidently, $e_f = e^{\K(f^\sharp)}_{\I(f^\flat)}$.  To emphasize dependence on $\wk$, we will sometimes write $e^\wk_f$ for $e_f$.   If $\wk = \{f,g\}$ is a coproduct structure, then $e^\wk_f = e^{\I(g)}_{\I(f)}$.  Similarly, if $\wk$ is a product structure then $e^\wk_f = e^{\K(g)}_{\K(f)}$.
  
  \begin{proposition}  If $\wk$ is a coproduct structure, then
  \begin{align}
  \sum_{f \in \wk} e_f = 1.  \label{eq:complementaryidempotents}
  \end{align}
  \end{proposition}
\begin{proof}
For convenience, assume  $\wk$ to be a product structure.  Since $\times \wk$ is an isomorphism, 
$$
(\times \wk) a = (\times \wk) b
$$
iff $a = b$.   Equivalently, $a = b$ iff $ga = gb$ for all $g\in \wk$.  Since $g  \left( \sum e_f \right ) = g = g  1$ for all $g \in \wk$,  the desired conclusion follows.
\end{proof}

The chief application of idempotents, for us, will be to generate new complements from old.   The instrument that executes such operations is Lemma \ref{lem:complementaryprojection}, which says that projection along a subspace $L$ yields a linear isomorphism between any two of its complements.  See Examples \ref{rb1} and \ref{rb2}  for  illustration.

\begin{lemma} \label{lem:complementaryprojection} Let  $\wk\op = \{f\op, g\op\}$ and $\wk = \{f, g\}$ be product and coproduct structures on $W\op$ and $W$, respectively. 
\begin{enumerate}
\item If  $h: \D(h) \to W$ and either of the two families
\begin{align*}
\{f,h\} && and && \{f,e^\wk_g h\}
\end{align*}
is a coproduct structure, then so is the other.
\item If  $h\op: W \op \to \D\op(h\op)$ and either of the two families
\begin{align*}
\{f\op,h\op\} && and && \{f\op,e^{\wk\op}_{g\op} h\op\}
\end{align*}
is a product structure, the so is the other.
\end{enumerate}
\end{lemma}

%\subsection{Examples}

\newcounter{rtaskno}
\DeclareRobustCommand{\rtask}[1]{%
   \refstepcounter{rtaskno}%
   \thertaskno\label{#1}}

\subsubsection{Example \rtask{rb1}: Complements  in $\R^2$}
\label{ex:complementsinr2}

Let $\lk = \{\lk_1, \lk_2\} \su \Hom(\R, \R^2)$ be the canonical unindexed coproduct structure on $\R^2$.  Recall from Example \ref{ex:canonical} that this means $\lk_1$, $\lk_2$ are the morphisms $\R \to \R^2$ such that
\begin{align*}
\lk_1(1) & = (1, 0)  && \lk_2(1) = (0,1).
\end{align*}
The idempotents  $e_{\lk_1}$ and $e_{\lk_2}$ are then the orthogonal projection maps onto the first and second coordinate axes, respectively.  

If $\uk = \{\uk_1, \uk_2\} \su \Hom(\R, \R^2)$ is defined by
\begin{align*}
\uk_1(1) & = (1, 0)  && \uk_2(1) = (1,1).
\end{align*}
then 
$$
e_{\lk_2} \uk_2 = \lk_2
$$
complements $\lk_1$.  Thus Lemma \ref{lem:complementaryprojection} confirms what was clear from the outset: that the family  $\uk = \{\lk_1, \uk_2\}$ is a coproduct structure.

\subsubsection{Example \rtask{rb2}: Complements in $\R^n$}
\label{ex:complementsinrn}

If $\lk = \{\lk_1, \lk_2\}$ is any coproduct structure on $\R^n$ then the  subspaces 
\begin{align*}
L_1 = \I(\lk_1)  && and && L_2 = \I(\lk_2)
\end{align*}
are complementary.  Suppose that $U_2$ is a  subspace of $\R^n$ complementary to $L_1$, and let $\uk_2$ denote the inclusion $U_2 \to \R^n$.     Since $\K(e_{\lk_2}) = \I(\lk_1)$ intersects $U_2$ trivially, one may deduce
$$
L_2 = e_{\lk_2} U_2
$$
by dimension counting.  Thus $\{\lk_1, e_{\lk_2} \uk_2\}$ is a coproduct structure.   Alternatively, one could have arrived at this conclusion by noting  that $\{\lk_1, \uk_2\}$ is a coproduct structure, and applying Lemma \ref{lem:complementaryprojection}.

\section{Arrays}

Let $T$ be any map $V \to W$.  We will frequently compose elements of (co)product structures with maps of this form, and in such cases musical symbols may be dropped without ambiguity.  For instance, if $\uk$ is a product structure on $V$ and $f \in \uk$, then one may write
$$
T f
$$
without ambiguity, since $Tf^\flat$ is well defined and $Tf ^ \sharp$ is not.  This convention has the beneficial consequence of  reducing unwieldy notation.

If $\ak$ is a family of maps out of $W$ and $\bk$ is a family of maps into $V$, the \emph{matrix representation} of $T$ with respect to $(\ak, \bk)$ is an $\ak \times \bk$ array, also denoted $T$,  defined by
$$
T(a,b) =  a T b.
$$
In keeping with convention, when $\ak$ and $\bk$ are (co)products, we write $T(\ak, \bk)$ for $T(\ak^\sharp, \bk^\flat)$.

There is a natural multiplication on matrix representations, which agrees with composition.  Recall that to every (co)product $\uk$ on the domain of $T$ corresponds a $\uk$-indexed family of idempotents $e_f = f^\flat f ^\sharp$ such that $\sum_\uk e_f = 1$.  Thus for any composable $T$ and $U$ one has
\[
(TU)(a,b) = \sum_{f } a  T e_f U b = \sum_{f } T(a,f)U(f,b).
\]
When $\lk$ and $\uk$ are composed of linear maps $\field \to \D\op(T)$ and $\field \to \D(T)$, respectively,  $T(\lk, \uk)$ is the standard matrix representation of $T$ with respect to bases $\{f(1) : f \in \lk\}$ and $\{g(1) : g \in \uk\}$.

\setcounter{rtaskno}{0}

\subsubsection{Example \rtask{rc1}: Endomorphisms on $\R^2$}
\label{sec:arraysr2}

Let $\lk = \{\lk_1, \lk_2\} \su \Hom(\R, \R^2)$ be the canonical unindexed coproduct structure on $\R^2$.  Recall from Example \ref{ex:canonical} that this means $\lk_1$, $\lk_2$ are the morphisms $\R \to \R^2$ such that
\begin{align*}
\lk_1(1) & = (1, 0)  && \lk_2(1) = (0,1).
\end{align*}
Let $\uk = \{\uk_1, \uk_2\} \su \Hom(\R, \R^2)$ be any other coproduct structure, and define $(u_{ij})$ by
\begin{align*}
\uk_1(1) & = (u_{11}, u_{21})  && \uk_{2}(1) = (u_{12},u_{22}).
\end{align*}
If  we identify $\Hom(\R, \R)$ with $ \R$ via $ f \leftrightarrow f(1)$, then  
$$
\begin{array}{c}
\\
\multirow{2}{*}{$1(\lk, \uk) \;\;\; \; = \; \;$ } \\
\\
\end{array}
\begin{array}{c|cc|}
\multicolumn{1}{c}{} &\;\; \uk_1 \;\;& \multicolumn{1}{c}{\;\;\uk_2\;\;}\\ \cline{2-3}
 \lk_1^\sharp &u_{11} & u_{12}  \\ 
\lk_2^\sharp& u_{21} & u_{22} \\\cline{2-3}
\end{array} \hspace{.5cm}
\vspace{.5cm}
$$
Likewise, if  $(u^{ij}) = (u_{ij}) \inv \in  GL_2(\R)$, then 
$$
\begin{array}{c}
\\
\multirow{2}{*}{$1(\lk, \uk) \;\;\; \; = \; \;$ } \\
\\
\end{array}
\begin{array}{c|cc|}
\multicolumn{1}{c}{} &\;\; \lk_1 \;\;& \multicolumn{1}{c}{\;\;\lk_2\;\;}\\ \cline{2-3}
 \uk_1^\sharp &u^{11} & u^{12}  \\ 
\uk_2^\sharp& u^{21} & u^{22} \\\cline{2-3}
\end{array} \hspace{.5cm}
\vspace{.5cm}
$$
Thus if $u^i = (u^{i 1}, u^{i 2})$, then 
$$
\uk_i^\sharp(w) = \nr{u^i, w},
$$
where $\nr{\; \cdot \;, \; \cdot \;}$ is the standard inner product on $\R^2$.

\subsubsection{Example \rtask{rc2}: Endomorphisms on $\R^n$}

The observations in Example \ref{rc1}  extend naturally to $\R^n$.  Let $\lk$ be the canonical unindexed coproduct structure on $\R^n$.  Let $\uk = \{\uk_1, \ld, \uk_n\} \su \Hom(\R, \R^n)$ be a second coproduct structure, and put 
$$
u_{q} =  \uk_q(1) = (u_{1q}, \ld, u_{nq}).
$$
If one defines $(u^{pq}) = (u_{pq})\inv $, then
\begin{align*}
1(\lk, \uk) = (u_{pq}) && and &&  1 (\uk ,\lk) = (u^{pq}).
\end{align*}
The righthand identity states $\uk_p^\sharp \lk_q = u^{pq}$ for all $p, q \in \{1 ,\ld, n\}$.  Since
$$
w = (w_1 \lk_1 + \cd+ w_n \lk_n)(1)
$$ for all $w \in \R^m$, it follows that 
$$
\uk_p^\sharp(w) = \nr{u^p, w}
$$
for all $p \in \{1, \ld, n\}$, where $u^p = (u^{p1}, \ld, u^{pn})$.
%
%\begin{align*}
%\left(
%\arraycolsep=9pt\def\arraystretch{.75}
%\begin{array}{ccc}
%- & u^1 & - \\
%- & u^2 & - \\
%& \vdots  & \\
%- & u^n & -
%\end{array}
%\right)
%&&
%\arraycolsep=4pt\def\arraystretch{.85}
%\left(
%\begin{array}{ccccc}
%\vspace{-.7cm} \\
%| & | & & | \\
%u_1 & u_2 & \cd  & u_n \\
%| & | \vspace{.23cm}& & | 
%\end{array}
%\right)
%\end{align*}

\begin{figure}
    \centering
    \begin{subfigure}{0.4\textwidth}
            \begin{align*}
            \left(
            \arraycolsep=9.25pt\def\arraystretch{1.55}
            \begin{array}{ccc}
            - & u^1 & - \\
            - & u^2 & - \\
            & \vdots  & \\
            - & u^n & -
            \end{array}
            \right)
            \end{align*}
%     \caption{}
     \label{fig:ul}
    \end{subfigure} 
    \hspace{1cm}
    \begin{subfigure}{0.4\textwidth}
            \begin{align*}
            \arraycolsep=4pt\def\arraystretch{1.85}
            \left(
            \begin{array}{ccccc}
            \vspace{-.7cm} \\
            | & | & & | \\
            u_1 & u_2 & \cd  & u_n \\
            | & | \vspace{.23cm}& & | 
            \end{array}
            \right)
            \end{align*}
%        \caption{}
        \label{fig:lu}
    \end{subfigure} \vspace{0cm}
    \caption{Arrays associated to the coproduct structures $\lk$ and $\uk$.     \emph{Left:} $1(\uk, \lk)$.  The  $p$th row of this array is the unique tuple $u^p$ such that  $\nr{u^p, w} = \uk_p^\sharp (w)$ for all $w \in \R^n$.  \emph{Right:} $1(\lk, \uk)$.  The $p$th column of this array is the tuple $u_p = \uk_p(1)$.   The matrix products $1(\uk, \lk) 1(\lk, \uk) = 1(\uk, \uk)$  and $1(\lk, \uk) 1(\uk, \lk) = 1(\lk, \lk)$ are   Dirac delta functions on $\uk \times \uk$ and $\lk \times \lk$, respectively.  }\label{fig:cobarrays}
\end{figure}

%\begin{align*}
%\left(
%\arraycolsep=7pt\def\arraystretch{.55}
%\begin{array}{ccc}
%- & u^1 & - \\
%- & u^2 & - \\
%& \vdots  & \\
%- & u^n & -
%\end{array}
%\right)
%&&
%\arraycolsep2pt\def\arraystretch{.8}
%\left(
%\begin{array}{ccccc}
%| & | & & | \\
%u_1 & u_2 & \cd  & u_n \\
%| & | & & | 
%\end{array}
%\right)
%\end{align*}

%$$
%\left(\begin{array}{ccccc}
%\vspace{-.7cm}
%\smash{\vrule width .5pt depth 42pt height 0pt}&
%\smash{\vrule width .5pt depth 42pt height 0pt}&&&\smash{\vrule width .5pt depth 42pt height 0pt}  \\ [-8pt]
%\\
%\\
%u_1& u_2&\multicolumn{2}{c}{\cd}&u_n\\
%\smash{\vrule width .5pt depth 20pt height 20pt}&
%\smash{\vrule width .5pt depth 20pt height 20pt}&&&\smash{\vrule width .5pt depth 20pt height 20pt}
%  \vspace{-.7cm}\\ [-8pt]
%\\
%\end{array}\right)
%$$

%\begin{align*}
%(u^{pq}) = (u_{pq})\inv && and && u^p = (u^{p1}, \ld, u^{pn}),
%\end{align*}
%then
%\begin{align*}
%1(\lk, \uk) = (u_{pq}), && 1(\uk, \lk) = (u^{pq}), && and &&
%\uk_p^\sharp(w) = \nr{u^p, w}.
%\end{align*}
%
%\begin{align*}
%1(\lk, \uk) = 
%\left(
%\begin{array}{cc}
%u_{11} & u_{12} \\
%u_{21} & u_{22}
%\end{array}
%\right).
%\end{align*}
%
%Suppose that $\lk = (\lk_0, \lk_1)$ and $\uk = (\uk_0, \uk_1)$ are ordered coproduct structures on $\R^2$.  
%

\section{Kernels}
\label{sec:kernels}

A morphism $k$ is \emph{kernel} to a morphism $f$ if the following  are equivalent for every  $g$ such that $fg$ is well defined:
\begin{enumerate}
\item $fg = 0$.
\item There exists exactly one morphism $h$ so that $kh = g$.  Equivalently, there exists exactly one  $h$ such that the following diagram commutes.
\end{enumerate} 
\vspace{-.4cm}

\[
\xymatrix@R=2.5pc@C=5.5pc@M=.7pc{
\ar[];[dr]_{k}   \ar@{.>}[rr];[]_h  \ar[rr];[dr]^g
\bullet && \bullet  \\
& \bullet
}
\]
\vspace{-.4cm}

\begin{lemma}  \label{lem:ker}  A morphism  $k$ is kernel to $f$ iff $k$ is a monomorphism and 
$$
\I(k) = \K(f).
$$
\end{lemma}

The notion of a kernel has a natural dual: we say that $k\op: \D\op(f\op) \to \D\op(k\op)$ is \emph{cokernel} to $f\op$ if the following are equivalent for every $g\op$ for which the composition $g \op f\op$ is  defined.
\begin{enumerate}
\item $g\op f\op = 0$.
\item There exists exactly one morphism $h\op$  such that the following diagram commutes
\end{enumerate}
\vspace{-.4cm}

\[
\xymatrix@R=2.5pc@C=5.5pc@M=.7pc{
\ar[dr];[]^{k\op}   \ar@{.>}[];[rr]^{h\op}  \ar[dr];[rr]_{g\op}
\bullet && \bullet  \\
& \bullet
}
\]
\vspace{-.5cm}

that is, such that $g\op = k\op h\op$.

\begin{lemma} \label{lem:coker} A morphism $k\op$ is cokernel to $f\op$ iff $k\op$ is an epimorphism and 
$$
\I(f\op) = \K(k\op) .
$$
\end{lemma}

Biproducts provide natural examples of kernel and cokernel maps.

\begin{lemma}  If $\{f,g\}$ is a coproduct structure, then
\begin{enumerate}
\item $g$ is kernel to $f^\sharp$.
\item $g^\sharp$ is cokernel to $f$.
\end{enumerate}
\end{lemma}
\begin{proof}  The first assertion follows from Lemma \ref{lem:ker}, since $g$ is a monomorphism  and $\I(g) = \K(f^\sharp)$.  The second assertion follows from Lemma \ref{lem:coker}, similarly.
\end{proof}

\subsection{The Splitting Lemma}
\label{subsec:splittinglemma}

The relationship between complements and kernel morphisms is essential to the current discussion.  This relationship, as formalized by Lemma \ref{lem:splitting} (the splitting lemma), represents a bridge between the matroid-theoretic notion of \emph{complementarity} (read: matroid duality) and the category-theoretic notion of a biproduct (read: linear duality).    

While the formulation of Lemma \ref{lem:splitting} is specially suited to the task at hand, several others are standard in mathematics.  Students of linear algebra know that
$$
V \cong \I(T) \oplus \K(T)
$$
for every linear $T: V \to W$.  Students of homological algebra know that  for any diagram 
$$
0 \lr{} A \lr{k}B\lr{f} C \lr{}  0
$$
where $k$ is kernel to $f$,  {every} right-inverse to $f$ is a complement to $k$.  This is known as the \emph{splitting lemma}.    Both results are implied by Lemma \ref{lem:splitting}.  The proof, being neither difficult nor specially enlightening, is omitted.

\begin{lemma}[Splitting Lemma]  
\label{lem:splitting}
Suppose that 
$$
f \op f = 1,
$$
where $f\op$ and $f$ are composable morphisms in a preadditive category.  
\begin{enumerate}
\item If $k$ is kernel to $f\op$, then $\{f,k\}$ is a coproduct structure.
\item If $k\op$ is cokernel to $f$, then $\{f\op, k\op\}$ is a product structure.
\end{enumerate}
\end{lemma}

\begin{remark}  Since $\K(h) = \K(\phi h)$ for any isomorphism $\phi$, the condition $f\op f = 1$ may be replaced in Lemma \ref{lem:splitting} by the condition that $f\op f$ be invertible.
\end{remark}

\subsection{The Exchange Lemma}

If  the purpose of the splitting lemma is to understand duality, then the purpose of duality, for us, is to understand \emph{exchange}.   A pair of product structures $\lk$ and $\wk$  are related by an \emph{exchange on} $(\ak, \bk)$ if
\begin{align}
\ak = \wk - \lk  && and &&  \bk = \lk - \wk.  \label{eq:exchangenondegeneracy}
\end{align}
An ordered tuple  $(\ak, \bk)$ is a \emph{(nontrivial) exchange pair for} $\wk$ if   \eqref{eq:exchangenondegeneracy} holds for some disctinct product structure $\lk \neq \wk$.    Exchanges and exchange pairs  for coproduct structures are defined similarly.  

If  $\ak = \{a\}$ and $\bk = \{b\}$ are singletons then the pair $(\{a\},\{b\})$ and the operation $\wk \mapsto \lk$ are  said to be \emph{elementary}.   To avoid proliferation of parentheses, we will often write 
$$
\wk - \ak \cup \bk
$$
for the product structure $\lk = (\wk - \ak) \cup \bk$ related to $\wk$ by \eqref{eq:exchangenondegeneracy}.

%
%Any pair $(\ak, \bk)$ for which there exists a
%
%We call  $\{ \ak, \bk\}$ the \emph{exchange pair} for $\{\wk,\lk\}$.  If, for reasons of legibility, one denotes $\wk - \ak \cup \bk = (\wk - \ak) \cup \bk$, then the operation 
%$$
%\wk \mapsto \lk
%$$
%is an \emph{exchange} on $(\ak, \bk)$.  

%To \emph{exchange} a pair of disjoint sets $\ak$ and $\bk$ in a product structure $\wk$ means to replace $\wk$ with a product structure $\lk$ such that 

%If we define 
%$\wk - \ak \cup \bk = (\wk - \ak) \cup \bk$, then to exchange $\ak$ for $\bk$ means replacing $\wk$ with $\wk - \ak \cup \bk$.
%\begin{align}
%\ak  = \wk -  (\wk - \ak \cup \bk)   && \bk = (\wk - \ak \cup \bk) - \wk. \label{eq:exchangenondegeneracy}
%\end{align}
%Note that this condition implies $\wk - \ak \cup \bk = (\wk \cup \bk ) - \ak$.   \emph{Coproduct} exchange is defined similarly.

%In practice the  conditions encoded  by \eqref{eq:exchangenondegeneracy} may be entirely ignored;\  we impose them only to avoid degenerate cases, e.g.\ where the sets $(\wk - \ak)$ and $\bk$ intersect nontrivially, hence where $\bk$ ``adds'' elements to $\wk$ that are already present.

Exchange operations are fundamental both to modern methods in matrix algebra and to matroid theory.  One of the simplest questions concerning this subject is the following: Given $\wk$,   when is $(\ak, \bk)$ an exchange pair?  This question reduces to determining when $(a,b)$ is an \emph{elementary} exchange pair, since, for example, if $\wk$ is a coproduct then $\wk - \ak \cup \bk$ is a coproduct iff 
$$
(\wk - \ak)  \cup \{\oplus \bk\}
$$
is a coproduct.   

The question of determining exchange pairs is therefore answered exhaustively by Lemma \ref{lem:exchange}, the \emph{Exchange Lemma}.  This result is mathematically equivalent to the splitting lemma, and retains much of its flavor.  Like the splitting lemma, it has several variations,  the Steinitz Exchange Lemma prominent among them.  This correspondence outlines the foundational overlap between homological algebra matroid theory.  As with the splitting lemma, the (not terribly difficult or instructive) proof is left as an exercise to the reader.

\begin{lemma}[Exchange Lemma] 
\label{lem:exchange}  
Let $\{f,g\}$ and $\{f\op, g\op\}$ be product and coproduct structures, respectively.  
\begin{enumerate}
\item An unordered pair $\{f, h\}$ is a product iff $g ^\sharp h$ is invertible.
\item An unordered pair $\{f\op, h\op \}$ is a coproduct iff $h\op \left ( g \op \right)^\flat$ is invertible.
\end{enumerate}
\end{lemma}

To build some familiarity with kernel maps, as well as the splitting and exchange lemmas, let us compute some examples.

%\newcounter{rtaskno}
%\DeclareRobustCommand{\rtask}[1]{%
%   \refstepcounter{rtaskno}%
%   \thertaskno\label{#1}}

\setcounter{rtaskno}{0}   

\subsubsection{Example \rtask{r1}: Kernels in $\R^2$}
\label{ex:rker}

 Let $\lk$ be the canoncial indexed coproduct structure on $\R^2$, and let 
$$
T: \R^2 \to \R^2
$$
be any endomorphism.  If $a_{ij} = \lk_i^\sharp T \lk_j$, then
$$
\begin{array}{c}
\\
\multirow{2}{*}{$T(\lk, \lk) \;\;\; \; = \; \;$ } \\
\\
\end{array}
\begin{array}{c|cc|}
\multicolumn{1}{c}{} &\;\; \lk_1 \;\;& \multicolumn{1}{c}{\;\;\lk_2\;\;}\\ \cline{2-3}
 \lk_1^\sharp &a_{11} & a_{12}  \\ 
\lk_2^\sharp& a_{21} & a_{22} \\\cline{2-3}
\end{array} \begin{array}{c}
\\
\multirow{2}{*}{.} \\
\\
\end{array}
 \hspace{.1cm}
\vspace{.5cm}
$$

Suppose that $a_{11} \neq 0$, and let $k$ be the map $\R \to \R^2$ such that
$$
\begin{array}{c}
\\
\multirow{2}{*}{$1(\lk, k) \;\;\; \; = \; \;$ } \\
\\
\end{array}
\begin{array}{c|c|}
\multicolumn{1}{c}{} & \multicolumn{1}{c}{\;\;k\;\;}\\ \cline{2-2}
 \lk_1^\sharp &-a_{11}\inv a_{12}   \\ 
\lk_2^\sharp& 1\\\cline{2-2}
\end{array} 
\begin{array}{c}
\\
\multirow{2}{*}{.} \\
\\
\end{array}
 \hspace{.1cm}
\vspace{.5cm}
$$
Since
$$
\begin{array}{c}
\\
\multirow{2}{*}{$T(\lk, \lk) 1(\lk, k) \;\;\; \; = \; \;$ } \\
\\
\end{array}
\begin{array}{c|c|}
\multicolumn{1}{c}{} & \multicolumn{1}{c}{\;\;k\;\;}\\ \cline{2-2}
 \lk_1^\sharp &0   \\ 
\lk_2^\sharp& a_{22} - a_{21}a_{11}\inv a_{12}\\\cline{2-2}
\end{array} \hspace{.5cm}
\vspace{.5cm}
$$
one has, in particular, that
$$
\lk_1^\sharp T k = 0.
$$
As the kernel of $\lk_1^\sharp T$ has dimension 1, it follows that
$$
\I(k) = \K(\lk_1^\sharp T).
$$
Since, in addition, $k$ is a monomorphism, we have that $k$ is  kernel to $\lk_1 ^\sharp T$.

\subsubsection{Example \rtask{r2}: Preadditive kernels}
\label{ex:gker}

The construction described in Example \ref{r1} is in fact very general.  Let $W$ and $W\op$ be any two objects with coproduct structures  $\lk = \{\lk_1, \lk_2\}$ and $\lk\op =(\lk_1\op, \lk_2\op)$, respectively, and fix
$$
T: W \to W\op.
$$
For convenience, let $a_{ij} = \lk_i \op T \lk_j$.

\begin{proposition}  \label{prop:kernelformula}  If $a_{11}$ is invertible, then in any preadditive category
$$
k = \lk_2 -     \lk_1 a_{11} \inv a_{12} 
$$
is kernel to $\lk_1\op T$.
\end{proposition}
\begin{proof}
The matrix representation of $T$ with respect to $\lk\op$ and $\lk$ is 
$$
\begin{array}{c}
\\
\multirow{2}{*}{$T(\lk\op, \lk) \;\;\; \; = \; \;$ } \\
\\
\end{array}
\begin{array}{c|cc|}
\multicolumn{1}{c}{} &\;\; \lk_1 \;\;& \multicolumn{1}{c}{\;\;\lk_2\;\;}\\ \cline{2-3}
 \lk_1\op &a_{11} & a_{12}  \\ 
\lk_2\op& a_{21} & a_{22} \\\cline{2-3}
\end{array} \hspace{.5cm}
\vspace{.7cm}
$$
by definition, and it is  simple to check that
$$
\begin{array}{c}
\\
\multirow{2}{*}{$1(\lk, k) \;\;\; \; = \; \;$ } \\
\\
\end{array}
\begin{array}{c|c|}
\multicolumn{1}{c}{} & \multicolumn{1}{c}{\;\;k\;\;}\\ \cline{2-2}
 \lk_1^\sharp &-a_{11}\inv a_{12}   \\ 
\lk_2^\sharp& 1\\\cline{2-2}
\end{array} 
\begin{array}{c}
\\
\multirow{2}{*}{.} \\
\\
\end{array}
 \hspace{.1cm}
\vspace{.5cm}
$$

As in Example \ref{r1}, we have
$$
\begin{array}{c}
\\
\multirow{2}{*}{$T(\lk\op, \lk) 1(\lk, k) \;\;\; \; = \; \;$ } \\
\\
\end{array}
\begin{array}{c|c|}
\multicolumn{1}{c}{} & \multicolumn{1}{c}{\;\;k\;\;}\\ \cline{2-2}
 \lk_1^\sharp &0   \\ 
\lk_2^\sharp& a_{22} - a_{21}a_{11}\inv a_{12}\\\cline{2-2}
\end{array} 
\begin{array}{c}
\\
\multirow{2}{*}{.} \\
\\
\end{array}
 \hspace{.1cm}
\vspace{.5cm}
$$
In particular, $\lk\op T k = 0$.  In the category finite-dimensional  $\field$-linear spaces and maps between them,  one can confirm that $k$ is  a kernel to $\lk\op T$ by checking
$$
\I(k) = \K(\lk_1\op T),
$$
via dimension counting.  To establish the result for  arbitrary preadditive categories, however, we will check the definition directly.

Recall that $e_f = f^\flat f^\sharp$ is the idempotent projection operator generated by a morphism $f$ in a (co)product structure $\wk$, and that $\sum_{f \in \wk} e_f = 1$.  If $h$ is any morphism such that  $0 = \lk_1\op T h$, then
\begin{align*}
0 =  \lk_1\op T (e_{\lk_1} + e_{\lk_2} ) h  = a_{11} \lk_1^\sharp h + a_{12} \lk_2^\sharp h,
\end{align*}
whence
$$
\lk_1^\sharp h = -a_{11}\inv a_{12} \lk_2^\sharp h,
$$
and therefore
$$
h = (\lk_2 - \lk_1 a_{11}\inv a_{12} ) \lk_2^\sharp h .
$$
To wit, the diagram

\[
\xymatrix@R=2.5pc@C=5.5pc@M=.7pc{
\ar[];[dr]_{k}   \ar@{.>}[rr];[]_{\lk_2^\sharp h}  \ar[rr];[dr]^h
\bullet && \bullet  \\
& \bullet
}
\]

commutes.  Therefore $h$ factors through $k$ when  $\lk_1\op Th = 0$.  Since $\lk_2^\sharp k = 1$, we have that $k$ is a monomorphism.  Thus the factorization is unique. The desired conclusion follows.
\end{proof}

\subsubsection{Example \rtask{r3}: Cokernels in $\R^2$}
\label{sec:cokerinr2}

As in Example \ref{r1}, let $\lk$ be the canonical unindexed coproduct structure on $\R^2$, and fix
$$
T: \R^2 \to \R^2
$$
If $a_{ij} = \lk_i^\sharp T \lk_j$, then
$$
\begin{array}{c}
\\
\multirow{2}{*}{$T(\lk, \lk) \;\;\; \; = \; \;$ } \\
\\
\end{array}
\begin{array}{c|cc|}
\multicolumn{1}{c}{} &\;\; \lk_1 \;\;& \multicolumn{1}{c}{\;\;\lk_2\;\;}\\ \cline{2-3}
 \lk_1^\sharp &a_{11} & a_{12}  \\ 
\lk_2^\sharp& a_{21} & a_{22} \\\cline{2-3}
\end{array} \begin{array}{c}
\\
\multirow{2}{*}{.} \\
\\
\end{array}
 \hspace{.1cm}
\vspace{.5cm}
$$

Suppose that $a_{11} \neq 0$, and let $k\op$ be the map $\R \to \R^2$ such that
$$
\begin{array}{c}
\\
\multirow{2}{*}{$1(k\op,\lk) \;\;\; \; = \; \;$ } \\
\\
\end{array}
\begin{array}{c|cc|}
\multicolumn{1}{c}{} & \lk_1& \multicolumn{1}{c}{\;\;\lk_2\;\;}\\ \cline{2-3}
k\op& -a_{21} a_{11}\inv & 1\\  \cline{2-3}
\end{array} 
\begin{array}{c}
\\
\multirow{2}{*}{.} \\
\\
\end{array}
 \hspace{.1cm}
\vspace{.5cm}
$$
Since
$$
\begin{array}{c}
\\
\multirow{2}{*}{$1(k\op,\lk)T(\lk,\lk) \;\;\; \; = \; \;$ } \\
\\
\end{array}
\begin{array}{c|cc|}
\multicolumn{1}{c}{} & \lk_1& \multicolumn{1}{c}{\;\;\lk_2\;\;}\\ \cline{2-3}
k\op& 0 & a_{22}-a_{21} a_{11}\inv a_{12}\\  \cline{2-3}
\end{array} 
\begin{array}{c}
\\
\multirow{2}{*}{.} \\
\\
\end{array}
 \hspace{.1cm}
\vspace{.5cm}
$$
one has, in particular, that $k\op T \lk_1 =0$.  As the image of $T\lk_1$ has dimension 1, it follows that
$$
\I(T \lk_1)  = \K(k\op).
$$
Since, in addition, $k$ is an epimorphism, we have that $k\op$ is a cokernel to $T \lk_1$.

\subsubsection{Example \rtask{r4}: Preadditive cokernels}
\label{ex:gcoker}

Like the kernel in Example \ref{r1},  the cokernel constructed in Example \ref{r3}  has a natural analog for arbitrary two-element corpoducts.  Let $W$ and $W\op$ be any two objects with coproduct structures  $\lk = \{\lk_1, \lk_2\}$ and $\lk\op =(\lk_1\op, \lk_2\op)$, respectively, and fix
$$
T: W \to W\op.
$$
As before, set $a_{ij} = \lk_i \op T \lk_j$.

The proof of Proposition \ref{prop:cokernelformula} entirely dual to that of Proposition \ref{prop:kernelformula}.  The details are left as an exercise to the reader.

\begin{proposition} \label{prop:cokernelformula}   If $a_{11}$ is invertible, then in any preadditive category
$$
k\op = \lk_2\op -     a_{11} \inv a_{21} \lk_1\op
$$
is cokernel to $T\lk_1$.
\end{proposition}

\subsection{Idempotents}

The maps $k$ and $k\op$ described in the preceding sections bear a simple description in terms of idempotents.  For a first example, take morphism $k$.  By design, the image of $k$ is the kernel of $\lk\op_1 T$.  There is a canonical projection operator that acts as identity on this subspace, and that vanishes on the image of $\lk_1$.  We will show (Proposition \ref{prop:idemker}) that  $k$ is merely the composition of $\lk_2$ with this projection.  

Dually, $k\op$ is a morphism that vanishes on the image of $T \lk_1$.  There is canonical projection  that vanishes on this subspace, and that acts as identity on the kernel of $\lk_1\op$.  We claim that $k\op$ is merely the composition of $\lk_2\op$ with this operator.   Proposition \ref{prop:idemker} formulates these assertions symbolically. 

%Informally, $k$ is the projection of $\lk_2$ onto $\K(\lk_1\op T)$ along $\I(\lk_1)$, and $k\op$ is projection onto $\K(\lk_1\op)$ along $\I(T \lk_1)$.  Formally, we have the following.

\begin{proposition} \label{prop:idemker}  Let  $k$ and $k\op$ be the kernel and cokernel maps constructed in Examples \ref{r2} and \ref{r4}.  Then
\begin{align*}
k = e^{\I(\lk_1)}_{\K(\lk_1\op T)} \lk_2 && and && k\op = \lk_2\op e_{\K(\lk_1\op)}^{\I(T \lk_1)} .
\end{align*}
\end{proposition}
\begin{proof}
We will argue the first identity;  the proof of the second is essentially dual.  From Example \ref{r2} we have that 
$$
\begin{array}{c}
\\
\multirow{2}{*}{$1(\lk, \{\lk_1, k\}) \;\;\; \; = \; \;$ } \\
\\
\end{array}
\begin{array}{c|cc|}
\multicolumn{1}{c}{} &\;\; \lk_1 \;\;& \multicolumn{1}{c}{\;\;k\;\;}\\ \cline{2-3}
 \lk_1^\sharp &1 & -a_{11}\inv a_{12}  \\ 
\lk_2^\sharp& 0 & 1 \\\cline{2-3}
\end{array} 
\begin{array}{c}
\\
\multirow{2}{*}{.} \\
\\
\end{array}
\hspace{.5cm}
\vspace{.7cm}
$$
Since $1(\lk, \{\lk_1, k\}) 1(\{\lk_1, k\},\lk) = 1(\lk, \lk)$, it follows that
$$
\begin{array}{c}
\\
\multirow{2}{*}{$1( \{\lk_1, k\}, \lk) \;\;\; \; = \; \;$ } \\
\\
\end{array}
\begin{array}{c|cc|}
\multicolumn{1}{c}{} &\;\; \lk_1 \;\;& \multicolumn{1}{c}{\;\;\lk_2\;\;}\\ \cline{2-3}
 \lk_1^{\sharp\sharp} &1 & a_{11}\inv a_{12}  \\ 
k^{\sharp\sharp}& 0 & 1 \\\cline{2-3}
\end{array} 
\begin{array}{c}
\\
\multirow{2}{*}{} \\
\\
\end{array}
\hspace{.6cm}
\vspace{.7cm}
$$
where $\sharp\sharp$ is the sharp operator on $\{\lk_1, k\}$, as distinguished from that of $\lk$.  
%If for economy we write $e$ for
%$$
%e^{\I(\lk_1)}_{\K(\lk_1\op T)}  = e^{\{\lk_1, k\}}_k = k k^{\sharp \sharp}
%$$
%then 
Thus
$$
\begin{array}{c}
\\
\multirow{2}{*}{$(kk^{\sharp \sharp})(\lk, \lk) \;\;\; \; = \; \;$ } \\
\\
\end{array}
\begin{array}{c|cc|}
\multicolumn{1}{c}{} &\;\; \lk_1 \;\;& \multicolumn{1}{c}{\;\;\lk_2\;\;}\\ \cline{2-3}
 \lk_1^{\sharp} &0 & a_{11}\inv a_{12}  \\ 
\lk_2^{\sharp}& 0 & 1 \\\cline{2-3}
\end{array} 
\begin{array}{c}
\\
\multirow{2}{*}{.} \\
\\
\end{array}
\hspace{.5cm}
\vspace{.7cm}
$$
In particular we have 
$$
(k k^{\sharp \sharp}) \lk_2 = \lk_2 - \lk_1 a_{11}\inv a_{12} = k.
$$
Since $kk^{\sharp \sharp}  = e^{ \{\lk_1, k\}}_k = e^{\I(\lk_1)}_{\K(\lk_1\op T)}$, the desired conclusion follows.
\end{proof}

\subsection{Exchange}
\label{subsec:exchange}

Let us consider the dual structures $\lk^\sharp$ and  $\left (\lk\op \right ) ^\flat$.   Since by hypothesis
$$
(\lk_1\op T) \lk_1 = a_{11} = \lk_1\op ( T \lk_1)
$$
is an isomorphism, the exchange lemma provides that
\begin{align*}
\uk = (\lk_1\op T, \lk_2^\sharp) && and && \uk\op = (T \lk_1,\lk_2^{op \; \flat})
\end{align*}
are ordered product and coproduct structures, respectively.   Since $k$ was designed as a kernel to $\lk_1\op T$ and $k\op$ was designed as a kernel to $T \lk_1$, we naturally have
\begin{align*}
\uk_1 k = 0 && and &&  k\op \uk\op_1 = 0.
\end{align*}
Since moreover
\begin{align*}
\uk_2 k  = \lk_2^\sharp \left( \lk_2\ -     \lk_1 a_{11} \inv a_{12} \right)   = 1  &&
k\op \uk_2\op  = ( \lk_2\op -     a_{11} \inv a_{21} \lk_1\op)    \lk_2^{op \;\flat} = 1
\end{align*}
if follows that 
\begin{align*}
k = \uk_2^\flat && and &&k\op = \uk_2^{op \; \sharp}.
\end{align*}
To wit, $k$ and $k\op$ are  \emph{dual morphisms} in the biproduct structures generated by
\begin{align*}
\lk^\sharp - \{\lk_1^\sharp\} \cup \{\lk_1 \op T\} && and &&
 \lk^{op \; \flat} - \{\lk_1^{op \; \flat} \} \cup \{T \lk_1\},
\end{align*}
respectively.

%
%.  If  $k$ is  kernel to $\lk_1\op T$, then  $\{\lk_1, k\}$ is a coproduct structure on $W$, by the splitting lemma.   Dually, if $k\op$ is  cokernel to $T \lk_1$, then $\{\lk_1\op, k\op\}$ is a product structure on $W\op$.   We claim that
%\begin{align*}
%k = \lk_2 -     \lk_1 a_{11} \inv a_{12} && and && k\op = \lk_2\op - a_{21} a_{11} \inv \lk_1\op
%\end{align*}
%are two such morphisms. 
%
%
% It is simple to check that 
%\begin{align*}
%\lk_1 \op T (\lk_1 a_{11} \inv a_{12}) =a_{12} && and && k\op = \lk_2\op - a_{21} a_{11} \inv \lk_1\op
%\end{align*}
%
%\begin{align*}
%k = \lk_2 - ( \lk_1 a_{11}\inv \lk_1\op) T\lk_2&& and && k\op  = \lk_2\op - \lk_2\op T (\lk_1 a_{11}\inv \lk_1\op)
%\end{align*}
%are two such morphisms.  Indeed, since
%\begin{align*}
%(\lk_1\op T) (\lk_1 a_{11}\inv \lk_1\op) =  \lk_1\op && and && (\lk_1 a_{11}\inv \lk_1\op)(T \lk_1) = \lk_1
%\end{align*}
%one has  
%$$
%\lk_1\op T  k = 0 = k\op T \lk_1 
%$$
%

\section{The Schur Complement}

Let $A$ be an array with block structure
$$
A = 
\left( 
\begin{array}{cc}
a_{11} & a_{12}\\
a_{21} & a_{22}
\end{array}
\right).
$$
We allow  entries $a_{ij}$ to be either block submatrices with coefficients in a ground field $\field$, or maps of form $\lk_i T \uk_j$, where $\lk$ and $\uk$ are (co)product structures on the (co)domain of a morphism $T$.   

If $a_{11}$ is invertible, then the \emph{column clearing operation} on $a_{11}$ is the operation $A \mapsto  AU$, where
$$
U = 
\left( 
\begin{array}{cc}
1 & -a_{11}\inv a_{12}\\
0 & 1
\end{array}
\right).
$$
Similarly, a \emph{column pivot} on $a_{11}$ is an operation of form $A \mapsto A \tilde U$, where
$$
\tilde U = 
\left( 
\begin{array}{cc}
a_{11}\inv & -a_{11}\inv a_{12}\\
0 & 1
\end{array}
\right).
$$
Dually, the \emph{row clearing operation} on $a_{11}$ is the operation  $A \mapsto L A$, where
$$
L = 
\left( 
\begin{array}{cc}
1 & 0\\
-a_{21}a_{11}\inv & 1
\end{array}
\right)
$$
and the \emph{row pivot} is an operation of form $A \mapsto \tilde L A$, where
$$
\tilde L = 
\left( 
\begin{array}{cc}
a_{11}\inv & 0\\
-a_{21}a_{11}\inv & 1
\end{array}
\right).
$$
We group row pivots and row clearing operations under the common heading of a \emph{row operation}.  Likewise, we refer to both column clearing operations and column pivots as \emph{column operations}.

The name ``clearing operation'' finds motivation in the following  matrix identities, where  $*$ serves as a placeholder for  arbitrary block submatrices.
\begin{align*}
AU = 
\left( 
\begin{array}{cc}
a_{11} & 0\\
a_{21} & a_{22} - a_{21}a_{11}\inv a_{12}
\end{array}
\right) 
			\hspace{.5cm}
A \tilde U = 
\left( 
\begin{array}{cc}
1 & 0\\
* & a_{22} - a_{21}a_{11}\inv a_{12}
\end{array}
\right)  
\\
\\
LA = 
\left( 
\begin{array}{cc}
a_{11} & a_{12}\\
0 & a_{22} - a_{21}a_{11}\inv a_{12}
\end{array}
\right) 
			\hspace{.5cm}
\tilde L A   = 
\left( 
\begin{array}{cc}
1 & *\\
0 & a_{22} - a_{21}a_{11}\inv a_{12}
\end{array}
\right) 
\end{align*}

It is clear that each row (respectively, column) operation on $a_{11}$ produces an array with an invertible block in position $(1,1)$.  Thus any finite sequence of row and column operations may be performed on the upper-lefthand block of $A$.

If this sequence includes at least one row, one column, and one pivot operation, then the resulting array will have form 
\[
\left( 
\begin{array}{cc}
1 & 0\\
0 & a_{22} - a_{21}a_{11}\inv a_{12}
\end{array}
\right) .
\]
If the sequence includes at least one row and one column operation but no pivots, then the resulting array will have form
\[
\left( 
\begin{array}{cc}
a_{11} & 0\\
0 & a_{22} - a_{21}a_{11}\inv a_{12}
\end{array}
\right) .
\]

The  submatrix that appears in all six of the preceding arrays, denoted
$$
\sk = a_{22} - a_{21} a_{11}\inv a_{12}
$$
has a special name in matrix algebra: the \emph{Schur complement}.

The Schur complement is an object of fundamental  importance in mathematics.  It is basic to modern algebra, analysis, geometry, probability, graph theory, combinatorics, and optimization, and plays a commensurate role in most branches of of the sciences and engineering.  
Specific subfields  touched by the Schur complement include functional analysis (via Fredholm operators), conditional independence, matroid theory, geometric bundle theory, convex duality, and semidefinite programing.  For historical details see the excellent introductory text \emph{The Schur Complement and its Applications}, by Zhang \  \cite{zhang2006schur}. Of particular interest to our story, the Schur complement encapsulates LU factorization, hence to a tremendous body of work in computational linear algebra.  

We have  seen that the Schur complement appears  as the result of row and column operations.  Let us give several new characterizations.  As in Examples \ref{r2} and \ref{r4} of \S\ref{sec:kernels}, let 
\begin{align*}
\lk = \{\lk_1, \lk_2\} &&and &&\lk\op =(\lk_1\op, \lk_2\op),
\end{align*}
be coproduct structures on objects $W$ and $W^{op}$, respectively, and fix
$$
T: W \to W\op.
$$
Let $k$ be the kernel to $\lk_1\op T$ obtained by composing $\lk_2$ with projection onto the nullspace along the image of $\lk_1$.  In symbols,
$$
k = e^{\I(\lk_1)}_{\K(\lk_1\op T)} \lk_2.
$$
Dually, let $k\op$ be the cokernel to $T \lk_1$ obtained by precomposing $\lk_2\op$ with projection onto the nullspace of $\lk_1\op$ along the image.  In symbols,
$$
k\op = \lk_2\op e_{\K(\lk_1\op)}^{\I(T \lk_1)} 
$$
We have seen (Propositions  \ref{prop:kernelformula}, \ref{prop:cokernelformula}, and \ref{prop:idemker}) that
\begin{align}
k = \lk_2 - \lk_1 a_{12} a_{11}\inv  && and && k\op  = \lk_2\op - \lk_1\op a_{11}\inv a_{21}.   \label{eq:kerform}
\end{align}
%We showed in Propositions \ref{prop:kernelformula} and  \ref{prop:cokernelformula} that $k$ and $k\op$ are kernel and cokernel to $\lk_1\op T$ and $T\lk_1$, respectively.  Although \eqref{eq:kerform} appears rather opaque, recall from  Proposition \ref{prop:idemker} the maps themselves are described quite simply:  $k$, for example, is the ``projection'' of $\lk_2$  onto the kernel of $\lk_1\op T$.   
%\begin{align*}
%k = e^{\I(\lk_1)}_{\K(\lk_1\op T)} \lk_2 && and && k\op = \lk_2\op e_{\K(\lk_1\op)}^{\I(T \lk_1)} .
%\end{align*}
We have also noted that 
\begin{align*}
\begin{array}{c}
\\
\multirow{2}{*}{$1(\{\lk_1\op,k\op\},\lk) \; \; =  $} \\
\\
\end{array} &
\begin{array}{c|cc|}
\multicolumn{1}{c}{} & \lk_1& \multicolumn{1}{c}{\;\;\lk_2\;\;}\\ \cline{2-3}
\lk_1\op & 1 & 0 \\
k\op& -a_{21} a_{11}\inv & 1\\  \cline{2-3}
\end{array} 
\begin{array}{c}
\\
\multirow{2}{*}{ } \\
\\
\end{array}
\\
\end{align*}
and
\begin{align*}
\\
\begin{array}{c}
\\
\multirow{2}{*}{$1(\lk, \{\lk_1, k\}) \;\; = $ } \\
\\
\end{array}&
\begin{array}{c|cc|}
\multicolumn{1}{c}{} &\;\; \lk_1 \;\;& \multicolumn{1}{c}{\;\;k\;\;}\\ \cline{2-3}
 \lk_1^\sharp &1 & -a_{11}\inv a_{12}  \\ 
\lk_2^\sharp& 0 & 1 \\\cline{2-3}
\end{array} 
\begin{array}{c}
\\
\multirow{2}{*}{.} \\
\\
\end{array}
\end{align*}
whence
$$
\begin{array}{c}
\\
\multirow{2}{*}{$T(\{\lk\op_1, k\op\}, \{\lk_1, k\}) \;\;\; \; = \; \;$ } \\
\\
\end{array}
\begin{array}{c|cc|}
\multicolumn{1}{c}{} &\;\; \lk_1 \;\;& \multicolumn{1}{c}{\;\;\lk_2\;\;}\\ \cline{2-3}
 \lk_1\op &a_{11} & 0  \\ 
k\op& 0 & \sk \\\cline{2-3}
\end{array} 
\begin{array}{c}
\\
\multirow{2}{*}{.} \\
\\
\end{array}
\hspace{.5cm}
\vspace{.7cm}
$$
and
\begin{align*}
\begin{array}{c}
\\
\multirow{2}{*}{$T(\lk\op, \{\lk_1, k\}) \;\;\; \; = \; \;$ } \\
\\
\end{array} &
\begin{array}{c|cc|}
\multicolumn{1}{c}{} &\;\; \lk_1 \;\;& \multicolumn{1}{c}{\;\;k\;\;}\\ \cline{2-3}
 \lk_1\op &a_{11} & 0  \\ 
\lk_2\op& a_{21} & \sk \\\cline{2-3}
\end{array} 
\begin{array}{c}
\\
\multirow{2}{*}{,} \\
\\
\end{array}
\\
\\
\begin{array}{c}
\\
\multirow{2}{*}{$T(\{\lk\op_1, k\op\}, \lk) \;\;\; \; = \; \;$ } \\
\\
\end{array}&
\begin{array}{c|cc|}
\multicolumn{1}{c}{} &\;\; \lk_1 \;\;& \multicolumn{1}{c}{\;\;\lk_2\;\;}\\ \cline{2-3}
 \lk_1\op &a_{11} & a_{12}  \\ 
k\op& 0 & \sk \\\cline{2-3}
\end{array} 
\begin{array}{c}
\\
\multirow{2}{*}{.} \\
\\
\end{array}
\hspace{.5cm}
\vspace{.7cm}
\end{align*}

Thus $\sk$ may be characterized as any of 
\begin{align*}
k\op T \lk_2 && k\op T k && \lk_2 T k
\end{align*}
where $k$ is 
\begin{align*}
 \lk_2 - \lk_1 a_{12} a_{11}\inv, && e^{\I(\lk_1)}_{\K(\lk_1\op T)} \lk_2, 
\end{align*}
or the dual to $\lk_2^\sharp$ in $\{\lk_1\op T, \lk_2^\sharp\}^\flat$, and 
 $k\op$ is
\begin{align*}
 \lk_2\op - \lk_1\op a_{11}\inv a_{21}, &&\lk_2\op e_{\K(\lk_1\op)}^{\I(T \lk_1)},
 \end{align*}
or the dual to $\lk_2^{op \;\flat}$ in $\{T \lk_1, \lk_2^{\op \; \flat}\}^\sharp$.

\subsection{Diagramatic Complements}

The following characterization of $\sk$ will not enter our later discussion directly, however it is of basic algebraic interest.  Some knowledge of the language of category theory is assumed.

Let $D$ be the category on two objects and three morphisms, exactly one of which is not an endomorphism.  For each preadditive category ${\mathcal C}$, let $[D,{\mathcal C}]$ be the preadditive category of diagrams $D \to {\mathcal C}$.  

If
\begin{align*}
\uk\op = (\lk_1\op, k\op) && and && \uk = (\lk_1,k),
\end{align*}
then it may be checked directly that the diagrams

\begin{align*}
\xymatrix@R=3pc@C=3pc@M=0.5pc{
\ar[rr];[]_{a_{11}}  \ar[d];[]^{\uk^{op}_1}  \ar[drr];[rr]_{\uk_1^\sharp}  \ar[drr];[d]^T
\bullet & & \bullet \\
\bullet & & \bullet \\
}
&&
\xymatrix@R=3pc@C=3pc@M=0.5pc{
\ar[rr];[]_{\sk}  \ar[d];[]^{\uk_2^{op}}  \ar[drr];[rr]_{\uk_2^\sharp}  \ar[drr];[d]^T
\bullet & & \bullet \\
\bullet & & \bullet \\
}
\end{align*}

\noindent determine a product structure on the the diagram $D \to {\mathcal C}$ that sends the nonidentity morphism to $T$.  
The dual coproduct structure is given by

\begin{align*}
\xymatrix@R=3pc@C=3pc@M=0.5pc{
\ar[rr];[]_{a_{11}}  \ar[];[d]_{\uk^{op\; \flat}_1}  \ar[rr];[drr]^{\uk_1}  \ar[drr];[d]^T
\bullet & & \bullet \\
\bullet & & \bullet \\
}
&&
\xymatrix@R=3pc@C=3pc@M=0.5pc{
\ar[rr];[]_{\sk}  \ar[];[d]_{\uk^{op\; \flat}_2}  \ar[rr];[drr]^{\uk_2}  \ar[drr];[d]^T
\bullet & & \bullet \\
\bullet & & \bullet \\
}
\end{align*}

\begin{remark}  While there has been at least one category-theoretic treatment of the subject \cite{smith1983schur} the interpretation of a Schur complement as a literal complement to $\ak$ in $T$ in the category $[D,{\mathcal C}]$ has, to our knowledge, gone unremarked in the published literature.   
\end{remark}

%k = e^{\I(\lk_1)}_{\K(\lk_1\op T)} \lk_2 && and && k\op = \lk_2\op e_{\K(\lk_1\op)}^{\I(T \lk_1)} .

\section{M\"obius Inversion}
\label{sec:mobius}

Let $A$ be an array of form $T(\wk, \wk)$, where $T$ is an endomorphism and $\wk$ is a (co)product structure on the domain of $T$.   We write $\pc^m$ for the set of all sequences of form
\begin{align*}
p : \{0,\ld, m\} \to \wk  %\label{eq:path}
\end{align*}
and $\pc$ for $\cup_{m \ge 0} \pc^m$.  Given any family of sequences ${\mathcal Q}$, we write 
$$
{\mathcal Q}(i,j) = \{q \in {\mathcal Q} : q(0) = i, \; q(\ell(q)) = j\}
$$
where  $\ell(q) = \max( \D(q))$ is the last integer on which $q$ is defined.

We  define $A(p) = 1$ when $m = 0$ and
\[
A(p) =  A(p_{0}, p_{1}) A(p_1, p_2) \cd A(p_{m-1}, p_m)
\]
when $m > 0$.  Consequently
\begin{align}
A^{m}(i,j) = \sum_{\pc^m(i,j)} A(p)  
\label{eq:pathproduct}
\end{align}
for  nonnegative $m$.  
\begin{remark}
Identity \eqref{eq:pathproduct} is vacuous when $m \in \{0,1\}$.   It is the  {definition} of matrix multiplication when $m = 2$.
\end{remark} 

\begin{remark}
Identity \eqref{eq:pathproduct}  remains valid if one replaces $\pc^m(i,j)$ with
$$
\pc_A^m(i,j) = \{p \in \pc^m(i,j) : A(p) \neq 0\}.
$$
%This set includes into the family of all  $p$ for which 
%$$
%\{ (p(k),p(k+1)) : 0 \le k < m \} \su \Supp(A).
%$$
%Thus
%$$
%\Supp(A^m) \su \cl_{TR} \left ( \Supp(A) \right)
%$$
%where by definition
%$$
%\cl_{TR}(R) = \bigcup_{p = 1}^{\infty} \left (R \cup \ad (I \times I) \right ) ^p
%$$
%is the transitive reflexive closure of any relation $R \su I \times I$.   
%
\end{remark}  %In the special case  where and $A$ is (strictly) upper-triangular and $I = \tb n$, the set $\pc_A$ includes into the set of (strictly) monotonically decreasing paths $i$ to $j$.  Likewise, when $A$ is lower-triangular  $\pc_A$ includes into the set of monotonically increasing paths.

Given any pair of binary relations $\sim_I$ and $\sim_J$ on $I$ and $J$, respectively, let us say that a map $p: I \to J$  \emph{increases monotonically}  if 
$$
i \sim_I j \implies p(i) \sim_J p(j).
$$
The following is a conceptually helpful characteristic of $\pc_A^m$.

\begin{lemma} If
$$
p: \{0,\ld, m\} \to I
$$
and $A(p) \neq 0$, then $p$ increases monotonically with respect to the transitive closure of $\Supp(A)$ and the canonical order on $\{0,\ld, m\}$.
\end{lemma}

The main result of this section is a generalization of the classical \emph{M\"obius inversion formula} of number theory and combinatorics \cite{rota1964}.  Lemma \ref{lem:mobiusunit} expresses one part of this extension.

%\begin{lemma}
%If ${\mathcal Q}^m$ is the set of all paths 
%$
%p: \{0, \ld, m\} \to I
%$
%that increase monontonically with respect to the canonical order on $\{0, \ld, m\}$ and the transitive closure of $\Supp(A)$, then
%$$
%\pc^m_A \su {\mathcal Q}^m.
%$$.  
%\end{lemma}

%\begin{corollary}  If $I = \{1, \ld, n\}$ and $A$ is upper-triangular, then every element of $\pc_A$ is a finite monotone increasing sequence of integers.
%\end{corollary}

%\begin{lemma}  Let $\sim_m$ denote the canonical order on $\{1, \ld, m\}$ and $\sim_I$ denote the transitive closure of $\Supp(A)$.  If $[\sim_m, \sim_I]$ is the family of all monotone paths $\{0,\ld ,m\} \to I$, then
%$$
%\pc_A^m \su [\sim_m, \sim_I].
%$$
%If the nonzero entries of $A$ are isomorphisms, then reverse inclusion holds, also.
%\end{lemma}

%\begin{lemma} \label{lem:moninc} If
%$A(p) \neq 0$ then $p$ increases monotonically with respect to the canonical order on $\Z$ and the transitive closure of $\Supp(A)$. 
%\end{lemma}
%
%
%
%\begin{lemma}
%If the nonzero entries of $A$ are isomorphisms, then $A(p) \neq 0$ iff
%$$
%p: \{0, \ld, m\} \to I
%$$
%increases monotonically with respect to the canonical order on $\Z$ and the transitive closure of $\Supp(A)$.
%\end{lemma}
%
%

\begin{lemma} 
\label{lem:mobiusunit} 
Suppose $A = T(\wk,\wk)$, where $T$ is an endomorphism in a preadditive category and $\wk$ is a finite (co)product structure on  $\D(T)$.  If $A = 1 - t$ and $t$ has strictly acyclic support, then
\begin{align}
A\inv (i,j) = \sum_{\pc_t(i,j)} t(p)  \label{eq:strictlyunitalmobiusinversion}
\end{align}
for all $i,j \in \wk$.
\end{lemma}
\begin{proof}
If $A$ has $N$ rows and columns, then $t^{N} = 0$, by Lemma \ref{lem:acyclic}.   Thus  $(1-t) \inv =   \sum_{k = 0}^{\infty} t^k$.  The desired conclusion follows from \eqref{eq:pathproduct}.
\end{proof}

\begin{remark}  
In the special case where $A = \chi_{\preceq}$ for some partial order $\preceq$ on $I$, identity \eqref{eq:strictlyunitalmobiusinversion} recovers the classical {M\"obius inversion formula}.
\end{remark}

%Lemma \ref{lem:mobiusunit} admits two extensions  First, the proof remains correct for any array  $A = T(\wk, \wk)$ that may be expressed in form $1 - t$, where $t$ is acyclic and $\wk$ is a (co)product structure on the (co)domain of $T$.

 % For the reader unfamiliar with the language of categories, the phrase \emph{automorphism in a preadditive category} may be replaced by \emph{linear isomorphism of form $T: W \to W$}.
We are now prepared to state and prove the main result.  For convenience, let 
\[
A[p] =   A(p_{0}, p_{0})\inv A(p_{0}, p_1) A(p_{1}, p_{1})\inv\cd A(p_m,p_m)\inv.
\]
%$$A[p] =\frac{\prod_{k=1}^m A(p_{k-1}, p_k)}{\prod_{k=0}^{m-1} A(p_k, p_k)}.$$  % \frac{A(p_{k-1}, p_k)} {A(p_{k-1}, p_{k-1})}$.
%$\prod_{k=1}^m  \frac{A(p_{k-1}, p_k)} {A(p_{k-1}, p_{k-1})}$.
Under the hypotheses of Lemma \ref{lem:mobiusunit}, one has 
$$
t(p) =  (-1)^{\ell(p)+1} A[p],
$$ 
so Theorem \ref{thm:mobiusgeneral} is a strictly more general.

\begin{theorem}  
\label{thm:mobiusgeneral}  
Let $A = T(\wk,\wk)$, where $T$ is an automorphism in a preadditive category and $\wk$ is a finite (co)product structure on  $\D(T)$. If $A$ has acyclic support, then
\[
A\inv(i,j) = \sum_{\pc_t(i,j)}   (-1)^{\ell(p)+1}A[p] ,
\]
where $t$ is the strictly-triangular part of $A$.
\end{theorem}
\begin{proof}  
Let $d$ and $t$ be the unique diagonal and off-diagonal arrays, respectively, such that  $A = d - t$.   If $B = d \inv A $ and $s = d \inv t$, then by definition 
$$
B(p_{k-1},p_k) = A(p_{k-1},p_{k-1})\inv A(p_{k-1},p_k).
$$  
The terms on the righthand side of $(A\inv d )(i,j)= B\inv (i,j) = \sum_{\pc_t(i,j)} s(p)$ are of the form
\[
(-1)^{\ell(p)-1}    A(p_{0}, p_{0})\inv A(p_{0}, p_1) \cd   A(p_{m-1},p_{m-1})\inv    A(p_{m-1}, p_m).
%(-1)^{\ell(p)-1}\prod_{k=1}^{\ell(p)-1} \frac{A(p_{k-1}, p_k)}{A(p_{k-1}, p_{k-1})}.
\]
Right-multiplication with $d \inv$ yields the desired result.
\end{proof}

\begin{example}
Let $A$ be the following matrix $4 \times 4$ matrix with coefficients in $\R$,
\vspace{.5cm}

$$
\left(
\arraycolsep=10pt\def\arraystretch{1}
\begin{array}{cccc}
1&0 & 0 & 0 \\
0& 1 & 0 & 0\\
1& 1& 1 & 0\\
1 & 1& 1& 1
\end{array}
\right).\vspace{.5cm}
$$
One has
$$
\pc_A(1,4) = \{(1,4), (1,3,4)\},
$$
so $A \inv (1,4) = 0$.  Likewise
\begin{align*}
\pc_A(2,4) = \{(2,4),(2,34)\} &&\pc_A(1,2) = \emptyset  &&\pc_A(3,4) = \{(3,4)\} 
\end{align*}
whence
\begin{align*}
A\inv(2,4) = 0 && A \inv(1,2) = 0 && A\inv(3,4) = -1.
\end{align*}
Altogether, $A\inv$ has form
\vspace{.5cm}

$$
\left(
\arraycolsep=8pt\def\arraystretch{1}
\begin{array}{cccc}
1&0 & 0 & 0 \\
0& 1 & 0 & 0\\
-1& -1& 1 & 0\\
0 & 0& -1& 1
\end{array}
\right).\vspace{.7cm}
$$
\end{example}

\chapter{Exchange Formulae}

This section is devoted to the calculation of certain algebraic and relational identities involved in structure exchange.  It is recommended that, on a first pass, the reader skim the definitions and results of \S\ref{sec:relations} before passing directly to the following chapter.  The facts cataloged in \S\ref{sec:exchangerelations} may be accessed as needed for later portions of the discussion.

\section{Relations}
\label{sec:relations}

A \emph{relation} on a set product $I \times J$ is a subset of $I \times J$.  A \emph{binary relation} on $I$ is a relation on $I \times I$. It is standard to write $i R j$ when $(i,j) \in R$.

We say that  $R$ is \emph{reflexive} if $i R i$ for all $i \in I$, \emph{antisymmetric} if $i = j$ when both $i R j$ and $j R i$, and \emph{transitive} if $i R k$ whenever $i R j$ and $j R k$.  A  relation with all three properties is  a \emph{partial order}.  A \emph{linear order} is a partial order for which every pair of elements is comparable, that is, when $i R j$ or $j R i$ for all $i, j \in I$.  The \emph{transitive (reflexive) closure} of $R$ is the least transitive (reflexive) relation that contains $R$.

The following is simple to verify for finite sets.  Recall that $S$ \emph{extends} $R$, or $R$  \emph{extends to} $S$ if $R\su S$.  A set that meets any of the conditions in Lemma \ref{lem:acyclic} is  \emph{acyclic}. An acyclic relation is \emph{strict}  if it  contains no pair of form $(i,i)$.    
%A linear order that contains $R$ is a \emph{linear extension}.
%  An \emph{acyclic relation} on $I$ is a binary relation contained in a partial order.
\begin{lemma} 
\label{lem:acyclic}  
For any binary relation $R$, the following are equivalent.
\begin{enumerate}
\item $R$ extends to a partial order on $I$.
\item $R$ extends to a linear order on $I$.
\item The transitive closure of $R$ is antisymmetric.
%\item The transitive reflexive closure of $R$ is a partial order on $I$.
\end{enumerate}
If $R$ is strictly acyclic and $i_p R i_{p+1}$ for $0 \le p \le m$, then $I$ has cardinality strictly greater than $m$.
\end{lemma}

%Acyclic relations enter our story by means of arrays.  Recall that an array $A$ has \emph{support on} $R$ if the support of $A$ is a subset of $R$.  Thus an $m \times m$ array is upper triangular iff it has support on the standard order on $\Z$.   More generally, let us say that $\mk$ {linearizes}  $R$ if $i \le j$ whenever $\mk(i) R \mk(j)$.   The following  holds for any array $A$ on $I\times I$ and any bijection $\mk : \{1, \ld, m\} \to I$.
%
%\begin{lemma} \label{lem:acyclicut} Array $A \com (\mk \times \mk)$ is upper-triangular iff $\mk$ linearizes $\Supp(A)$.
%\end{lemma}

%Recall that the support of $A$ is the set $R = \{(i,j) : A(i,j) \neq 0\}$.  A has \emph{off-diagonal} support if $A(i,i)$ vanishes for all $i \in I$.  The following is self-evident for strictly upper-triangular arrays, thus follows from Lemma \ref{lem:acyclicut}.

%\begin{corollary}
%If $R$ is strictly acyclic and $i_p R i_{p+1}$ for $p = 0,\ld, m$, then  $m < |I|$.
%\end{corollary}

A \emph{partial matching} on $R \su I \times J$ is a subset $\d \su R$ such that
\begin{align*}
\# \{(i,j) \in \d: i = i_0\}  \le 1&& and &&\# \{(i,j) \in \d: j=j_0\} \le 1
\end{align*}
for all $i_0 \in I$ and all $j_0 \in J$.  To every partial matching we may associate a set $\d \der \sharp = \{i : (i,j) \in \d\}$, called the \emph{domain} of $\d$,  and a set $\d \der \flat = \{j : (i,j) \in d\}$, called the \emph{image}.    A \emph{perfect matching} is a partial matching with $\d \der{\sharp} = I$ and $\d \der {\flat} = J$.

Every partial matching determines a unique function  $(\flat)$ sending  $i \in \d\der {\sharp}$ to the unique  $i \der \flat $ such that $(i, i \der \flat)\in \d$.  Symmetrically, there is a unique function  $(\sharp)$  such that $(j \der \sharp, j) \in \d$ for all $j$.  These maps are mutually inverse bijections, and we refer to each as a \emph{pairing function}.

\begin{example}  
The \emph{diagonal} of a set product $I \times I$ is the set 
$$
\ad(I \times I) = \{(i,i) : i  \in I\}.
$$   
The diagonal is a partial matching on $I \times I$.  Each of the two pairing functions on $\ad(I \times I)$ is  the identity map on $I$.
\end{example}

A partial matching on $R$ determines a pair of \emph{induced relations}.  The induced relation on $I$, denoted $R_\d$, is the transitive reflexive closure of
\[
\{(i,j)  : (i,j \der \flat) \in \d\}.
\]
in $I \times I$.  The induced relation on $J$, denoted $R^\d$, is the transitive reflexive closure of
\[
\{(i,j)  : (i \der \sharp,j) \in \d\}
\]
in $J \times J$.

The \emph{composition} of relations $R$ and $S$ is the set $RS = \{(i,k) : i R j, \; j S k\}$.  We will be interested, in future discussion, in compositions of form $S = R_\d R R^\d$.

\begin{proposition} \label{prop:holdacyclic} If $S = R_\d RR^\d$, then $R_\d = S_\d$ and $R^\d = S^\d$. %Any transitive relation that extends $R_d$ extends $S_d$ also.  Likewise, any transitive relation that extends  $R^d$  extends $S^d$.
\end{proposition}
\begin{proof}
Let us restrict to the first identity, as the second is similar.  Relabeling elements as necessary, assume without loss of generality that there exists a set
$$
K =  \d \der \sharp = \d \der \flat
$$
such that $\d  = \ad (K \times K)$.   Then $R_\d$ and $R^\d$ are the transitive reflexive closures of the left and righthand intersections, respectively, in
\begin{align*}
R \cap (I \times K) && R \cap (K \times K) && R \cap (K \times J).
\end{align*}
Let $R^\d_\d$ denote the transitive reflexive closure in $K$ of the center intersection.

It is evident that $R_\d \su S_\d$, so it suffices to show the reverse inclusion.  For this, it is enough to establish that $R_\d$ contains $(R_\d R R^\d) \cap I \times K$.  A pair $(i,l)$ belongs to this set iff there exist $j,k$ so that $i(R_d)j$, $jRk$, and $k(R^d)l$.  Since $l \in K$, one must in fact have $(k,l) \in  R^\d_\d \su R_\d$, hence $(j,k) \in R_\d$.  The desired conclusion follows.
%
%, so let $(i,j,k,l)$ be any tuple so that $i(R_d)j$, $jRk$, $k(R^d)l$.  If $i \neq j$, then necessarily $j \in K$, and in this case $(j,k) \in R^d_d$.
%
%pair in $S_d$.  If $i \neq l$ then necessarily $l \in K$.  Moreover, if there exists no chain of form
%
%Let $Q^d$, $Q^d_d$, and $Q^d$ denote the transitive reflexive closures of these sets in $I$, $K$, and $J$, respectively, so that $R_d \su Q_d$ and $R^d  \su Q^d$.
%
%We will argue the first assertion,  the second being similar.  For this let $T = Q_d R Q^d$.  Since $R_d \su S_d \su T_d$, it suffices to prove  that any transitive relation  extending $R_d$ also extends $T_d$, or equivalently,  that the transitive closure of $R_d$ contains $T_d$.
%
%Therefore suppose $(i,l)$ lies in the transitive closure of $T_d$.
\end{proof}

We say that a partial matching is \emph{acyclic} on $R$ if $R^\d$ and $R_\d$ are acyclic.   Corollary \ref{cor:clearingpattern} will enter later discussion as a simple means to relate the sparsity structures of arrays engendered by an LU decomposition with a variety of partial orders on their respective row and column indices.

\begin{corollary}   \label{cor:clearingpattern}
A pairing is acyclic on $R$ iff it is acyclic on $R_dR R^d$.
\end{corollary}
\begin{proof}  This is an immediate consequence of \ref{prop:holdacyclic}.
\end{proof}

\section{Formulae}  
\label{sec:exchangerelations}

%For the remainder of this chapter we will be  concerned with operations of form $\lk \mapsto \lk - a \cup b$, where $a \su \lk$ and $\lk - a \cup b$ is a (co)product.  We refer such operations as \emph{set exchange}.  

This section is devoted to the calculation of matrix identities involved in structure exchange.   %A detailed knowledge of these calculations will not be needed in later discussion.   Rather, the main contributions of this section are (i) the exchange formulae themselves, and (ii) the demonstration that in general, any calculation involving partial orders and elementary exchange can often be handled quite directly by a combination of the M\"obius inversion and induced relations.   
It is suggested that the reader  skip  this content,  returning as needed where the various exchange formulae arise.

\subsubsection{Relations}

Recall that in \S\ref{sec:relations} we defined a partial matching on a relation $R$.  For economy, let us now define a partial matching on an \emph{array} $A$ to be a partial matching $\d$  on $\Supp(A)$  such that $A(\d\der \sharp, \d\der \flat)$ is invertible.  By abuse of notation, we will use the same symbols that denote the \emph{sets}  $\d \der \sharp$ and $\d \der \flat$ for the \emph{product and coproduct maps} $\times \d \der \sharp$ and $\oplus \d \der \flat$, respectively.

It will be shown in \S\ref{sec:lu} that if $\d$ is a partial matching and $A = T(\lk, \uk)$ for some some coproduct (not product) structure $\lk$ and some product (not coproduct) structure $\uk$, then the left- and right-hand families
\begin{align*}
\lk - \left( \d  \der \sharp \right )^\flat \cup T \left (\d \der \flat \right) ^\flat  && \uk - \left(\d \der \flat \right )^\sharp \cup \left( \d \der \sharp \right)^\sharp  T
\end{align*}
form coproduct and product structures, respectively.    Denote these  by $\lk[\d]$ and $\uk[\d]$, for brevity.   

Recall that in order to avoid an excess of superscripts we  sometimes drop the sharp and flat notation when composing  maps, for instance\ writing
$$
T \d \der \flat
$$
for $T \left (\d \der \flat \right) ^\flat$.  There is no risk of ambiguity in this practice, as  $T \left (\d \der \flat \right) ^\sharp$ is not defined.  

 If $g$ is the dual to $T \d \der \flat $ in $\lk[\d]$, then $gT \d \der \flat  = 1$ by definition, so every partial matching on $T(\lk, \uk)$ determines a partial matching $\d_0$ on $T(\lk[\d], \uk)$.   We call the operation $(\lk, \uk) \mapsto (\lk[\d], \uk[\d_0])$ a \emph{row-first pivot} on $d$.  The row-first pivot defined in Section \ref{sec:lu} is the special case of this operation, when $d = \{(a,b)\}$.

A remark on notation: while the support of $1(\lk, \lk[\d])$ is formally a relation on $\lk \times \lk[\d]$, it may be naturally realized  as an relation on $\lk \times \lk$ by means of the associated pairing functions. 

\begin{lemma}  If $\d$ is a partial matching on $A = T(\lk, \uk)$ and $R = \Supp(A)$, then 
\begin{enumerate} 
\item $R_\d$ is the transitive symmetric closure of $\Supp(1(\lk, \lk[\d]))$
\item $R^\d$ is the transitive symmetric closure of $\Supp(1(\uk[\d_0], \uk))$
\end{enumerate}
under the canonical identification
\begin{align*}
\lk \times \lk \leftrightarrow \lk \times \lk[d] && \uk \times \uk \leftrightarrow \uk[\d_0] \times \uk
\end{align*}
given by the pairing functions on $\d$.
\end{lemma}
% If $R$ denotes the support of $T(\lk, \uk)$, then under this identification the induced relation $R_\d$ is the transitive symmetric closure of $1(\lk, \lk[\d])$.  %Likewise, $R^d$ is the transitive symmetric closure of $1(\uk[d], \uk)$.

\begin{proposition}  
Let $\d$ be a partial matching on $T(\lk, \uk)$ and $\d_0$ the associated pairing on $T(\lk[\d], \uk)$.  If $\d$ is a cyclic, then $\d_0$ is acyclic, also.
\end{proposition}
\begin{proof}  
That $\d$ is a pairing on $T(\lk[\d], \uk)$ holds by fiat.  To see that $\d$ is acyclic, let $R$ denote the support of $T(\lk, \uk)$, and note
%$$
%T(\lk[d\der \sharp, d \der \flat], \uk) = 1(\lk[d\der \sharp, d \der \flat], \lk) T(\lk, \uk) = 1( \lk,\lk[d\der \sharp, d \der \flat])\inv T(\lk, \uk)
%$$
\begin{align}
T(\lk[\d], \uk) = 1(\lk[\d], \lk) T(\lk, \uk) = 1( \lk,\lk[\d])\inv T(\lk, \uk).  
\label{eq:holdacyclic}
\end{align}
The transitive closure of $\Supp(1 (\lk, \lk[\d]))$ is the acyclic induced relation $R_\d$, so the righthand side of \eqref{eq:holdacyclic} lies in $R_\d R$, by M\"obius inversion.  Proposition \ref{prop:holdacyclic} implies that $\d$ is acyclic in $R_\d RR^\d$, and so in \eqref{eq:holdacyclic} \emph{a foriori}.
\end{proof}

%\begin{proposition}  Suppose $d$ is a partial matching on $T(\lk, \uk)$, and let $\wk = \lk[d \der \sharp, d \der \flat]$.   If $d$ is acyclic on $T(\lk, \uk)$, then it is acylic on $T(\wk, \uk)$.
%\end{proposition}

\begin{corollary} 
\label{cor:potchkai} 
Any transitive relation that extends the support of $1(\uk[\d], \uk)$ extends the support of $1(\uk[\d_0], \uk)$, also.
\end{corollary}
\begin{proof}  
If $R$ is the support of $T(\lk, \uk)$ then $R^\d$ is the transitive reflexive closure of the support of  $1(\uk[\d], \uk)$.  Proposition \ref{prop:holdacyclic} implies that any transitive relation extending $R^\d$ extends $(R_\d RR^\d)^{\d_0}$, also. The latter extends $1(\uk[\d_0], \d)$, and the desired conclusion follows.
\end{proof}

\begin{proposition}  
An invertible array with acyclic support has exactly one perfect matching.
\end{proposition}
\begin{proof}
Every invertible acyclic array has a column $c$ supported on a single row, $r$.   This row-column pair  belongs to every perfect matching.  The rows and columns complementary to $r$ and $c$  index a strictly smaller invertible acyclic array, and it may argued similarly that this, too, has a row-column pair that must be contained in any perfect matching.  The desired conclusion follows by a simple induction.
\end{proof}

\begin{proposition}  
If  $A$ is invertible with acyclic support, then the transitive closures of $\Supp(A)$ and $\Supp(A\inv)$ are identical.
\end{proposition}
\begin{proof}  
One has $\Supp(A \inv) \su \Supp(A)$ by M\"obius inversion.  The reverse inclusion holds by symmetry.
\end{proof}

\subsubsection{Schur Complements}

Propositions \ref{prop:schur1}, \ref{prop:schurexchange}, and \ref{prop:jordanexchange} are a elementary consequences of the splitting lemma (Lemma \ref{lem:splitting}).  The proofs are left as exercises to the reader.

\begin{proposition} \label{prop:schur1}  If $\lk_0\op T \lk_0$ is an isomorphism, then its Schur complement in $T(\lk\op, \lk)$ is the block   indexed by $(\lk\op_1,\lk_1)$ in the matrix representation of $T$ with respect to $(\lk\op,\lk)[\lk\op_0,\lk_0]$.
\end{proposition}

\begin{proposition} 
\label{prop:schurexchange} 
Let $\d$ be a partial matching on  $T(\lk, \uk)$, and let $f$ be any element of $\uk$ not removed by the operation $\uk \mapsto \uk[\d_0]$.  Then the dual to $f$ in $\uk[\d_0]^\flat$ is 
$$
e_k f^{\flat}
$$
where $e_k$ is idempotent projection onto the null space of $\d \der \sharp T$ along $\d \der \flat$, and $h^\sharp$ is the dual to $f$ in $\uk$.  Likewise, if $f\op$ is any element of $\lk$ not removed by $\lk \mapsto \lk[\d]$, then the dual to $f\op$ in $\lk[\d]^\sharp$ is
$$
f^{op \; \sharp}  e_{k \op}
$$
where $e_{k\op}$ is the natural opposite to $e_k$   (that is, the  idempotent so that $e_{k^{op}} T \d\der{\flat} = 0$   and $h e_{k^{op}} = h$ iff $h$ vanishes on the image of $T \d \der \flat$), and where $f^{op\; \flat}$ is the dual to $f\op$ in $\lk$.   Under these assumptions
\begin{align}
 f ^{op \; \flat} e_{k\op} T f^{\sharp} = f ^{op \; \flat} e_{k\op} T e_k f^{\sharp} = f ^{op \; \flat} T e_k f^{\sharp} .   \label{eq:schurformula}
\end{align}
The matrix representation of $T$ with respect to $(\lk,\uk)[\d]$ has form
$$
\diag(1, \sk)
$$
where $\sk$ is the Schur complement of the invertible submatrix indexe $\d\der \sharp \cup \d \der \flat$.  The elements of $\sk$ are given by \eqref{eq:schurformula}.
\end{proposition}

\begin{proposition} \label{prop:jordanexchange} Suppose that $\{f,g,h\}$ is a coproduct structure on $W$.  If $T^2 = 0$ and $g^\sharp T f$ is invertible, then $(f,g)$ is a Jordan pair, and $e_k h$ complements $Tf$ in $K(g^\sharp T)$, where $e_k$ is idempotent projection onto the kernel of $g^\sharp T$ along $f$.  With respect to the ordered coproduct structure $(e_k h, f, Tf)$, $T$ has matrix representation
\vspace{.2cm}

$$
\left(
\arraycolsep=7pt\def\arraystretch{.7}
\begin{array}{ccc}
\tk & 0 & 0 \\
0 & 0 & 0 \\
0 & 1 & 0.
\end{array}
\right)
$$
\vspace{.2cm}

\noindent where $e_{k\op}$ is the natural dual to $e_k$  (that is, the  idempotent so that $e_{k^{op}} T f = 0$   and $h e_{k^{op}} = h$ iff $h$ vanishes on the image of $Tf$), $h^\sharp$ is the dual to $h$ in $\{f,g,h\}$, and $\tk$ is any of the following
\begin{align*}
h^\sharp e_{k^{op}} T h && h ^\sharp e_{k^{op}} T e_k h && h^\sharp T e_k h.
\end{align*}
That is, $\tk$ is the submatrix indexed by $(e_k h, e_k h)$ in the Schur complement of $g^\sharp T f$.
\end{proposition}

%
%Proposition  \ref{prop:jordanexchange} likewise yields the following.  We call submatrix $\tk$ in Proposition \ref{prop:jordancomplement} the \emph{Jordan complement} to $(f,g)$ in $T$.
%
%\begin{proposition} \label{prop:jordancomplement} Suppose that $\{f,g,h\}$ is a coproduct structure on the base space of an operator $T$ such that $T^2 = 0$.  If $g^\sharp T f$ is invertible, then $(f,g)$ is a Jordan pair, and $e_k h$ complements $Tf$ in $K(g^\sharp T)$, where $e_k$ is the idempotent that realizes projection onto the kernel of $g^\sharp T$ along $f$.  With respect to the ordered coproduct structure $(e_k h, f, Tf)$, $T$ has matrix represenation
%$$
%\left(
%\begin{array}{ccc}
%\tk & 0 & 0 \\
%0 & 0 & 0 \\
%0 & 1 & 0.
%\end{array}
%\right)
%$$
%where 
%\end{proposition}

\subsubsection{M\"obius Inversion}

\begin{proposition} 
\label{prop:schurmobius} 
Let $\d$ be an acyclic partial matching on $T(\lk, \uk)$.  If $A = 1( \lk,\lk[\d])$ and  $C = T(\lk[\d], \uk[\d_0])$, then
\begin{align*}
C(f,g) =
\begin{cases}
\dk_{f \der \flat ,g} &  f \in \d \der \sharp \;  or \; g \in \d \der \flat  \\
\sum_{h \in \lk} \sum_{\pc_A(f,h)} (-1)^{\ell(p)+1} A[p] T(h,g) & {\rm otherwise}.
\end{cases}
\end{align*}
\end{proposition}
\begin{proof}  
The first case holds by fiat.  For the second, let $B = T(\lk[\d], \uk)$.  It is a simple consequence of the splitting lemma (Lemma \ref{lem:splitting}) that $C(f,g) = B(f,g)$ when $f \notin \d \der \sharp$ and $g \notin \d \der \flat$.  Consequently the second case follows from the observation  $B  = 1(\lk[\d], \lk) T(\lk, \uk) = 1(\lk, \lk[\d]) \inv T(\lk, \uk)$, plus M\"obius inversion.
\end{proof}

Let us say that a partial matching $\d$ is \emph{Jordan} if it is a proper Jordan pair in the sense of Section \ref{sec:jordanalgebra}.  The proof of the following statement is entirely analogous to that of Proposition \ref{prop:schurmobius}.

\begin{proposition}  \label{prop:jordancomplementmobiusformula}
Suppose $T^2 = 0$, and let $\d$ be an acyclic Jordan pairing on $T(\lk, \uk)$.  If   $A = 1(\lk[\d], \lk)$ and $C$ is the matrix representation of $T$ with respect to $\lk[[\d]]$, then
\begin{align*}
C(f,g) =
\begin{cases}
\dk_{f \der \flat ,g} &  \{f, g\} \cap (\d \der \sharp \cup \d \der \flat) \neq \emptyset \\
\sum_{h \in \lk} \sum_{\pc_A(f,h)} (-1)^{\ell(p)+1} A[p] T(h,g) & {\rm otherwise}.
\end{cases}
\end{align*}
\end{proposition}

%\afterpage{\blankpage}

\chapter{Exchange Algorithms}  
\label{ch:elementaryexchange}

\section{LU Decomposition}
\label{sec:lu}

Let us recall Lemma \ref{lem:splitting} and its logical equivalent, Lemma \ref{lem:exchange}.

{
\renewcommand{\thetheorem}{\ref{lem:splitting}}
\begin{lemma}[Splitting Lemma]
Suppose that 
$$
f \op f = 1,
$$
where $f\op$ and $f$ are composable morphisms in a preadditive category.  
\begin{enumerate}
\item If $k$ is kernel to $f\op$, then $\{f,k\}$ is a coproduct structure.
\item If $k\op$ is cokernel to $f$, then $\{f\op, k\op\}$ is a product structure.
\end{enumerate}
\end{lemma}
\addtocounter{theorem}{-1}
}
\begin{remark}  As noted in \S\ref{subsec:splittinglemma}, the criterion $f\op f = 1$ may be loosened to the condition that $f\op f$ be {invertible}.
\end{remark}

%Recall Example \ref{ex:comp} of Section \ref{sec:biproducts}:  If $\lk$ is any family of {injective} maps into  $W$, then $\lk$ is a coproduct structure iff the family of images $\{\I(f) : f \in \lk\}$ is complementary in $W$.  This observation yields a simple proof of the following.

%\begin{lemma}[Exchange Lemma] 
%\label{lem:exchange}  
%If $\{f,g\}$ is a coproduct  then $\{f,h\}$ is a coproduct  iff $g^\sharp h$ is an isomorphism.  Dually, if $\{f,g\}$ is a product  then $\{f,h\}$ is a product iff $h g ^\flat$ is an isomorphism.
%\end{lemma}

%A map $k$ is \emph{kernel to} $f$ if every $g$ for which $fg = 0$ factors through $k$.  Equivalently, $k$ is kernel to $f$ if it is injective and its image is the null space of $f$.  Dually, we say that $k\op$ is \emph{cokernel} to $f$ if every $g\op$ for which $g\op f  = 0$ factors through $k\op$.  Equivalently, $k\op$ is a cokernel to $f$ if it is surjective and its null space is the image of $f$.

%\begin{lemma}[Splitting Lemma]  
%\label{lem:splitting}
%If $f\op f$ is invertible then $\{f,k\}$ is a coproduct for any kernel  to  $f\op$.  Dually, $\{f\op, k\op\}$   is a product for any cokernel to $f$.  
%\end{lemma}

%\begin{remark}  
%Those familiar with the language of categories may recognize Lemma \ref{lem:splitting} as a variation of a result by the same name in homological algebra.  As stated, Lemmas \ref{lem:exchange} and \ref{lem:splitting} remain true in any preadditive category.
%\end{remark}

As previously observed, the main application of the splitting lemma in this discussion   is to clarify the relationship between certain dual structures.  Suppose, for example, that $\lk = \{\lk_f, \lk_g\}$ and $\uk = \{\uk_f, \uk_h\}$ are products.  How does $\lk_f^\flat$ relate to $\uk_f^\flat$ when $\lk_f = \uk_f$?    Recall that every $k$ in a (co)product $\wk$ engenders an idempotent $e_k = k^\flat k^\sharp$.  It is elementary to show that  $\lk_f^\flat = e_k \uk_f^\flat$ when  $\wk= \{\lk_f^\flat, k\}$ and $k$ is  kernel to $\lk_f$.   It can be shown similarly that $\lk_g^\flat = \lk_h^\flat$ when $\lk_g \lk_h^\flat = 1$.  If $\lk$ and $\uk$ are coproducts, then $\lk_f^\sharp = v^\sharp_f e_{k\op}$ whenever $k^{op}$ is cokernel to $\lk_f$; likewise  $\lk_g^\sharp = \lk_h^\sharp$, when $\lk_g^\sharp \lk_h = 1$.  These identities represent what are essentially a variety of different access points to the same underlying mathematical structure, and we group all  under the common heading of the splitting lemma.

The main application of duality, for this discussion, is to describe exchange.  Suppose $\lk$ and $\uk$ are coproduct and product structures, respectively, on the codomain and domain of an operator $T$.  An \emph{(elementary) exchange pair} for $\lk$ and $\uk$  is an ordered pair $(a,b) \in \lk \times \uk$ such that $a^\sharp T b^\flat$ is invertible.  By the Exchange Lemma,  $(a,b)$ is an exchange pair if and only if $\lk[a,b^\flat] = \lk - a \cup T b^\flat$ is a coproduct and $\uk[a^\sharp,b] = \uk - b \cup a^\sharp T$ is a product.  We call the operation that produces these structures  \emph{elementary exchange}.

{
\renewcommand{\thetheorem}{\ref{lem:exchange}}
\begin{lemma}[Exchange Lemma]
Let $\{f,g\}$ and $\{f\op, g\op\}$ be product and coproduct structures, respectively.  
\begin{enumerate}
\item An unordered pair $\{f, h\}$ is a product iff $g ^\sharp h$ is invertible.
\item An unordered pair $\{f\op, h\op \}$ is a coproduct iff $h\op \left ( g \op \right)^\flat$ is invertible.
\end{enumerate}
\end{lemma}
\addtocounter{theorem}{-1}
}

Individual  exchanges may be combined in an important way.  Suppose $\sharp$ and $\flat$ are the usual sharp and flat operators on $\lk$ and $\uk$, and let $\sharp \sharp$, $\flat \flat$ denote the corresponding operators on $\lk[a,b^\flat]$ and $\uk[a^\sharp,b]$, respectively.  The composition $(Tb^\flat)^{\sharp \sharp}T(b^\flat)$ is identity by definition, so the righthand side of
\begin{align}
 \lk[a,b^\flat]   &&    \uk[(Tb^\flat)^{\sharp \sharp},b]
\label{eq:rowfirstpivot}
\end{align}
is a product.   We write $(\lk,\uk)[a,b]$ for the pair  \eqref{eq:rowfirstpivot}, and refer to the operation 
$$
(\lk,\uk) \mapsto (\lk,\uk)[a,b]
$$ 
as  a \emph{row-first pivot} on $(a,b)$.  There is a natural mirror to this operation, the \emph{column-first pivot}, which maps $\uk$  to $\uk[a^\sharp, b]$ and  $\lk$  to $\lk[a, (a^\sharp T)^{\flat \flat}]$.  Column pivots will not enter our story directly, though an equivalent, antisymmetric story may be told exclusively in terms of these operations.

The following algorithm initiates with a pair $(\ak,\bk) = (\lk, \uk)$, where $\uk$ is a product on the domain of $T$ and $\lk$ is a coproduct on the codomain of $T$.     To avoid degeneracies, we stipulate that all elements of $\lk$ and $\uk$ have rank one, so that $\ak \times \bk$ fails to contain an exchange pair iff $(\times \ak ^\sharp) T(\oplus \bk^ \flat)= 0$.
\vspace{.55cm}

%Our description makes use of a function that, informally, sends each element of $\lk - a$ to the first element of $\lk[a,b]$ one might think of.  Formally, this is the map $\sharp_0 \com \flat$, where $\sharp_0$ is the sharp operator on $\lk[a,b]$ and $\flat$ is the flat operator on $\lk$.  We write $\ak \dot - a$ for the subset of $\lk[a,b]$ corresponding to $\ak - a$ under this map.  We define $\bk \dot - b$ analogously.   As we have introduced notation anyway, let us write $(\ak, \bk) \dot - (a,b)$ for the operation that \emph{assigns} value $\ak \dot - a$ to variable $\ak$ and value $\bk \dot - b$ to $\bk$.  Algorithm \ref{alg:BPLU} may now be described as follows.

\begin{algorithm}
\begin{algorithmic}
\WHILE{ $\ak \times \bk$ contains an exchange pair}
\STATE   fix an exchange pair  $(a,b) \in \ak \times \bk$.%\in \Supp(M)$ with $i \notin I_0$, $j \notin J_0$
\STATE   $(\lk, \uk)  \leftarrow (\lk, \uk)[a,b]$
%\STATE $\uk \leftarrow \uk[a,b]$
%\STATE $L \leftarrow L(M, (i,j))$
%\STATE $U \leftarrow U(L\inv M, (i,j))$
%\STATE  $M \leftarrow LMU$
%\STATE $p \leftarrow p+1$
\STATE $(\ak, \bk) -= (a,b)$
\ENDWHILE
\end{algorithmic}
\caption{LU Decomposition}
\label{alg:BPLU}
\end{algorithm}

Let $\lk^p$ and $\uk^p$ denote the coproduct and product structures generated on the $p$th iteration of Algorithm \ref{alg:BPLU}, with $\infty$ denoting the final iteration.  For convenience, let us index $\lk$ and $\lk^\infty$ so that the elements of $\lk^p$ may be arranged into a a tuple  of form  $(\lk^\infty_1, \ld , \lk^\infty_p, \lk_{p+1}, \ld, \lk_m)$ for each $p$, and order the elements of $\uk^p$ into a tuple $(\uk^\infty_1, \ld, \uk^\infty_p, \uk_{p+1}, \ld, \uk_n)$, similarly.  %We would like to understand something about the duals to these structures.
\begin{remark}
One benefit of working with indexed families is that there is no need for the special notation to differentiate between the various sharp (respectively, flat) operators.  Whereas previously one had to write double scripts $\sharp \sharp$ and $\flat \flat$, under this  convention one simply has $\uk_p^p = (\lk_p^p)^\sharp T$ and $\lk_p^p = T(\uk^{p-1}_p)^\flat$.
\end{remark}

\begin{lemma} 
\label{lem:diracLU} 
For any $r\ge p$, one has $(\uk^{r }_p)^\flat = (\uk^{p-1 }_p)^\flat$ and $ (\lk^{r}_p)^{ \sharp} = (\lk^{p}_p)^{\sharp}$.  Moreover,
\begin{align}
(\lk_q^p)^\sharp T (\uk_p^p)^\flat = \dk_{qp} && (\lk_p^p)^\sharp T (\uk_q^p)^\flat = \dk_{qp}   \label{eq:deltarowcol}
\end{align}
for any $p$, $q$.
\end{lemma}
\begin{proof}  
Since  $(\lk^{p}_p)^{ \sharp} T (\uk_p^{p-1})^{\flat} = 1$ by definition, the splitting lemma implies $(\uk_p^p)^\flat = (\uk_p^{p-1})^\flat$.  Thus $\lk_p^p = T(\uk_p^p)^\flat$ and $\uk_p^p = (\lk_p^p)^\sharp T$, whence \eqref{eq:deltarowcol}.  It follows from these identities that $(\lk_p^p)^\sharp$ and $(\uk_p^p)^\flat$ remain invariant under any increase in upper index, again by the splitting lemma.
%The first identity follows from  a simple induction.  The base case holds by the splitting lemma, since by definition $(\lk^{p}_p)^{ \sharp} T (\uk_p^{p-1})^{\flat} = 1$, hence   $(\lk_p^p)^\sharp T(\uk_p^p)^\flat = 1$.  Thus $\lk_p^p = T(\uk_p^p)^\flat$ and $\uk_p^p = (\lk_p^p)^\sharp T$.
%  Suppose by induction that $(\uk^{q }_p)^\flat = (\uk^{p-1 }_p)^\flat$ for $p-1 < q < r$, and recall $\lk_p^q = \lk^\infty_p = T(\uk_p^{p-1})^\flat$.  It follows that $(\lk_p^q)^\sharp T (\uk_p^q)^\flat = 1$.
\end{proof}

\begin{lemma} 
\label{lem:triangleformula} 
For any $p$ and $q$ one has
\begin{align}
\lk_p^\sharp \lk^\infty_q = (\lk_p^{q-1})^\sharp T (\uk_q^{q-1})^\flat &&  \uk_p^\infty \uk_q^\flat = (\lk_p^p)^\sharp T (\uk_q^{q-1})^\flat.
\end{align}
The left-hand operator vanishes when $p<q$;  the right-hand when $q<p$.
\end{lemma}
\begin{proof}  
Let $f$ denote the righthand side of the lefthand identity.  When $p < q$ one may increase the upper indices in this expression by one without changing its value, and consequently $f$ vanishes, by \eqref{eq:deltarowcol}.    Since $\lk^{p-1}=(\lk^\infty_1, \ld , \lk^\infty_{p-1}, \lk_{p}, \ld, \lk_m)$, it follows from the splitting lemma that $f = \lk_p^\sharp T(\uk_q^{q-1})^\flat$ for any $p$ and $q$.  This suffices for the lefthand identity, as   $T(\uk_q^{q-1})^\flat = \lk^\infty_q$ by definition.  The  righthand identity and  may be argued similarly, as may its vanishing properties.
\end{proof}

In the language of arrays, one may interpret Lemmas \ref{lem:diracLU} and \ref{lem:triangleformula} to say that
\begin{align*}
1(\lk, \lk^p)& = [ \; \mk_1 \; | \; \cd \; | \; \mk_p \; | \; \chi_{p+1} \; | \; \cd \; | \; \chi_m \; ] \\
1(\uk^p, \uk) &= [ \; \nk_1 \; / \; \cd \; / \; \nk_p \; / \; \chi_{p+1} \; / \; \cd \; / \; \chi_n \; ]
\end{align*}
are lower and upper triangular, respectively,  where $\mk_p$ is the $p$th column of $T(\lk^{p-1}, \uk^{p-1})$   and $\nk_p$ is the $p$th row of $T(\lk^p, \uk^{p-1})$.  Moreover,
\begin{align*}
T(\lk^\infty, \uk^\infty) & = \diag(1, \ld, 1, 0, \ld, 0)
\end{align*}
with nonzero entries appearing on exchange pairs.  Thus
\[
T(\lk, \uk) = 1(\lk, \lk^\infty)  T(\lk^\infty, \uk ^\infty) 1(\uk^\infty, \uk)
\]
where $1(\lk, \lk^\infty)$ is lower triangular, $1(\uk^\infty, \uk)$ is upper-triangular, and $T(\lk^\infty, \uk^\infty)$ is a zero-one array with at most one nonzero entry per row and column.  If $\lk$ and $\uk$ are composed of maps into and out of $\field$, then the entries in these arrays are maps $f:\field \to \field$.   Under the standard identification  $f \leftrightarrow f(1)$,  the preceding identity corresponds exactly to an LU decomposition of  $T(\lk, \uk)$.

\begin{remark}  The preceding observations did not depend on our restriction to the category of linear maps on finite-dimensional vectors spaces.  Rather, they yield an LU decomposition for any biproduct of finitely many simple objects in a preadditive category.  The existence of such a factorization was remarked  by Smith \cite{smith1983schur}, though we are unaware of  combinatorial treatments of the subject.
\end{remark}

\section{Jordan Decomposition}  
\label{sec:jordanalgebra}

In what follows, readers unfamiliar with the language of category theory may replace the phrase \emph{endomorphism in an abelian category} with \emph{linear map $W \to W$}, with no loss of correctness.   
%In the sequel we will make use of the following trivial (if somewhat boring to check) lemma.  In the familiar language of linear algebra, this boils down to something akin to the observation that
%$$
%W = \I(f) +  \I(g_1) + \cd + \I(g_p)
%$$
%and $D(f) = \I(h_0) + \cd + \I(h_q)$, then
%$$
%W = \I(f h_0) + \cd \I(f h_q) + \I(g_1) + \cd + \I(g_p).
%$$
%
%%, which codifies the observation that whenever $W \cong X \oplus Y$ and $Y \cong A \oplus B$, one has $W \cong X \oplus A \oplus B$.
%
%\begin{lemma}  \label{lem:compositioncoproduct} Let $\wk$ be a coproduct structure, and fix $f \in \wk$.  If $\lk$ is a coproduct structure on $D(f)$, then 
%$$
%\{f  h : h \in \lk \} \cup \{g  : g \in \wk, \; g \neq f\}
%$$
%is a coproduct structure, also.
%\end{lemma}

%\begin{theorem}  Let $T$, $f$, and $g$ be morphisms in an arbitrary abelian category such that $T^{n+1} = 0$ and  $g T^n f = 1$.
%\begin{enumerate}
%\item If $k_g$ is inclusion from  $\K(gT^n) \cap \cd \cap \K(g T^0)$ to the domain of $T$, then
%\begin{align}
%(k_g,  T^0f, \ld, T^{n}f)  
%\label{eq:jordanpiece2}
%\end{align}
%is a coproduct structure.
%\item If in addition 
%\begin{align}
%(k, T^0 f, \ld, T^n f)
%\end{align} 
%with respect to which $T$ has form
%\end{enumerate}
%\end{theorem}

\begin{theorem}  \label{prop:jordansplit}   Let $f$, $g$, and $T$ be morphisms in an arbitrary abelian category.  If $T^{n+1} = 0$ and $g T^{n}f = 1$ then
\begin{align}
(k,  T^0f, \ld, T^{n}f)  
\label{eq:jordanpiece2}
\end{align}
is a coproduct structure, where $k$ is the inclusion of
$$
\K(gT^0) \cap \cd \cap \K(g T^n)
$$
into the domain of $T$.  
\end{theorem}
%If $h$ is dual to $T^n f$ in this structure and $l$  the inclusion for 
%$$
%\K(hT^0) \cap \cd \cap \K(h T^n),
%$$
%then 
%\begin{align}
%(l,  T^0f, \ld, T^{n}f)   \label{eq:coprodstruct2}
%\end{align}
%is also a coproduct structure, and with respect to which $T$ has form

\begin{proof}
Set
$\K_p = \K(g T^{p}) \cap \cd \cap \K(gT^{n}).$
As per custom for intersections over the empty set,  $\K_{n+1} = \D(T)$.    To each $p$ correspond unique morphisms $f_p$ and $g_p$ such that the following diagram commutes (the vertical arrow is inclusion).
$$
\xymatrix@R=2.5pc@C=5.5pc@M=.7pc{
&  \D(T)&\\
\ar[rr];[ur]_{T^{n-p}f}  \ar[ur];[]_{gT^{p}}  \ar[rr];[dr]^{f_p}  \ar[dr];[]^{g_p}  \ar[dr];[ur]_{} 
\D(f) &   &\D(f) \\
& \K_{p+1} &
}
$$
The lower lefthand map is simply a restriction of $gT^p$.  Its kernel (object) is  $\K(g T^{p}) \cap \K_{p+1} = \K_p$, so its kernel (map) may be taken to be the inclusion
$k_p: \K_p \to K_{p+1}$.

The lower two diagonal arrows compose to 1, by commutativity, so the splitting lemma provides a coproduct structure of form  $\{k_p,f_p\}$.   Every such structure corresponds to a pair of arrows in Figure \ref{fig:jordantree} sharing a common codomain.  This figure may be regarded informally as a directed tree rooted at $\K_{n+1}$.  Composition of morphisms gives a natural map from directed paths in the tree to morphisms in our abelian category.  The maps attained from \emph{maximal} paths -- those  that begin at a leaf and end at the root -- form a coproduct structure on $K_{n+1} = \D(T)$.  Said structure is precisely \eqref{eq:jordanpiece2}.
\begin{figure}
$$
\xymatrix@R=2.5pc@C=4.5pc@M=.5pc{
& \K_{n+1} &&&\\
 \D(f) \ar[ur]^{f_n}&&\ar[ul]_{k_n} \K_n& && \\
& \D(f)\ar[ur]^{f_{n-1}}&&\ar[ul]_{k_{n-1}} \K_{n-1}&\\
&& \D(f)\ar[ur]^{f_{n-2}}&&\ar[ul]_{k_{n-2}} \ddots
}
$$
\caption{Morphisms generating a coproduct structure on $\K_{n+1} = \D(T)$.}
\label{fig:jordantree}
\end{figure}
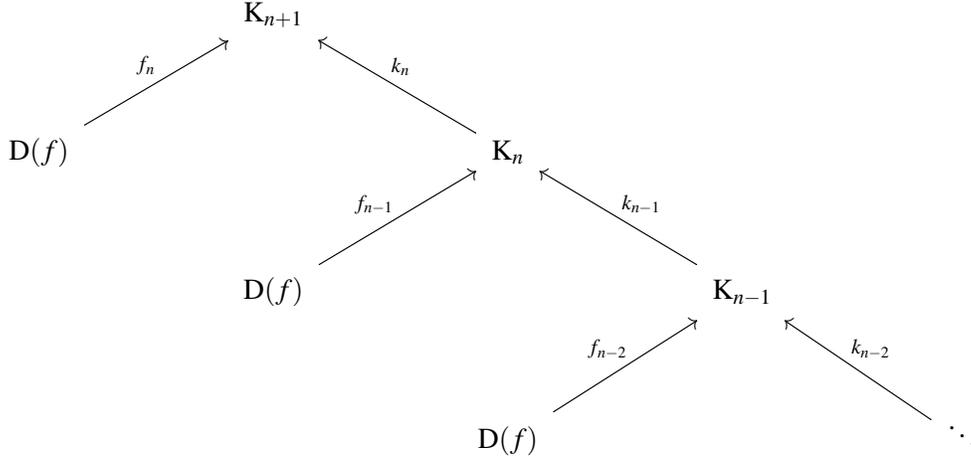
%This establishes the first claim.  That \eqref{eq:coprodstruct2} is a coporoduct structure follows as a special case, since by hypothesis $h T^n f = 1$.  That the matrix representation of $T$ with respect to \eqref{eq:coprodstruct2}  has form \eqref{eq:jordanproof} follows almost immediately from the definition of a biproduct structure.
\end{proof}

\begin{corollary} 
\label{cor:jordansplit}  
If $h$ is the biproduct dual to $T^n f$ in \eqref{eq:jordanpiece2}, then
\begin{align}
(l,  T^0f, \ld, T^{n}f)  
\label{eq:jordanpiece}
\end{align}
is a coproduct structure, where $l$ is the inclusion of
$$
\K(hT^0) \cap \cd  \K(hT^n)
$$
into the domain of $T$.   The matrix representation of $T$ with respect to this structure has form \eqref{eq:jordanproof}.

\begin{align}
\left(
\arraycolsep=7pt\def\arraystretch{.7}
\begin{array}{cccccccc}
* & 0 & 0 &    \cd &0& 0 \\
0 & 0 & 0 &\cd& 0& 0 \\
0&1& 0&  \cd& 0&0\\
\vdots & \vdots &   \vdots & \ddots & \vdots &\vdots\\
0 & 0 & 0 & \cd & 0&0\\
0 & 0  & 0 & \cd  & 1 & 0
\end{array}
\right).
\label{eq:jordanproof}
\end{align}

\end{corollary}

\begin{proof}
That \eqref{eq:jordanpiece} is a coproduct structure follows from Theorem \ref{prop:jordansplit}, since $h T^n f = 1$ by hypothesis.   That the matrix representation of $T$ has the desired form follows almost immediately from the definition of a biproduct structure.
\end{proof}

By analogy with linear spaces, let us say that a map $f$ \emph{complements}  $g$ if $\{f, g\}$ is a coproduct structure.  We refer to any map $k$ that realizes the conclusion to Theorem \ref{prop:jordansplit} as a \emph{Jordan complement} of $f$.

%\begin{corollary}  A map $k$ realizes the conclusion to Proposition  \ref{prop:jordansplit} if and only if it complements $T^n f$ in $K(gT^0 \oplus \cd \oplus g T^{n-1})$.
%\end{corollary}
%\begin{corollary} The dual to \eqref{eq:jordanpiece} has form
%$$
%(k^*, gT^n, \ld, gT^0).
%$$
%\end{corollary}

This proposition points to a general method for computing  Jordan complements for endomorphisms on finite semisimple objects in general, and on finite dimensional linear spaces in particular.   Suppose that $\wk$ is a coproduct structure whose elements are inclusions from simple objects (read: copies of the ground field) into $W$.  If $T^{n+1}$ vanishes and $T^n$ does not, then there exists an $f \in \wk$ such that  $T^n f \neq 0$.  For such $f$  there exists at least one 
$$
g \in \wk - \{f\}
$$
so that $g ^\sharp T f$ is an isomorphism --  otherwise  either (i) every element of $\wk^\sharp$  post-composes to zero with $T^n f$ (impossible, since  $T^n f$ would then have to vanish), or (ii) $f^\sharp$ is the sole exception to this rule (also impossible, since then $f^\sharp T f$ would be the only nonzero entry in column $f$ of the associated matrix representation, so that  $f^\sharp T^p f$ would be nonzero for all positive $p$;  in particular, $T$ would fail to be nilpotent).   The pair $(g,f)$ is then an exchange pair, and we may write
$$
\wk^1 = \wk \cup \{Tf\} - \{g\}
$$
for the coproduct obtained by exchanging $Tf$ for $g$.  The map $Tf$ indexes a column of $T(\wk^1, \wk^1)$, so provided that $n> 1$ we may find a
$$
g \in \wk^1 - \{f, Tf\}
$$ 
so that $g^{\sharp} T^2 f$ is invertible, and set
$$
\wk^2 = \wk^1 \cup \{T^2 f\} - \{g\}.
$$
This process repeats.  To see that on each iteration there exists a
$$
g \in \wk^p - \{T^0f, \ld, T^pf\}
$$
such that $g^\sharp T^p f$ is an isomorphism, we it helpful to visualize the sparsity pattern of $T(\wk^p, \wk^p)$.  Once one recognizes that the columns indexed by $T^0, \ld, T^{p-1} f$ are zero-one arrays concentrated on distinct elements of $\wk^p$, the desired conclusion becomes clear.     

For convenience, let us sum the elements of $\wk^n$ that do not take the form $T^p f$  into a single complement $g$, thus forming a coproduct
$$
(g,T^0f, \ld, T^n f).
$$
Let $h$ denote the dual to $T^n f$ in this structure.  If  $hT^{n-p}$ is dual to $T^n f$ for all $p$, then we are done, but such will not always be the case.  It will be true, however, that the composition
\[
(hT \times  \cd \times hT^{n})  \com  (T^{n-1} f \oplus \cd \oplus T^0 f)
\]
is  identity, since the matrix representation of $T$ with respect to $\wk^n$ agrees with \eqref{eq:jordanproof} in all but the first column.    Thus we may exchange $\{h T, \ld, h T^{n}\}$  for the duals to  $\{T^{0} f, \ld, T^{n-1} f\}$  in $(\wk^n)^\sharp$.  

  By the splitting lemma, the resulting coproduct structure will have form 
  $$
  (eg,T^0f, \ld, T^{n-1} f, e T^n f),
  $$ 
  where $e$ is projection onto 
$$
\K_1 = \K(hT^1 \times  \cd \times hT^{n})
$$
along $T^{0} f \oplus \cd \oplus T^{n-1} f$.  Since $T^n f$  factors through $\K(T) \su \K_1$, this structure agrees with
\begin{align}
(eg,T^0 f, \ld, T^{n} f).  \label{eq:themoney}
\end{align}
The matrix representation of $T$ with respect to \eqref{eq:themoney} will agree with \eqref{eq:jordanproof} in all but the second row of the first column.  This too must agree, however, since otherwise $T^{n+1}$ would fail to vanish.   Thus we have established the following.

\begin{corollary}
If $e$ is the idempotent operator on $W$ that realizes projection onto $K_1$ along $T^0f \oplus \cd \oplus T^{n-1} f$, then 
$eg $
is a Jordan complement to $f$. 
\end{corollary}

%In addition to furnishing a general procedure for computing Jordan complements, this argument provides the following.

\begin{corollary}  
A map $k$  is a Jordan complement to $f$ iff $k$ complements $T^n f$ in the kernel of a morphism
 $H$ such that 
 $$
 H \com  (T f \oplus \cd \oplus T^{n} f) = 1.
 $$% is an %$k$ complements $T^n f$ in $K(g T^0\times \cd \times gT^{n-1})$.
\end{corollary}

Let us now describe a general algorithm to compute Jordan decompositions (not just Jordan complements) for a nilpotent morphisms. For convenience, denote the structure \eqref{eq:themoney} by $\wk[[f,h]]$.  We call $(f, h)$ a \emph{Jordan pair} in $T$, and refer to the operation 
$$
\wk \mapsto  \wk[[f, h]]
$$ 
as a \emph{Jordan exchange}.  

Informally, the decomposition algorithm works by splitting orbits off of Jordan complements.  On each iteration is given a coproduct structure $\wk$, expressed as a disjoint union of form $J_1 \cup \cd \cup J_p \cup I$, where $J_1, \ld, J_p$ are maximal (with respect to inclusion) orbits.  One selects an element $f \in \tm{argmax}_{i \in I}  |\orb(i)|$, and rotates first $\orb(f)$, and then the necessary duals to $\orb(f)$ into the biproduct structure.  The result is is a biproduct a strictly greater number of orbits and a strictly smaller complement, $I$.  Once this complement vanishes, the process terminates.

%From among the orbits of the $l_p$'s we choose one of maximal length, and rotate its elements into the structure, removing other $l_p$
%
%From among the $l_p$ is chosen an element with a maximal length orbit, and swap the elements of this orbit in for
%
%by induction.  Iteration $k+1$ initiates with a coproduct structure for which $T$ has form $\diag(J_1, \ld, J_k, \sk)$, and executes a Jordan exchange on $\sk$, producing a new Jordan block, $J_{k+1}$.  Once $\sk$ vanishes the algorithm terminates and the resulting coproduct is returned.

To aid in a formal description,  let us say that $(f,h)$ is  a \emph{proper Jordan pair} if (i) it is a Jordan pair, (ii)  $\orb(f)$ is not already a Jordan block in $\wk$, and (iii) among the pairs that meet criteria (i) and (ii),  $\orb(f)$  has maximal cardinality.
\vspace{.55cm}

\begin{algorithm}
\begin{algorithmic}
\WHILE{ $\wk$ has a proper Jordan pair }
\STATE   fix a proper Jordan pair for $\wk$.%\in \Supp(M)$ with $i \notin I_0$, $j \notin J_0$
\STATE   $\wk  \leftarrow \wk[[f,h]]$
%\STATE $\uk \leftarrow \uk[a,b]$
%\STATE $L \leftarrow L(M, (i,j))$
%\STATE $U \leftarrow U(L\inv M, (i,j))$
%\STATE  $M \leftarrow LMU$
%\STATE $p \leftarrow p+1$
\ENDWHILE
\end{algorithmic}
\caption{Jordan Decomposition}
\label{alg:BPJE}
\end{algorithm}

If $T^2 = 0$, then the structure produced on iteration $p$ of Algorithm \ref{alg:BPJE}  bears a simple relationship to that of step $p-1$.   Let us denote the $p$th Jordan pair by $(f_p, h_p)$, and the corresponding coproduct structure $\wk^p$.  In the 2-nilpotent regime only one element rotates directly into the coproduct on each iteration, namely $Tf_p$, and one into the product.   Since neither $f_p$ nor $Tf_p$ are modified by subsequent iterations Algorithm \ref{alg:BPJE} (a consequence of the splitting lemma) we may arrange the elements of $\wk^p$ into a tuple of form
\begin{align}
(f_1, \ld, f_p, \ld, T f_p, \ld, Tf_1).  
\label{eq:ordering}
\end{align}
\begin{proposition}  
Array $1(\wk, \wk^{\infty})$ is upper-triangular with respect to  \eqref{eq:ordering}.
\end{proposition}
\begin{proof}
Let $n$ be the dimension of $W$.  With respect to the given ordering, the off-antidiagonal support of $T(\wk^{p-1}, \wk^{p-1})$ concentrates on $[p, n-p] \times [p, n-p]$.  If $\uk^{p-1}$ is the coproduct produced by rotating $Tf_p$ into $\wk^{p-1}$, then the same assertion holds for the support of $T(\uk^{p-1}, \uk^{p-1})$.  Consequently array $1(\wk^{p-1}, \uk^p)$, which may be obtained by replacing the $p$th column of the identity array by the $p$th column of $T(\wk^{p-1}, \wk^{p-1})$, is upper-triangular.
%Since $(\uk^p)^\sharp = (\wk^{p-1})^\sharp - \{(Tf_p)^\sharp\} \cup \{(g_pT)^\sharp\}$,  array $1(\uk^p, \wk^{p-1})$ is obtained by replacing row $p$ of the $p\times p$ identity array by the $p$th row of $T(\wk^{p-1}, \wk^{p-1})$.   Consequently this array is upper-triangular.  Likewise, $1(\uk^p,\wk^p)$
The same holds for $1(\uk^p,\wk^{p+1})$, by a similar argument, and so, too for
%  By a trivial application of M\"obius inversion, the matrix inverse $1(\uk^p, \wk^{-1})$ is also upper-triangular.
$1(\wk^{p-1},\wk^{p}) = 1(\wk^{p-1}, \uk^p) 1(\uk^p,\wk^{p})$.   The desired conclusion follows by a simple induction.
\end{proof}

\section{Filtered Exchange}
\label{sec:filtexch}
%To bring our discussion of biproducts into alignment with the study of minimal bases, recall the canonical isomorphism between the space $\hom{}(\field, W)$ of maps $\field \to W$ and $W$ itself, given by $f \leftrightarrow f(1)$.
%
%
%let us define the \emph{image} of a family of maps $\lk$ into $W$ to be

We define the \emph{closure} of  a collection $\lk$ of maps into $W$ to be  the class $\cl(\lk)$ of all $g$ for which some diagram of form
\[
\xymatrix@R=2.5pc@C=5.5pc@M=.7pc{
\ar[];[dr]_{\oplus \lk}   \ar[rr];[]  \ar[rr];[dr]^g
\bullet && \bullet  \\
& W
}
\]
commutes.

Likewise,  the \emph{coclosure} of a collection $\uk$ of maps out of $W$ is the class $\cl(\uk)$  of all $g$ for which some diagram of form
\[
\xymatrix@R=2.5pc@C=5.5pc@M=.7pc{
\ar[dr];[]^{\times \lk }   \ar[];[rr]  \ar[dr];[rr]_g
\bullet && \bullet  \\
& W
}
\]
commutes.  Coclosures will not enter our discussion directly, however an equivalent dual story may be told in terms of this operations.

Given any family $\Omega$  of maps into $W$,  define $\cl_\Omega(\wk) = \Omega \cap \cl(\wk)$. We say  $\wk$  is \emph{closed}  if $\wk = \cl(\wk)$, and \emph{closed in $\Omega$} if $\wk = \cl_\Omega(\wk)$.  A set is \emph{independent} in $\Omega$ if no $f \in \wk$ lies in the closure of $\wk - \{f\}$.

In the special case where $\Omega = \Hom(\field, W)$, the  closure of $\wk$ in $\Omega$ is the family of  maps $\field \to W$ whose image lies in $\I(\oplus\wk )$.  Consequently, $\wk$ is independent in $\Omega$ iff $\{f(1) : f \in \wk\}$ is independent as a subset of $W$.  Thus  the independent sets of $\Omega$ are  the linearly independent subsets of $\Hom(\field, W)$.   The benefit of this alternate characterization is that it  describes independence in the language of function composition, hence the language of biproducts.   The following, for example, is an elementary consequence of the splitting lemma.

\begin{lemma} 
\label{lem:containedin}  
If $\lk$ is a coproduct structure on $W$ and $\wk \su \lk$, then $g \in \cl(\wk)$ iff $(\lk - \wk)^\sharp g = 0$.
\end{lemma}

%The benefit of this circuitous recharacterization  is that it describes independence in the language of function composition, hence the language of biproducts.   Suppose, for example, that $\lk$ is a coproduct structure on $W$, and $\wk$ any subset of $\lk$.  When does a map $g$ lie in the closure of $\wk$?  By definition this holds when there is some $h$ so that $g = (\oplus_\wk f) \com  h$, but the splitting lemma implies that this condition is equivalent $(\lk - \wk)^\sharp f = 0$, a much easier criterion to check, in practice.

Let us apply this observation to a problem involving minimal bases.   Suppose are given a linear filtration $\fc$ on $\Omega$, and an $\fc$-minimal basis $\lk$.  Under what conditions will a second basis, $\uk$, be minimal as well?

Denote the relation $\{(f,g) \in \aw : \fc(f) \le \fc(g)\}$ by $\sim_\fc$, and recall that a \emph{partial matching} on  $T(\lk, \uk)$ is a  partial matching $\d$ on $R = \Supp(T(\lk, \uk))$ such that $T\left (\d \der \sharp, \d \der \flat \right)$ is invertible.   Every exchange pair $(f,g)$ determines a partial paring $d = \{(f,g)\}$, and while the support of $1(\lk, \lk[\d])$ is formally a relation on $\lk \times \lk[\d]$ we may naturally regard it as subset of $\lk \times \lk$, via the pairing functions.  The relation  $R_\d \su \lk \times \lk$ coincides with  the transitive reflexive transitive closure of 
$$
\Supp(1(\lk, \lk[d]))
$$
under this identification.  A similar interpretation holds for $R^\d$.

\begin{proposition} 
\label{prop:minimalpairing}  
If $R$ is the support of $1(\lk, \uk)$,  then $\uk$ is $\fc$-minimal iff $\sim_\fc$ extends $R_\d$ for some perfect matching $\d$ on $R$.
\end{proposition}
\begin{proof}  
Let us assume for convenience that  $\lk$ and $\uk$ are ordered tuples, and reindex as necessary so that $\fc \com \lk$ and $\fc \com \uk$ are monotone-increasing functions of form 
$$
\{1, \ld, m\} \to \Z.
$$  
The first condition holds iff $\fc \com \lk = \fc \com \uk$.   Lemma \ref{lem:containedin} therefore implies
 %$\fc(g) = \max \{\fc( f) : f \in \lk, \; f^\sharp g \neq 0\}$ for all $g \in \Omega$.   If
$\uk$ is minimal iff the first $n_p$ columns of $1(\lk,\uk)$ have support on the first $n_p$ rows, where $n_p = | \lk_{\fc \le p}|$.  This is equivalent to the condition that the array be block-upper triangular, with diagonal blocks of size $n_p - n_{p-1}$.  As $1(\lk, \uk)$  it  is invertible, these blocks must be, also.  Any perfect matching that draws its elements from the support of these diagonal blocks will extend to $\sim_\fc$.  This establishes one direction.  The converse may be shown similarly.
\end{proof}

Suppose now that $\lk$ and $\uk$ are bases (equivalently, coproduct structures) in $\Hom(\field, W)$ and $\Hom(\field, V)$.  Posit weight functions $\fc$ and $\gc$ on $\lk$ and $\uk$, respectively, and extend these  to the weights on $\Hom(\field, W)$ and $\Hom(\field, V)$ such that
\begin{align*}
\fc_{\le p} = \cl_\Omega(\lk_{\fc \le p}) &&
\gc_{\le p} = \cl_\Omega(\uk_{\gc \le p}).
\end{align*}
We say that   $(f,g)$ is \emph{$\fc$-$\gc$ minimal} if 
$$
(R_d, R^d) \su (\sim_\fc, \sim_\gc),
$$
where $\d = (f,g)$.  Equivalently, $(f,g)$ is $\fc$-$\gc$ minimal  if $\fc(f)$ is the minimum value taken by $\fc$ on the row support of $T(\lk, \{g\})$ and $\gc(g)$ is the maximum value taken by $\gc$ on the column support of $T(\{f\}, \uk)$.   It is vacuous that the exchange pairs in any of the preceding algorithms may be chosen to be minimal, since the  candidate pairs on each iteration are  the nonzero elements in block $N_1$ of a block-diagonal array of form $\diag(N_0,N_1)$.  %Consequently, given minimal bases $\lk$ and $\uk$ the following returns a pair of structures $\lk^\infty$, $\uk^\infty$.
\vspace{.55cm}

\begin{algorithm}
\begin{algorithmic}
\WHILE{ $\ak \times \bk$ contains an exchange pair}
\STATE   fix a minimal exchange pair  $(a,b) \in \ak \times \bk$.%\in \Supp(M)$ with $i \notin I_0$, $j \notin J_0$
\STATE   $(\lk, \uk)  \leftarrow (\lk, \uk)[a,b]$
%\STATE $\uk \leftarrow \uk[a,b]$
%\STATE $L \leftarrow L(M, (i,j))$
%\STATE $U \leftarrow U(L\inv M, (i,j))$
%\STATE  $M \leftarrow LMU$
%\STATE $p \leftarrow p+1$
\STATE $(\ak, \bk) -= (a,b)$
\ENDWHILE
\end{algorithmic}
\caption{LU Decomposition (Filtered)}
\label{alg:BPLUfiltered}
\end{algorithm}

\begin{proposition}  
The bases $\lk^\infty$ and $\uk^\infty$ returned by Algorithm \ref{alg:BPLUfiltered} are $\fc$-minimal and $\gc$-minimal, respectively.
\end{proposition}
\begin{proof}  
Let $R$ denote the support of $1(\lk, \lk^\infty)$.  Since $1(\lk, \lk^\infty)$ is triangular up to permutation, $R$ has a unique perfect matching $\d$.  Relation $\sim_\fc$ extends $R_\d$ by construction, so minimality for $\lk^\infty$  follows by Proposition \ref{prop:minimalpairing}.  Minimality for $\uk^\infty$ may be argued similarly.
\end{proof}

A similar result holds for the Jordan algorithm in the special case where $T^2 = 0$.  As with LU decomposition, we are guaranteed to be able to restrict our selection to minimal exchange pairs thanks to block-diagonal structure of $T(\wk, \wk)$.
\vspace{.55cm}

\begin{algorithm}
\begin{algorithmic}
\WHILE{ $\wk$ has a proper Jordan pair }
\STATE   fix an $(\fc, \fc)$-minimal proper Jordan pair for $\wk$.%\in \Supp(M)$ with $i \notin I_0$, $j \notin J_0$
\STATE   $\wk  \leftarrow \wk[[f,h]]$
%\STATE $\uk \leftarrow \uk[a,b]$
%\STATE $L \leftarrow L(M, (i,j))$
%\STATE $U \leftarrow U(L\inv M, (i,j))$
%\STATE  $M \leftarrow LMU$
%\STATE $p \leftarrow p+1$
\ENDWHILE
\end{algorithmic}
\caption{Jordan Decomposition (Filtered)}
\label{alg:BPJEfiltered}
\end{algorithm}

\begin{proposition}  The structure $\wk^\infty$ returned by Algorithm \ref{alg:BPJEfiltered} is $\fc$-minimal.
\end{proposition}
\begin{proof}  Array $1(\lk^\infty, \lk)$ is a product of factors of form $1(\lk[\d],\lk)$ and $1(\uk,\uk[\d])$.  Relation $\sim_\fc$ extends the induced relations of $1(\lk,\lk[\d])$ and $1(\uk[\d],\uk)$ by minimality, and a trivial application of M\"obius inversion shows that it extends the induced relations of the factors of $1(\lk^\infty, \lk)$, also.  The desired conclusion follows by Proposition \ref{prop:minimalpairing}.
\end{proof}

\section{Block exchange}
\label{sec:blockexch}

%We have seen that Algorithm \ref{alg:BPLU} may be regarded  as combinatorial realization of Gauss-Jordan elimination.  Like that algorithm,  Algorithm \ref{alg:BPLU} has a natural ``block'' form.  In the same way that one pivots on invertible submatrices in block-elimination, one may also rotate \emph{sets} of maps into and out of biproduct structures.  In fact we have already encountered this operation in

Recall that in Section \ref{sec:exchangerelations} we introduced \emph{set exchange}, the natural extension of elementary exchange where entire sets, rather than single elements, rotate into and out of a (co)product.

The analysis of preceding sections carries through with minimal modification for set operations, for instance returning arrays $1(\lk, \lk^\infty)$ and $1(\uk^\infty, \uk)$ that are \emph{block} lower and upper triangular, respectively,  for Algorithm \ref{alg:BPLU}, with respect to the natural  grouping of elements by iteration.   Let us note some special cases.

\subsubsection{Acyclic blocks}

Suppose that block-elimination is carried out exclusively for acyclic pairings.   Such pairings index triangular-up-to-permutation blocks in the associated arrays, so   $1(\lk, \lk^\infty)$ and $1(\uk^\infty, \uk)$ will be triangular up to permutation.  The same holds for the block form of Jordan elimination, though here one must, as usual, invoke Proposition \ref{prop:holdacyclic} in tandem with M\"obius inversion.

\subsubsection{Minimal blocks}

The filtered versions of Algorithms \ref{alg:BPLU} and \ref{alg:BPJE} have natural block-generalizations as well.  Here the minimality requirement for pairs is replaced by the (rather more direct)  criterion that the relations induced by $\d$ on $I\times I$ and $J \times J$ extend to $\sim_\fc$ and $\sim_\gc$, respectively.   Correctness may be argued exactly as for Algorithms \ref{alg:BPLUfiltered} and \ref{alg:BPJEfiltered}.

\subsubsection{Linear blocks}

The two preceding cases coincide when $\fc$ and $\gc$ determine linear orders on $\lk$ and $\uk$.  When such is the case, minimal pairings coincide exactly with acyclic pairings  \emph{by definition}.   When $\fc$ and $\gc$ do not induce linear orders, one may construct modified functions $\fc'$ and $\gc'$ for which $\sim_{\fc'}$ and $\sim_{\gc'}$ are linear orders contained in $\fc$ and $\gc$.  Any bases for these modified functions will be minimal with respect to $\fc$ and $\gc$, by Proposition \ref{prop:minimalpairing}.   In fact, one may pick a different $\fc'$ and $\gc'$  on each iteration, thanks to the transitivity of $\sim_{\fc}$ and $\sim_{\gc}$.

In this linear setting, the minimal pairings have a special structure.  The proof of the following is vacuous.

\begin{lemma}  
If $\sim_\fc$ and $\sim_\gc$ are linear orders, then the set $S$ of all $(\fc,\gc)$-minimal pairs is itself a minimal acyclic pairing.  Every minimal pairing is a subset of $S$, in this case.
\end{lemma}

\begin{remark}  
The structure of minimal pairings for linear $\sim_\fc$ and $\sim_\gc$ is highly natural. It plays a foundational  role in work of M. Kahle on  probabilistic topology \cite{KahleTopology11}, and has been discussed by Carlsson in reference to persistent homology calculations, in personal correspondence.  More recently, the minimal pairs of a linear order have been remarked as ``obvious pairs'' in the work of U. Bauer on fast persistence algorithms, in reference to the computational library \emph{Ripser} \cite{ripserwebsite}.  %That a large number of researchers,  the author included, discovered these structures independently and in so short a period of time speaks to their fundamental simplicity and their closeness ties to combinatorial topology.
\end{remark}

\subsubsection{Graded blocks}
\label{sec:gradblocks}

In the following chapter we introduce the notion of a graded operator on a vector space $\oplus_p C_p$.  To such a space one may associate a relation $u \sim v$ iff $\{u, v\} \su C_p - \{0\}$ for some $p$.   This relation is  transitive, and so falls under the umbrella of Corollary \ref{cor:potchkai}.  It may thus be shown that any variant on the Jordan or LU exchange algorithms will return a graded  basis, provided that the initial (co)product structures are graded.

%\afterpage{\blankpage}

\part{Applications}

\chapter{Efficient Homology Computation}
\label{ch:efficienthomologycomputation}

\section{The linear complex}
\label{sec:lincomp}

A $\Z$-grading on a $\field$-linear space $C$ is a $\Z$-indexed family of linear subspaces $(C_p)_{p \in \Z}$ such that the coproduct map
$$
\oplus_p C_p \to C
$$
induced by the inclusions $C_p \to C$ is an isomorphism.  

\begin{remark}  This definition agrees with that of a grading on the matroid $(C, \ic)$, where $\ic$ is the family  of $\field$-linearly independent subsets of $C$, c.f.\ Example  \ref{ex:lineargrades}.  
\end{remark}

\begin{remark}  In the special case where $C = \oplus_p C_p$, we call $(C_p)$ the \emph{canonical} grading on $C$.
\end{remark}

% $G:\Z \to \lc(C)$ such that the coproduct map $\oplus_p G_p \to C$ induced by the natural inclusions $G_p \to C$ is an isomorphism.  It is customary to write, where $\lc(C)$ is the lattice of $\field$-linear subspaces of $C$, such that
%
%A $\Z$-\emph{grading} on a $\field$-linear vector space $C$ is a $\Z$-grading on the linear matroid engendered by $C$.  Equivalently, a $\Z$-grading is an indexed family of $\field$-linear subspaces $(C_p)_{p \in \Z}$ 
%
%Equivalently, if $\lc(C)$ is the family of linear subspaces of $C$, then a $\Z$-grading is a function $\Z \to \lc(C)$, which is customarily denoted $C$ as well,  definition coAs noted in Example \ref{ex:lineargrades}, this means a function $\Z \to \lc(C)$, complementary $\Z$-indexed family of subspaces in $C$.   It is customary to use the same symbol for a space an its grading function, whence
%$$
%C \cong \oplus_{p } C_p.
%$$   
If $C$ and $D$ are graded spaces, then a map $T: C \to D$ is \emph{graded  of degree $k$} if 
$$
TC_p \su T D_{p+k}
$$ 
for all $p \in \Z$.  A degree -1 endomorphism on $C$ is a \emph{differential} if  $T^2 = 0$.   A \emph{linear  complex} is a pair $(C,\partial)$, where $C$ is a graded vector space and $\partial$ is a differential on $C$.

It is customary to write
\begin{align*}
Z_p = K(\partial) \cap C_p && B_p = I(\partial) \cap C_p && \partial_p = \partial|_{C_p}
\end{align*}
when working with complexes.  We call $Z_p$ the space of \emph{cycles}, $B_p$ the space of \emph{boundaries}, and $\partial_p$ the \emph{$p$-dimensional boundary operator}.  Where context leaves  room for doubt, the first and second of these may be expressed $Z_p(C)$ and $B_p(C)$, respectively.  The   \emph{homology group in dimension $p$}, or simply the \emph{$p$th homology group}  of $C$  is the quotient space 
$$
H_p(C) = Z_p / B_p.
$$

By mathematical synecdoche, one generally denotes $(C,\partial)$ by $C$ alone, the associated differential being understood from context.   For example, a \emph{map} from one complex to another is a degree zero commutator of the respective differentials.  In symbols, this means a map  $T:C \to D$ for which $T \partial = \partial T$.  The differential on the left is that of $C$, and that on the right is that of $D$.  The situation of these maps relative to $T$ leaves no ambiguity as to the intended meaning.

Recall that the quotient of a vector space $W$ by a subspace $U$ is canonically realized as the set of  equivalence classes $[w] = \{w + u : u \in U\}$ equipped with an addition $[w] + [v] = [w + v]$ and  a scalar multiplication $\ak \cdot [w] = [\ak \cdot w]$.  It is an foundational fact of homological algebra that every map of chain complexes induces a map on homology groups
\begin{align*}
T_*: H_p(C) \to H_p(D) 
\end{align*}
determined by the rule
\begin{align*}
 T_*[w] = [Tw].
\end{align*}

\begin{remark}  
If we define $H(C) = K(\partial)/I(\partial)$, then there exists a well-defined linear map  $H_p(C) \to H(C)$, defined by the rule $[v] \mapsto [v]$.    If we identify $H_p(C)$ with its image under this inclusion, then the $\Z$ indexed family of homology groups $H_p(C)$ determines a natural grading on $H(C)$.
\end{remark}

A  \emph{combinatorial simplicial complex} $\xc$ on ground set $V$ is a family  of subsets of $V$ closed under inclusion.  Recall that this means $I \in V$ when $J \in V$ and $I \su J$.  An \emph{ordered} combinatorial simplicial complex is a pair $(\xc, <)$, where $<$ is a linear order on the ground set $V$.   We write $\xc \der p$ for the family of all subsets of $\xc$ of cardinality $p+1$.  

Ordered simplicial complexes provide a useful class of complexes in algebra and topology, by the following construction.   Let $C = \field^\xc$, and for each $I \in \xc$, define 
$$
\chi_I(J) = \dk_{IJ}.
$$
If $C_p$ is the linear span of
$$
\{\chi_I : |I| = p + 1\},
$$
then the family $(C_p)_{p \in \Z}$ is a grading on $C$.   If one arranges the elements of each  $I \in \xc$ into a tuple $(i_1, \ld, i_{|I|})$ such that $i_p < i_q$ when $p < q$, and setting
$$
I_p =  I - \{i_p\},
$$
then it can be shown that the map $\partial: C \to C$ defined  by
\[
\partial(\chi_I) = \sum_{p = 1}^{|I|} (-1)^{p} \chi_{I_p}
\]
is a differential on $C$.  The pair $(C, \partial)$ is the $\field$\emph{-linear chain complex} of $\xc$ with respect to $<$.

\begin{example}  
Let $G = (V, E)$ be any combinatorial graph on vertex set $V = \{1, \ld, m\}$.  The set $E \cup V \cup\{ \emptyset\}$ is a simplicial complex on ground set $V$.  The matrix representation of the associated differential with respect to basis 
$$
\{\chi_{I} : I \in E \cup V \cup \{\emptyset\}\},
$$ 
has block form
\vspace{.5cm}
\[
\arraycolsep=10pt\def\arraystretch{1}
\begin{array}{c|ccc|}
\multicolumn{1}{c}{}&E & V & \multicolumn{1}{c}{\emptyset} \\
\cline{2-4}
E& &  & \\
V&A& & \\
\emptyset&&  &  \\ \cline{2-4}
\end{array}
\vspace{.4cm}
\]
where blank entries denote zero blocks, as per convention, and $A$ is the node-incidence matrix of $G$.
%and $A$ the submatrix of $D$ indexed by $V \times E$, then $A$ is the node incidence matrix of $G$. % If we regard $G$ as a directed graph with the natural direction induced by $<$ on nodes, then when $\field = \R$ matrix $A$ is the oriented node incidence matrix of $G$.
\end{example}

\section{The linear persistence module}
\label{sec:linpersmod}
\renewcommand{\r}{r}
\newcommand{\s}{s}

A $\Z$-graded linear persistence module is a $\Z$-graded map of degree one, augmented by the following data.%   To form a persistence module one augments this data  by the following construction.%We will use the latter for simplicity, but let us at least define the former, so as to illustrate some connections with existing literature.

Let $\field[t]$ denote the  space of polynomials in  variable $t$ with coefficients in ground field $\field$.  %We will write $P$ for a  polynomial  $P(t) = P_0 t^0 + \cd + P_m t^m$.  %Every $\field$-linear operator $T$ on a vector space $W$ determines a linear map $\field[t] \to \Hom(W,W)$, sending  $P$ to  $P(T) = P_0 T^0 + \cd + P_m T^m$.
An \emph{action} of $\field[t]$ on a vector space $W$ is a bilinear mapping  $\mk: \field[t] \times W \to W$ such that
\begin{align*}
\mk(\r, \mk(\s,v))=\mk(\r\s,v)
\end{align*}
for all $\r,\s \in \field[t]$ and all $v \in W$.  One typically writes $\r v$ for $\mk(\r,v)$.  Every action is uniquely determined by its restriction to $\{t\} \times W  \to W$, so actions are in 1-1 correspondence with linear operators on $W$.

A $\field[t]$-\emph{module} is a vector space equipped with an action.  A \emph{graded} module of degree $k$ is a graded vector space on which $t$  acts by degree $k$ map, that is, for which 
$$
t W_p \su W_{p+k}
$$ 
for all $p \in \Z$.  A $\field$-linear \emph{persistence module} is a graded $\field[t]$-module of degree one.

A \emph{submodule} of $W$ is a subspace $V \su W$ for which $t V \su V$.  A submodule is \emph{graded} if the family of intersections $V_p = V \cap W_p$ is a grading on $V$.   If $U$ and $V$ are submodules, we say that the subspace $U + V$ is an (internal) direct sum of $U$ and $V$ if $U \cap V = 0$.  A submodule is \emph{indecomposable} if it cannot be expressed as the internal direct sum of two nontrivial submodules.  It is an elementary fact that the only indecomposable $\field[t]$ modules are those of form $\field[t] \ell= \{rv : r \in \field[t], \; v \in \ell\}$, where $\ell$ is a subspace of dimension zero or one.

The \emph{support} of a graded module $W$ is $\Supp(W) = \{p : W_p \neq 0\}$.  The support of every indecomposable module is  an integer interval.  The following structure theorem is a mainstay of modern applications in topological data analysis.

\begin{proposition} 
\label{prop:structuretheorem} 
Every finite-dimensional persistence module is an internal direct sum of  indecomposable submodules.  If $U_1 \oplus \cd \oplus U_m$ and $V_1 \oplus \cd \oplus V_m$ are any two such sums, then $|\{p: \Supp(U_p) = I\}| = |\{q: \Supp(V_q) = I\}|$ for every nontrivial interval $I$.
\end{proposition}

This result has a simple proof in the language of nilpotent maps on matroids.   The linear operator $T: W \to W$  induced by the action of $t$ is nilpotent, since it is positively graded and $W$ is finite-dimensional.   Recalling that
$$
\orb(v)= \{T^p v : p \ge 0, \; T^p v \neq 0\}
$$ 
is the \emph{orbit} of $v$ under $T$, observe that every orbit {freely generates} an indecomposable submodule.  Conversely, every indecomposable submodule is freely generated by an orbit.  %Consequently, every decomposition into  indecomposables  will engender a family of orbits whose union forms a basis, and every basis composed of orbits determines a decomposition into indecomposables.

For precision, let us say that the \emph{projective class} of a Jordan basis $\cup_p \orb(v_p)$ is the family of all  bases that may be expressed in form $\cup_p\orb(\ak_p v_p)$  for some family of nonzero scalars $\ak$.  Let us further say that a decomposition $W = \oplus_p U_p$ is \emph{proper} if each $U_p$ is indecomposable and nontrivial.  

\begin{lemma} 
\label{lem:jordanclasses} 
Proper decompositions are in 1-1 correspondence with the projective classes of graded Jordan bases.
\end{lemma}

The proof of Lemma \ref{lem:jordanclasses} an elementary exercise in definition checking.  In light of this fact, Proposition \ref{prop:structuretheorem}  follows directly from Proposition \ref{prop:uniquejordan}, which we repeat here for ease of reference.

{
\renewcommand{\thetheorem}{\ref{prop:uniquejordan}}
\begin{proposition}
If $I_1, \ld, I_m$ and $J_1, \ld, J_n$ are the orbits that compose two graded Jordan bases, then there exists a bijection $\fk: \{1, \ld, m\} \to \{1, \ld, n\}$ such that 
$$
\Supp(I_p) = \Supp(J_{\fk(p)})
$$
for each $p$ in $\{1, \ld, m\}$.
\end{proposition}
\addtocounter{theorem}{-1}
} 

\section{Homological Persistence}
\label{sec:hompers}

To each $k \in \Z$ and each  sequence of $\field$-linear complexes and complex maps 
\begin{align}
\cd \lr{} C \der {p-1} \lr{} C \der p \lr{} C \der{p+1} \lr{} \cd  \label{eq:complexsequence}
\end{align}
one may associate  a sequence of linear maps
\begin{align}
\cd \lr{} H_k(C \der {p-1}) \lr{} H_k(C \der p) \lr{} H_k(C \der{p+1}) \lr{} \cd. \label{eq:homologicalpm}
\end{align}
%The sequence $\left ( H_k( C \der p) \right)_{p \in \Z}$ determines a natural grading on $\oplus_{p \in \Z} H_k(C \der p)$, and the maps in \eqref{eq:homologicalpm} determine a $H_k(C \der *)$.
Sequence \eqref{eq:homologicalpm} determines an endomorphism $Q$ on $\oplus_p H_k(C \der p)$, which is graded of degree 1 with respect to the canonical grading.   The associated  module is the \emph{$k$-dimensional homological persistence module} of \eqref{eq:complexsequence}. 

Homological persistence modules play a premier role in the field of topological data analysis, and the primary mode of understanding these objects is by way of  the Jordan bases of $Q$.

Let us focus on the special case where each map $C \der p \to C \der {p+1}$ is an inclusion of complexes, and where $H_k(C \der p)$ vanishes for $p$ outside some integer interval $[0,m]$.   Better still, let us assume that $C \der p$ vanishes for $p < 0$, and
\begin{align}
\label{eq:vanishesoverm}
Z_k(C \der p) = B_k(C \der p) = B_k(C \der m)
\end{align}
for $p> m$.  The family of complexes that meet these criteria includes most of those found  in modern practice, either directly or  by minor technical modifications.

Let $\zc_p$ and $\bc_p$ denote the space of degree-$k$ cycles and boundaries, respectively,  in $C \der p$.   These spaces determine filtrations 
\begin{align*}
\zc = (\zc_0, \ld, \zc_m)  && \bc = (\bc_1, \ld, \bc_m)
\end{align*} 
which we may naturally regard as integer-valued functions on 
$$
V = Z_k(C \der m).
$$  
The space  $H_k(C \der p)$ is the linear quotient 
$$
 V_p = \zc_p / \bc_p = H_k(C \der p)
 $$ 
 and $Q$ is the induced map on $\oplus_p V_p$.   Consequently Proposition \ref{prop:gradedquotientbases}, which we repeat here for ease of reference, carries through without modification.

% in the introduction
{
\renewcommand{\thetheorem}{\ref{prop:gradedquotientbases}}
\begin{proposition}
The  graded Jordan bases of $Q$ are the orbits of the $\zc$-$\bc$ minimal bases of $V$.
\end{proposition}
\addtocounter{theorem}{-1}
} % note: these braces are here to take advantage of LaTeX scoping ...
  % \thetheorem is returned to its rightful definition outside of this group

This observation entails a significant reduction in complexity, as the space $\oplus_p V_p$ is in general many orders of magnitude larger (by dimension) than $V$.  Moreover, a simple means of computing $\zc$-$\bc$ minimal bases is readily available.  If we let $\fc$ denote the integer-valued function on $C \der m$ whose $p$th sublevel set is the subspace $C \der p$, then any $\fc$-minimal Jordan basis of the differential on $C \der m$ will necessarily generate $\zc$ and $\bc$.  Any variation on Algorithm \ref{alg:BPJEfiltered} may be used to obtain such a basis.

\section{Optimizations}
\label{sec:ops}

The grand challenge in a preponderance of data driven homological persistence computations is combinatorial explosion in the dimension of the input.   As in the Neuroscience applications of \cite{GPC+Clique15} alluded to in Chapter \ref{ch:intro}, complexes with many trillions of dimensions frequently derive from modest starting data.  Input reduction and the avoidance of fill in sparse matrix operations therefore bear directly on efficient computation.

\subsection{Related work}

The  \emph{standard algorithm} to compute persistent homology was introduced for coefficients in the two element field by Edelsbrunner, Letscher, and Zomorodian in \cite{ELZTopological02}.  The adaptation of this algorithm for arbitrary field coefficients was presented by Carlsson and Zomorodian in \cite{ZCComputing05}.  The standard algorithm is known to have worst case cubic complexity, a bound that was shown to be sharp by Morozov in \cite{MorozovPersistence05}.  Under certain sparsity conditions, the complexity of this algorithm is less than cubic.  An algorithm by Milosavljevi\'c, Morozov, and Skraba has been shown to perform the same computation in $O(n^{\omega})$ time, where $\omega$ is the matrix-multiplication coefficient \cite{MMSZigzag11}.   

A number of closely related algorithms share the cubic worst-case bound while demonstrating dramatic improvements in performance empirically.  These include the \emph{twist} algorithm of Chen and Kerber \cite{CKPersistent11}, and the \emph{dual} algorithm of de Silva, Morzov, and Vejdemo-Johansson \cite{SMVDualities11,SMVPersistent11}.  Some parallel algorithms for shared memory systems include the \emph{spectral sequence} algorithm \cite{EHComputational10},  and the \emph{chunk} algorithm \cite{BKRClear14}.   Algorithms for distributed computation include the \emph{distributed algorithm} \cite{Bauer:2014:DCP:2790174.2790178} and the spectral sequence algorithm of Lipsky, Skraba, and Vejdemo-Johansson \cite{lipsky2011spectral}.  

The \emph{multifield} algorithm is a sequential algorithm that computes persistent homology over multiple fields simultaneously \cite{boissonnat2014computing}.  Efficient algorithms for multidimensional persistence have been described by Lesnick and Wright in \cite{2015arXiv151200180L}.  The methods of nested dissection and simplicial collapse yielded highly efficient algorithms for spaces generated from finite metric data, e.g.\ \cite{Kerber:2016:PHN:2884435.2884521, 2016arXiv160907517D}.    Discrete Morse theory has yielded remarkable advances in speed and memory efficiency, for instance in algorithms of Mrozek and Batko \cite{MBcoreduction09}, of D\l{}otko and Wagner \cite{DWComputing12}, and of Mischaikow and Nanda \cite{HMM+Efficiency10,MNMorse13,nandathesis}.  

Efficient streaming algorithms have recently been introduced by Kerber and Schriber \cite{Kerber2017BarcodesOT}.   U.\ Bauer has recently introduced an algorithm that takes special advantage of the natural compression of Vietoris-Rips complexes vis-a-vis the associated distance matrix, which has yielded tremendous  improvements in both time and memory performance \cite{ripserwebsite}.

While this list is by no means comprehensive, it reflects the breadth and variety of work in this field.  A number of  efficient implementations are currently available, e.g. 
\cite{
ripserwebsite,
dipha,
Bauer2014,
DBLP:journals/corr/BubenikD15,
2016arXiv160907517D,
dlotckopltoolbox,
fasyrtda,
2014arXiv1411.1830F,
lesnickrivet,
maria2014gudhi,
morozovdionysus,
nandaperseus,
deysimpers,
deygicomplex,
perryplex,
TVAjavaPlex12}.

Work leveraging the matroid-theoretic structure of persistence modules has also been conducted.  In \cite{chenfreedman2010}, for example, Chen and Freedman exploit vector matroids to compute natural generators in a greedy manner.

\subsection{Basis selection}

Suppose that $T$ is a graded map of degree one on $W$, and fix any graded basis (not necessarily Jordan).  If the elements of this basis are arranged in ascending order by  grade, then the associated matrix representation will have block form
\vspace{.5cm}
\[
\arraycolsep=7pt\def\arraystretch{.7}
\left(
\begin{array}{cccccc}
0 & 0  & 0 & \cd & 0 &0 \\
* & 0 & 0 & \cd & 0 &0\\
0 & * & 0  & \cd & 0 &0\\
\vdots & \vdots & \vdots & \ddots & \vdots  & \vdots\\
0 & 0 & 0 & \cd & 0 & 0\\
0 & 0 & 0 & \cd & * & 0
\end{array}
\right)
\vspace{.4cm}
\]
with respect to the partition of rows and columns by grade.

Let us denote this array by $A$.  If $T^2 = 0$, then the change effected on $A$ by a row-first Jordan pivot on element $(i,j)$ of block $(p+1,p)$ admits a simple description: block $(p+1,p)$ will be replaced by the array formed by a row and column clearing operation on $(i,j)$, and in $A$ column $i$ and row $j$ will be cleared.

This observation suggests that  Jordan pivots may be understood, at least in  part,  in terms of standard row and column operations.  The influence of such operations on sparsity, in turn, relates closely to that of Gauss-Jordan elimination.  If $M$ is an $m \times n$  array on which sequential row-first pivot operations are performed on elements $(1,1), \ld, (p,p)$, then block $(p+1, \ld, m, p+1, \ld, n)$ of the resulting array will be identical to that of the array produced by row-only pivot operations on the same sequence of elements.

Let us briefly recall the elements of \S\ref{sec:circuits} regarding matrix representations and matroid circuits.  If $M \in \field^{I \times J}$ has form $[\; 1 \; | \; * \; ]$, where the lefthand block is an $I \times I$ identity matrix, then the \emph{fundamental circuit} of column  $j$ with respect to basis $I$ in the column space of $M$ is  $r \cup \{j\}$, there $r$ is the row support of column $j$.  In particular, the sparsity pattern of $M$ is entirely determined by the fundamental circuits of $I$.

It is in general difficult to model and control the fill pattern of a sparse matrices during factorization.  It has been shown, for example, that the problem of minimizing fill in Cholesky factorization of a positive definite array is NP hard \cite{yannakakis1981}.  Optimizations that seek to reduce fill must therefore rely on specialized structures particular to a restricted class of inputs. The class of inputs on which we focus are the boundary operators of cellular chain complexes, and the structure that we seek to leverage is the topology of the underlying topological space.

Even in this restricted setting, it is difficult to formulate provably efficient models for optimization.   Input to cellular homology computations is of an highly  complex nature, and including  highly regular, structured, and symmetric objects from idealized mathematical models and highly irregular, noisy, low-grade data sets  from all branches sciences and engineering .   No sound corpus of benchmarks exists in either case, nor do there exist broadly effective random models.    Efforts in each of these directions have yielded a number of important results over the past decade,  e.g. \cite{KahleTopology11,OPT+roadmap}, however none may be considered broadly representative of either mathematical or scientific data.  A principled approach to optimization, therefore, must rely on one or more simplifying assumptions.   %A principled approach to optimization, therefore, must rely in some measure on experience and observation.  
 
% The value of the matroid theoretic perspective in this setting to clarify which features to attend and how to analyze them.

The assumption we propose in the current work is a close in spirit to that which informs spectral graph theory.  Recall that spectral graph embedding formalizes the notion of ``smooth variation'' for a real-valued function on the vertices of a graph by means of the spectral decomposition of the  graph Laplacian.   A function is considered to be ``smooth'' if it may be expressed as a linear combination of eigenfunctions with eigenvalues  close to zero.  This formal notion of smoothness  has been observed to agree with subjective evaluations of smooth variation, and this correspondence, while imperfect, has proved sufficiently robust to support tremendous applications algebraic graph theory, both pure and applied.      A number of variations on this  approach have been employed in practice, using weighted or unweighted, normalized or unnormalized Laplacians, and while significant differences are observed in the performance of these variations across application domains, in the main, a consistent relationship between input and output may be seen throughout.

The model we propose  seeks to maximize another subjective quantity,  ``normality'' relative to a putative surface or submanifold.  Formally  this notion is no more meaningful than smoothness on a graph, since neither  graphs nor a cell complexes have smooth structures, and, like smoothness, it may be formalized in a  variety of different ways.

Why is normality to be desired?  Suppose $G_1$ and $G_2$ are two disjoint copies of the  cycle graph on vertex set $\{1, \ld, 12\}$, whose edges are  the unordered pairs 
\begin{align*}
\left \{\{p, p+1\} : 1 \le p \le 11 \right \} \cup \{1,12\}.
\end{align*} 
 Let $G$ be the graph obtained from the union of  $G_1$ and $G_2$ by adding an edge from vertex $p$ in $G_1$ to vertex $p$ in $G_2$, for all $p$.  The result may be visualized as a pair of concentric circles in the plane, with 12 spokes connecting the inner to the outer.   Heuristically, we  regard the spoke edges to be normal to some putative underlying manifold, and the remaining edges to be tangential.

Let $T_1$ be a spanning tree in $G$ composed of all the spoke edges plus every edge from the inner circle, save one.  Let $T_2$ be a tree  composed of the spoke connecting $12$ to $12$, plus every tangential edge except the two copies of $\{1,12\}$.   In our heuristic formulation, $T_1$ is relatively normal (having a large number of normal edges) and $T_2$  relatively tangential (having a larger number of tangential edges).

How do the fundamental circuits of these two bases compare?   Those of the normal basis are relatively small.  With  two exceptions, each fundamental circuit has cardinality four.  Among the two exceptional circuits, the larger  has size 14.  The circuits of the tangential basis are much larger, in general.  That determined by the spoke connecting $1$ to $1$, for example,  will have size 24, and in total the cardinalities of the fundamental circuits of  $T_2$ sum to $178$, while those of $T_1$ sum to just $70$.

The phenomenon described by this example is observable in practice.  In Figure \ref{fig:spanningtrees} are displayed two spanning trees for a graph $G = (V, E)$, where $V$ is a sample of 200 points drawn with uniform noise from the unit circle in $\R^2$, and $E = \{ \{u,v\}  \su V : \nn{u-v} \le 0.5\}$.   The basis on the left is obtained by Gauss-Jordan elimination on the rows of a permuted node-incidence matrix of $G$ whose columns have been  arranged in increasing order with respect to $w$,  the function that assigns to each simplex (in this case, edge) the volume of its convex hull in the normalized spectral embedding of the underlying graph. The basis on the right was likewise obtained by Gauss-Jordan elimination on a permuted node-incidence matrix, this one with columns ordered by $-w$.   It may be shown that these orders favor edges that lie normal and tangent, respectively, to the unit circle in $\R^2$.

Two observations are immediate upon inspection.  First, each basis appears qualitatively similar.  This fact reflects a loss of information incurred by the removal of all metric data relating pairs points in $V$, save for the adjacency information encoded by $E$, when one passes from the point cloud to $G$.  Second, the sparsity structures of the associated fundamental circuits are dramatically different, with a net difference  approaching a full order of magnitude in total number of edges.   This phenomenon is observably robust; we find that a similar relationship between sparsity and normality holds consistently across a wide range of examples and  a wide variety of formal metrics for normality.    % and while in practice we have found several highly efficient alternatives, the spectral approach serves to illustrate a clear parallel between the combinatorial and geometric features relevant to efficient homology computation and those common to computational geometry.

\begin{figure}
    \centering
    \begin{subfigure}[b]{0.4\textwidth}
        \includegraphics[width=\textwidth]{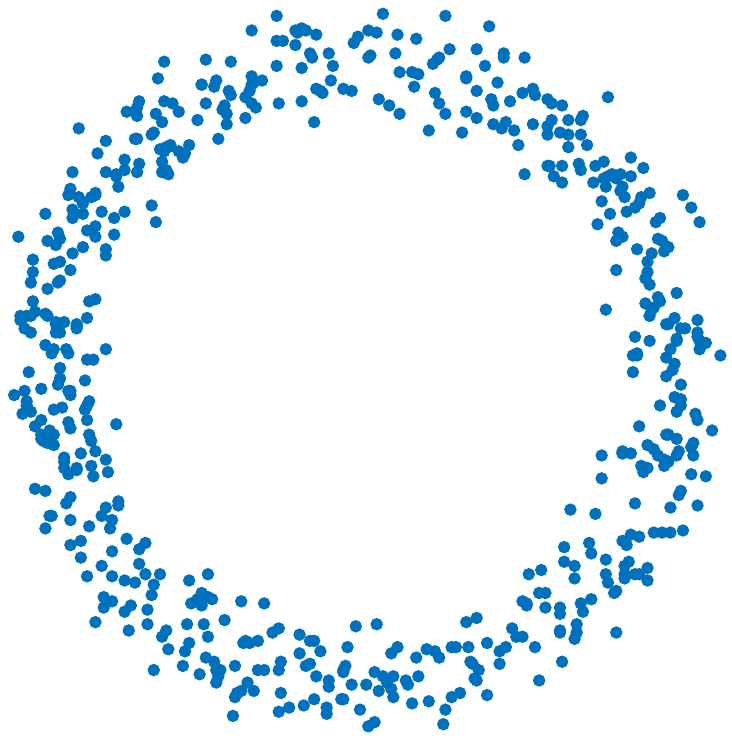}
        \label{fig:gull}
    \end{subfigure} \vspace{.5cm}
    \caption{A population of 600 points sampled with noise from the unit circle in the Euclidean plane.  Noise vectors were drawn from the uniform distribution on the disk of radius 0.2, centered at the origin.}\label{fig:pcloud}
\end{figure}

\begin{figure}
    \centering
    \begin{subfigure}[b]{0.4\textwidth}
        \includegraphics[width=\textwidth]{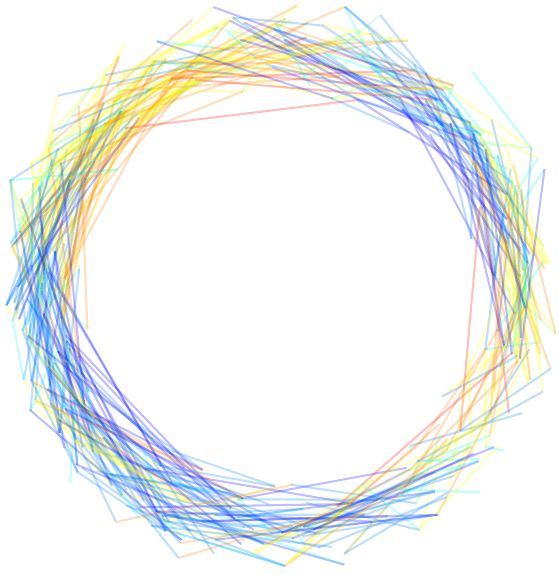}
        \vspace{0.35cm}
        \caption{A $w$ minimal spanning tree.}
        \label{fig:tiger}
    \end{subfigure} \\ \vspace{1.5cm}
    ~ %add desired spacing between images, e. g. ~, \quad, \qquad, \hfill etc.
    \begin{subfigure}[b]{0.4\textwidth}
        \includegraphics[width=\textwidth]{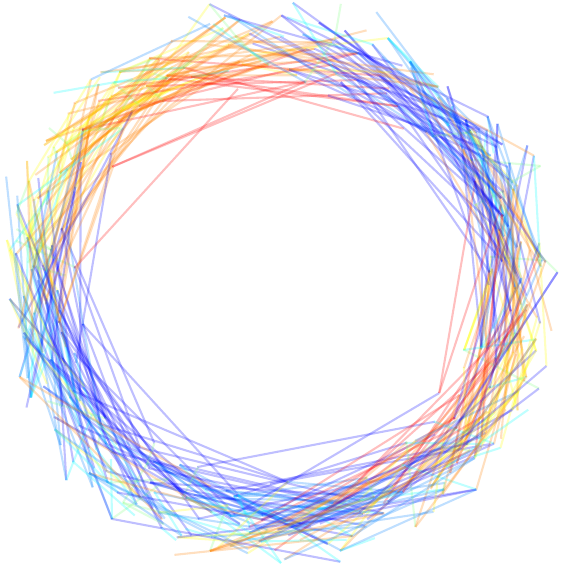}
        \vspace{0.35cm}
        \caption{A $-w$ minimal spanning tree.}
        \label{fig:mouse}
    \end{subfigure} \vspace{1cm}
    \caption{Spanning trees in the one-skeleton of the Vieotris-Rips complex with scale parameter 0.5 generated by the point cloud in Figure \ref{fig:pcloud}.  The cardinalities of the fundamental circuits determined by the upper basis sum to $7.3 \times 10^5$, with a median circuit length of 11 edges.  The cardinalities of the fundamental circuits determined by the lower basis sum to $4.9 \times 10^6$, with a median  length of 52 edges.   }\label{fig:spanningtrees}
\end{figure}

\begin{figure}
    \centering
	\hspace{.2cm}
    \begin{subfigure}[b]{0.49\textwidth}
        \includegraphics[width=\textwidth]{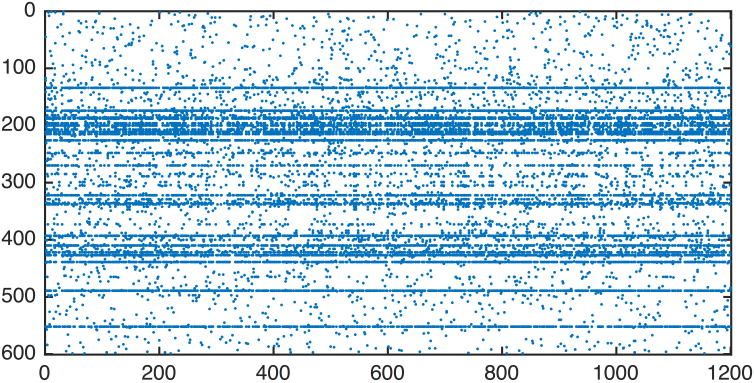}
        \vspace{0.35cm}
        \caption{A $w$ minimal spanning tree.}
        \label{fig:sparsemat}
    \end{subfigure} \\ \vspace{1.5cm}
    ~ %add desired spacing between images, e. g. ~, \quad, \qquad, \hfill etc.
    \begin{subfigure}[b]{0.49\textwidth}
        \includegraphics[width=\textwidth]{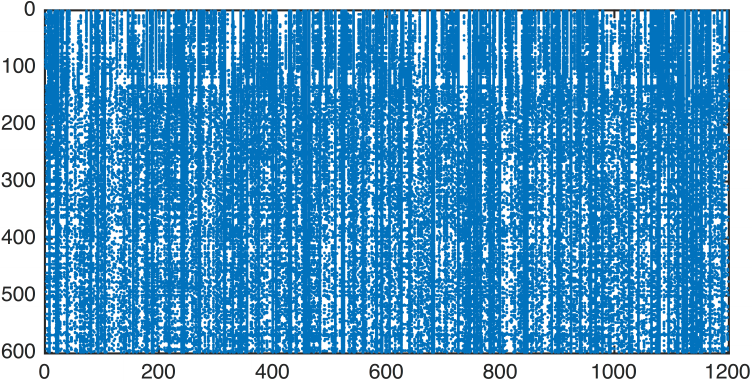}
        \vspace{0.35cm}
        \caption{A $-w$ minimal spanning tree.}
        \label{fig:densemat}
    \end{subfigure} \vspace{1cm}
    \caption{ (a) Sparsity pattern of a subsample of 1200 columns selected at random from the row-reduced node incidence array determined by  spanning tree (a) in Figure \ref{fig:spanningtrees}.  The full array has a total of $7.3\times10^5$ nonzero entries.    (b) Sparsity pattern of a subsample of 1200 columns selected at random from the row-reduced node incidence array determined by the spanning tree (b) in Figure \ref{fig:spanningtrees}.   The full array has a total of $4.9\times10^6$ nonzero entries.  }\label{fig:animals}
\end{figure}
\subsection{Input reduction}

In a majority of the regimes where persistence is currently computed,  only modules of relatively small dimension (10 and below) find application.  It is therefore standard to delete (or never to construct) the rows and columns associated to cells of dimension greater than $1 + p$, where $p$ is the greatest dimension of interest.  Moreover, as the basis vectors for $C_{p+1}$ bare only indirectly on the structure of $H_{p}(C)$,  it is generally advantageous to avoid storing them in memory.

One means to do so is to discard  cells of dimension $p+1$ after they have been matched with cells of dimension $p+2$ in an acyclic pairing.   In the special case of Vietoris-Rips complexes, such pairings may be determined before construction of the matrix representation for the differential operator has begun, thus avoiding construction of much of $C_{p+1}$ altogether.    This strategy is employed in the \emph{Eirene} library  by means of linear ordering of simplices, as described in \S\ref{sec:blockexch}.  

\section{Benchmarks}
\label{sec:benchmarks}

An instance of Algorithm \ref{alg:BPJEfiltered} incorporating the optimizations described above has been implemented in the \emph{Eirene}  library for homological persistence \cite{EIRENE}.   Experiments were conducted on a collection of sample spaces published recently in \cite{OPT+roadmap}.

The libraries in this publication were tested on both a cluster and a shared-memory system.  The cluster is  a Dell Sandybridge cluster, with 1728 (i.e.\ $108 \times 16$) cores of $2.0$GHz  Xeon SandyBridge,  RAM of 64 GiB in 80 nodes and RAM of 128 GiB in 4 nodes, and a scratch disk of 20 TB. It runs the operating system (OS) Red Hat Enterprise Linux 6. The shared-memory system is an IBM System x3550 M4 server with 16 (i.e., $2 \times 8$) cores of 3.3GHz, RAM of 768 GB, and storage of 3 TB. It runs the OS Ubuntu 14.04.01.11.   Results for eleg, Klein, HIV, drag 2, and random are reported for the cluster.  Results for fract r are reported for the shared memory system.  Detailed results for both systems may be found in the original publication.

Experiments for \emph{Eirene} were conducted on a personal computer with Intel Core i7 processor at 2.3GHz, with 4 cores, 6MB of L3 Cache, and 16GB of RAM.  Each core has 256 KB of L2 Cache.   Results for time and memory performance are reported in Tables \ref{tab:time} and \ref{tab:mem}.

\begin{table}
\begin{center}
\resizebox{\columnwidth}{!}{%
\begin{tabular}{ lllllll}
\hline
Data set  & \tb{eleg}  & \tb{Klein} & \tb{HIV} & \tb{drag 2} & \tb{random} & \tb{fract r}  \\ \hline
Size of complex &  $4.4\times10^6$ &$1.1\times10^7$ &$2.1\times10^8$&$1.3\times 10^9$ & $3.1\times 10^9$ & $2.8\times 10^9$ \\
Max. dim. & 2 &2 & 2 &2 & 8 & 3\\ \hline
\textsc{javaPlex} (st)		& 84 	&747 	& - 	&-		&-&-\\
\textsc{Dionysus} (st) 	&474 	& 1830 	&- 	&-		&-&-\\
\textsc{DIPHA} (st) 	&6 		&90		&1631 &142559&-&-\\
\textsc{Perseus} (st)  	& 543 	&1978	& - 	&- &-&-\\
\textsc{Dionysus} (d)  	& 513 	&145	& - 	&- &-&572764\\
\textsc{DIPHA} (d)  	& 4 		&6		&81	&2358&5096&3923 \\
\textsc{Gudhi}   		& 36 	&89		&1798&14368&-&4590 \\
\textsc{ripser}   		& 1 		&1		& 2	&6 & 348&1517\\
\textsc{Eirene}   		& 2 		&10		&193&138 & 16 & 63\\ \hline
\end{tabular}
}
\end{center}
\vspace{0.5cm}
\caption{Wall-time in seconds.  Comparison results for eleg, Klein, HIV, drag 2, and random are reported for computation on the cluster.  Results for fract r are reported for computation on the shared memory system. }
\label{tab:time}
\end{table}

\begin{table}
\begin{center}
\resizebox{\columnwidth}{!}{%
\begin{tabular}{ lllllll}
\hline
Data set  & \tb{eleg}  & \tb{Klein} & \tb{HIV} & \tb{drag 2} & \tb{random} & \tb{fract r}$^\dag$  \\ \hline
Size of complex &  $4.4\times10^6$ &$1.1\times10^7$ &$2.1\times10^8$&$1.3\times 10^9$ & $3.1\times 10^9$ & $2.8\times 10^9$ \\
Max. dim. & 2 &2 & 2 &2 & 8 & 3\\ \hline
\textsc{javaPlex} (st)		& $<5$ 	&$<15$ 	& $>64$ 	&$>64$	&$>64$& $>700$\\
\textsc{Dionysus} (st) 	&1.3 	& 11.6 	&- 		&-		&-		&-\\
\textsc{DIPHA} (st) 	&0.1 	& 0.2	&2.7 	&4.9	&-		&-\\
\textsc{Perseus} (st)  	& 5.1 	&12.7	& - 		&- 		&-		&-\\
\textsc{Dionysus} (d)  	& 0.5 	&1.1	& - 		&- 		&-		&268.5\\
\textsc{DIPHA} (d)  	& 0.1 	&0.2	&1.8	&13.8	&9.6	&276.1 \\
\textsc{Gudhi}   		& 0.2 	&0.5	&8.5	&62.8	&-		&134.8 \\
\textsc{ripser}   		& 0.007 	&0.02	& 0.06	&0.2 	&24.7	&155\\
\textsc{Eirene}   		& 0.36 	&0.19	&0.24	&2.61 	&0.63 	& 3.7\\ \hline
\end{tabular}
}
\end{center}
\vspace{0.5cm}
\caption{Max Heap in GB.  Comparison results for eleg, Klein, HIV, drag 2, and random are reported for computation on the cluster.  Results for fract r are reported for computation on the shared memory system.}
\label{tab:mem}
\end{table}

%\afterpage{\blankpage}

\chapter{Morse Theory}
\label{ch:morsetheory}

\emph{Morse theory} refers to a number of theories  that relate the topology of a space to the critical points of a function on that space.    The seminal work in this field is attributed to Marston Morse, with applications  in the study of geodesics on a Riemannian manifold: see \cite{MilnorMorse63,MilnorDifferential11}.

Each generation refashions Morse theory anew. Classical Morse theory specializes to smooth manifolds and smooth functions with nondegenerate Hessians: so strong a hold did this \emph{Morse condition} exert that relaxation came but slowly. \emph{Morse-Bott theory} permits smooth functions whose critical sets are closed submanifolds, and whose Hessians are non-degenerate in the normal direction.   R. Bott used this theory in his original proof of the \emph{Bott periodicity theorem} \cite{bott1982}. 
% R. Bott, Lectures on Morse theory, old and new, Bull. Amer. Math. Soc. (N.S.) Volume 7, Number 2 (1982), 331-358.
\emph{Stratified Morse theory}, initiated by Goresky and MacPherson, extends this theory to the more general domain of stratified spaces \cite{GMStratified88}, with phenomenal applications in algebraic geometry.  Basic ties to dynamical systems  came through Thom, Smale, and others, reaching their zenith in the ultimately general approach of C. Conley in his eponymous index theory \cite{ConleyIsolated78}. Since then, Morse theory has been a key ingredient in fields as far flung as knot theory and symplectic topology \cite{MSIntroduction98}. E.\ Witten rediscovered and repackaged elements of this theory in the language of quantum field theory via deformations in the function-space attached to a manifold by Hodge theory \cite{witten1982}.
% E. Witten, Supersymmetry and Morse theory, J. Diff. Geom. 17 (1982) 661-692.

In the late 1990s, Forman \cite{FormanMorse98}, following earlier work of Banchoff  \cite{BanchoffCritical70}, 
% T. Banchoff, Critical Points and Curvature for Embedded Polyhedral Surfaces,  The American Mathematical Monthly, Vol. 77, No. 5 (May, 1970), pp. 475-485.
related the ideas of smooth Morse theory to cellular complexes.  The resulting Discrete Morse Theory [DMT] has been successfully applied to a wide variety of setting. Batzies and Welker have used discrete Morse theory in an algebraic setting to construct minimal resolutions of generic and shellable monomial ideals \cite{BWDiscrete00,jollenbeck2009minimal}. Farley and Sabalka used DMT to characterize the configuration spaces of robots constrained to graphs \cite{farley2005discrete}. 
% D. Farley and L. Sabalka, Discrete Morse theory and graph braid groups, Alg. Geom. Topol. 5 (2005) 1075-1109 
Sk\"oldberg  has developed this theory in an abstract algebraic setting \cite{SkoeldbergMorse06}.  Recently, Curry, Ghrist, and Nanda have described the close relationship between discrete and algebraic Morse theories and applications on cellular sheaves \cite{Curry:2016:DMT:2979741.2979751}.

This list constitutes only a fraction of the principle branches of modern Morse theory, and we will not attempt a comprehensive description of ties between the observations made in the following sections with the  field as a hold.  Rather we limit the scope to two branches initiated by Forman and Witten, discrete Morse theory for algebraic complexes and discrete Morse-Witten theory for cell complexes.   The former has proved immensely productive in combinatorial topology and applications.  It is fundamental to modern methods of homology computation, as we will see, and has been beautifully treated in a number of publications and books.  The latter models after a program initiated by  Smale and Witten which remains highly active in modern math and physics, for example in the study of supersymmetry.

We briefly sketch  some ideas from these works before discussing the main results.

\section{Smooth Morse Theory}
This summary outlines the exposition of Bott \cite{bott1982} in recalling  some observations of Thom, Witten, and Smale on a closed manifold $M$ with Riemann structure $g$.  Assume a smooth function $f: M \to \R$.  A \emph{critical point} of $f$ is a point $p \in M$ for which the gradient
\[
\nabla f_p = 0,
\]
where $\nabla f$ is the gradient of $f$ with respect to $g$.  A critical point is \emph{nondegenerate} if  the Hessian
\[
\H_p f =  \left(  { \frac{\partial^2 f}{\partial x_i \partial x_j}}  \right)_{ij} 
\]
is nonsingular for some (equivalently) every local coordinate system $(x_i)$.     To every nondegenerate critical point we associate an \emph{index} $\lk_p$, the number of negative eigenvalues of $\H_p f$, counting multiplicity.   We say $f$ is \emph{Morse} if it has  finitely many critical points, none degenerate.

Every point $q$ that is not a critical point of $f$ lies on a unique 1-dimensional integral manifold $X_q$ of $\nabla f$ which will start (limit in backwards time) at some critical point $p$ and end (limit in forwards time) at another.   This integral manifold is a \emph{heteroclinic} or \emph{connecting} orbit. (The use of \emph{instanton} among physicists is an unfortunate obfuscation.)  
% THIS LANGUAGE, LIKE THE PHYSICISTS IT COMES FROM, IS FUCKING STUPID.  
%The language here is motivated by the observation that if one parametrizes $X_q$ by a curve $u$ satisfying
%\begin{align*}
%\frac{d u(t)}{dt}  = - \nabla f(u(t))  && u(0) = q
%\end{align*}
%then $u$ will be defined on all of $\R$ and $\lim_{t \to -\infty} u(t) = p$, the initial point of the trajectory, and $\lim_{t \to + \infty} u(t) = r$, the final point.   If we mark an arbitrary border between $p$ and $q$  along $u(-\infty, 0]$, say, $u(t_0)$, then in the first half of its existence $u$ will spend an infinite amount of time traversing $u(-\infty, t_0]$, and  only a finite interval in $u[t_0, 0]$.  The story is similar for the second half of its life.   Thus $u$ passes only for an instant through the vicinity of $q$.

The union of all connecting orbits having $p$ as an initial point, which we denote $W_p$, is a cell of dimension $\lk_p$.  We call this the ``descending cell'' through $p$.   The descending cells form a pairwise-disjoint partition $M$, however this partition may  be quite complicated in general.   It was an idea of Smale to introduce the following transversality condition to simplify their behavior.   Let us denote the descending cell of $-f$ for critical point $p$ by $W'_p$.  We call this the \emph{ascendening} cell of $f$ at $p$, and say $\nabla f$ \emph{transversal} if the ascending and descending cells of $f$ meet in the most generic way possible, that is, if at any  $q \in W_p \cap W_r'$ the tangent spaces $T_qW_p$ and $T_q W'_r$ span the full tangent space $T_q M$.  In this case $f$ is \emph{Morse-Smale}.

\begin{remark}  By \emph{modularity}, the transversality condition  equates to the criterion that
\[
\dim(T_qW_p \cap T_qW_r) = \lk_p - \lk_r.
\]
Thus the number of connecting orbits that join points $p$ and $q$ is finite when $\lk_p = \lk_r +1$.   As remarked in \ref{rmrk:modulartransversetame}, this is a concrete expression of modularity at the foundations of geometric topology.
%  Each $q \in W_p \cap W'_r$ has an exact seqeunce
%$$
%0 \to T_q X_q \to T_q W_p \oplus T_q W_r' \to T_q M \to 0.
%$$
%since the tangent to the trajectory of $\nabla f$ through $q$ lies in both $T_q$ and
%$$
%\dim(W_p \cap W_q') = \lk_p - \lk_q + 1
%$$
%so that the connecting orbits jointing two critical points whose indices differ by 1 is finite.  Here is a concrete instance of the tameness conditions referenced in Remark  \ref{rmrk:modulartransversetame} in connection with modularity for set families in a matroid.
\end{remark}

We may associate a natural chain complex to a Morse-Smale function $f$ as follows.  Orient each $T_q X_q$ by $- \nabla f_x$, and fix arbitrary orientations for descending cells $W_p$.  If we assume, for simplicity, that $M$ is oriented, then this assignment  determines  orientations for the ascending cells as well.   When $\lk_p =  \lk_r + 1$, we may assign to each connecting orbit $\ck \su W_p \cap W'_r$ a value of $\pm 1$ according to whether the exact sequence
\[
0 \longrightarrow T_q X_q \longrightarrow T_q W_p \oplus T_q W'_r \longrightarrow T_q M \longrightarrow 0
\]
preserves orientation.  Let us denote this value $e(\ck)$.   Let
\[
C^f(M) = \Z^{\{[W_p]\}_p}
\]
be the free group over $\Z$ generated by the descending cells, and grade $C^f(M)$ by the dimension of $W_p$.    We may then define a degree $-1$ operator $\partial$ by counting connecting orbits with sign:
\[
\partial[W_p] =  \sum_{\dim(W_q) = \dim(W_p)-1} \; \;\sum_{\ck \su W_p \cap W_q'} \; \; e(\ck) [W_q].
\]
The following is a consequence of work by Smale.  The complex $C^f(M)$ has been called by many names (various combinations of Thom, Smale,  and Witten). This is Morse Theory; it is the Morse complex.
\begin{theorem}  
The operator $\partial$ is a differential on $C^f(M)$.  With respect to this differential,  $H(M, \Z) \cong H(C^f(M))$.
\end{theorem}

A related construction, obtained by very different means, has been described by Witten.   One considers the de Rham complex $\aw^*$, given by the sequence
\[
 \aw^0 \longrightarrow \aw^1 \longrightarrow \cd  \longrightarrow \aw^n.
\]
with codifferential $d$.  The Riemannian metric determines an adjoint to $d$, and the spectral decomposition of the corresponding Laplacian
\[
\ad = d d^* + d^* d
\]
separates $\aw^q$ into a direct sum of finite-dimensional eigenspaces 
\[
\aw_\lk^q = \{\varphi \in \aw^q : \ad \varphi = \lk \varphi\}.
\]
If we denote the restriction of $d$ to $\aw^*_\lk = \oplus_q \aw_\lk^q$, then the Hodge theory provides that 
\[
0 \longrightarrow \aw^q_\lk \lr{d_\lk} \aw_\lk^{q+1} \lr{} \cd \lr{}\aw_\lk^n \lr{} 0
\]
is exact for $\lk > 0$, and that
$$
 H^q(M) \cong  \aw^q_0 
$$
for all $q$.  As an immediate consequence, for each $a > 0$ the finite-dimensional complex 
$$
\aw_{a}^* = \oplus_{\lk \le a} \aw^*_\lk
$$
has $H^*(M)$ as its cohomology.  To this construct Witten associated a family of operators
\[
d_t = e^{-t f} \com d \com e^{t f},
\]
parametrized by $t \in \R$.   One has a cohomology $H_t(M) = K(d_t) / I(d_t)$ since $d_t$ squares to the zero, and it is quickly verified that $H^*_t(M) \cong H^*(M)$, since  $d_t$ is obtained by conjugation.  As above, the Laplacians
\[
\ad_t = d_td_t^* + d^*_t d_t
\]
yield spectral decompositions $\aw^*(t) = \oplus_\lk \aw^*_\lk(t)$.   Likewise each $a>0$ yields a finite-dimensional complex of differential forms $\aw^*_a(t)$ spanned by eigenforms of $\ad_t$ with eigenvalues $\lk \le a$.

\newcommand{\cmf}{{\mathrm N}}
Witten argues that as  $t$ approaches $+\infty$ the spectrum of $\ad_t$  separates into a finite set of values  close to zero, and a complementary set of values much greater than zero.  More precisely, if $\cmf_q$ is the set of critical values of $f$ of index $q$, then there exists a $t$-parametrized family of maps $\psi_t: \cmf_q \to \aw_a^q(t)$ so that for large values of $t$,
\begin{enumerate}
\item $\psi_t(\cmf_q)$ is a basis of eigenfunctions for $\aw^q_a(t)$  and
\item $\psi_t(p)$ concentrates on a neighborhood of $p$, for each $p \in \cmf_q$.
\end{enumerate}

As Witten claimed and  Helder and Sj\"ostrand later confirmed \cite{helffer1987puits}, when $f$ is transversal the  induced codifferential on $\aw_a^q(t)$ may be calculated to vanishingly small error as $t \to +\infty$  by counting connecting orbits, in much the same spirit as that of the Morse complex.

\section{Discrete and Algebraic Morse Theory}
\label{sec:DMT}

The ingredients of the  complex introduced by Forman \cite{FormanMorse98} are conceptually close to those of the traditional Morse complex.  One begins with a topological CW complex, on which is defined some suitable function $f$.  From $f$  is derived a collection of critical cells (critical points) and integral curves (connecting orbits).   New cells form by combining critical cells with  the integral curves that descend from them, and the resulting family is organized into the structure of a complex via incidence relations.

Let us make this description  precise.   Posit a regular topological CW complex $X$, and let $X \der p$ denote the associated family of $p$-dimensional cells.  We write $\sk \der p$ when $\sk$ has dimension $p$ and $\tk > \sk$ when $\sk$ lies in the boundary of $\tk$.   A function
\[
f: X \to \R
\]
is a \emph{discrete Morse function}  if
\begin{align}
\# \{\tk \der{p+1} > \sk : f(\tk) \le f(\sk) \}& \le 1 \label{eq:morsecondition1}\\
\# \{\uk \der{p-1} > \sk : f(\uk) \ge f(\sk) \} &\le 1. \label{eq:morsecondition2}
\end{align}
for every  $\sk$.   We say $\sk\der p$ is \emph{critical} with index $p$ if equality fails in both cases, that is if
\begin{align*}
\# \{\tk \der{p+1} > \sk : f(\tk) \le f(\sk) \}& =0 \\
\# \{\uk \der{p-1} > \sk : f(\uk) \ge f(\sk) \} &= 0.
\end{align*}
The \emph{gradient vector field} of $f$ is the family 
$$
\nabla f = \left \{ (\sk \der p, \tk \der{p+1})  : \sk < \tk, \; \; f(\tk) \le f(\sk)\right \}.
$$
 It can be shown that
\begin{align}
\# \{(\sk, \tk) \in \nabla f : \sk  = \uk\;\; \mathrm{or} \;\; \tk = \uk\} \le 1 \label{eq:partialpairing}
\end{align}
for any $\uk \in X$, so $\nabla f$ is a partial matching on the incidence relation $<$.

Forman defines a \emph{discrete vector field} to be any partial matching  $R$ on $<$ for which \eqref{eq:partialpairing} holds with $R$ in place of $\nabla f$.    A \emph{gradient path} on $R$ is a sequence of form
\begin{align}
\ak_0 \der p, \bk_0 \der{p+1} , \ak_1 \der p, \bk_1 \der{p+1}, \ld, \bk_r \der{p+1}, \ak_{r+1} \der p
\end{align}
where 
\begin{align*}
(\ak_i, \bk_i) \in V && \ak_{i+1} < \bk_i && and && \ak_i \neq \ak_{i+1}
\end{align*} 
for all  $i \in \{0, \ld, r\}$.

Much like the Morse complex,  the discrete Morse complex has a basis indexed by critical cells and a differential expressed as a sum of values determined by integral curves.   One assigns an orientation to each cell of $X$, and defines the \emph{multiplicity} of a gradient path $\ck$ by
\[
m(\ck) = \prod_{p = 0}^{k-1} -\nr{\partial \bk_p, \ak_p} \nr{\partial \bk_p, \ak_{p+1}}.
\]
Let $C^f$ denote the free abelian group generated by the critical cells of $X$, graded by dimension, and let $\ag(\bk,\ak)$ denote the set of gradient paths that run from a maximal face of $\bk$ to $\ak$.  The \emph{discrete Morse boundary operator} is the degree $-1$ map on $\tilde \partial: C^f \to C^f$ defined by
\begin{align}
\tilde \partial \bk = \sum_{critical\; \ak \der p} c_{\ak, \bk} \ak,  \label{eq:morseformula}
\end{align}
where $c_{\ak, \bk} = \sum_{\ag (\bk, \ak)} m(\ck)$.

\begin{theorem}[Forman \cite{FormanMorse98}] 
\label{thm:formanmorsecomplex}  
The pair $(C^f, \tilde \partial)$ is a complex, and $H(C^f) \cong H(X, \Z)$.
\end{theorem}

The discrete Morse complex is a highly versatile algebraic tool, with natural applications in topology, combinatorics, and algebraic geometry.  The relation between smooth and discrete Morse theory has been refined by a number of results over the past two decades, e.g.\ \cite{gallais:hal-00264905}, where it was shown that every smooth Morse complex may be realized as a discrete one, via triangulation.   Given the breadth and depth of the correspondence between smooth and discrete regimes, it is reasonable to consider wether an analog to the Hodge-theoretic approach of Witten might exist for cellular spaces, as well.  This is indeed the case, as shown by Forman \cite{FormanWitten98}.

In the spectral setting, one begins with a suitably well-behaved CW complex $X$, equipped with a discrete Morse function $f$.  For each real number $t$ one defines an operator $e^{tf}$ on $C_*(X,\R)$ by 
$$
e^{tf} \sk = e^{t f(\sk)} \sk
$$
for $\sk \in X$.      To the associated  boundary operator one assigns a family of differentials
\[
\partial_t = e^{tf} \partial  e^{-tf}.
\]
to which correspond a $t$-parametrized family of chain complexes
\[
(C, \partial_t) : \; \; 0 \lr{} C_n \lr{\partial_t} C_{n-1} \lr{\partial_t} \cd \lr{\partial_t} C_0 \lr{} 0
\]
with associated Laplacians
\[
\ad (t) = \partial_t \partial_t^* + \partial_t^* \partial_t.
\]
Here $\partial_t^*$ denotes the linear adjoint to $\partial_t$ with respect to the inner product on $C_*(X, \R)$ for which  basis  $X$ is orthonormal.   The operator $\ad(t)$ is symmetric, hence diagonalizable, and  for each $\lk \in \R$ we denote the  $\lk$-eigenspace of $\ad(t)$ by
\[
E_p^\lk(t)  = \{c \in C_p : \ad (t) c = \lk c\}.
\]
Since $\partial_t \ad(t) = \ad(t) \partial_t$, operator $\partial_t$ preserves eigenspaces.  For each $\lk \in \R$ one therefore has a differential complex
\[
E^\lk(t) : \; \; 0 \lr{} E_n^\lk(t) \lr{\partial_t} E_{n-1}^\lk(t) \lr{\partial_t} \cd \lr{\partial_t} E_0^\lk(t) \lr{} 0.
\]
Forman defines a well-behaved class of Morse functions (\emph{flat} Morse functions) and shows that for flat $f$ 
\[
\ad(t) \to
\left (
\begin{array}{cc}
0 & 0 \\
0 & D
\end{array}
\right)
\]
as $t \to +\infty$, where $D$ is a block submatrix indexed by the set $\{\sk_1, \ld, \sk_k\}$ of critical cells of $f$, and $D$ has form
$$
\diag(a_{\sk_1}, \ld, a_{\sk_k}),
$$
for some  $a \in (\Z_{> 0} \cup \{+\infty\})^{\{\sk_1, \ld, \sk_k\}}$.

As $ t \to \infty$, therefore, the spectrum of $\ad(t)$ separates into a portion close to zero and a portion close to  one.  Fixing arbitrary $0 < \ek < 1$, we thus define the \emph{Witten complex} of $f$ to be
 $$
 \wc(t) = \oplus_{\lk < \ek} E^\lk(t)
 $$ 
writing
\[
\wc(t) : \; \;  0 \lr{} W_n(t) \lr{\partial_t} W_{n-1}(t) \lr{\partial_t} \cd \lr{\partial_t} W_0(t) \lr{} 0.
\]
for the natural grading induced by dimension.

Let $\pi(t)$ denote orthogonal projection from $C(X, \R)$ to $W_p(t)$, and for each critical $p$-cell $\sk$ put
\[
g_\sk(t) = \pi(t) \sk.
\]
One has $g_\sk(t) = \sk + O(e^{-tc})$ for some positive constant $c$.  The $g_\sk$ form a basis of $W_p(t)$ when $t$ is sufficiently large, but not an orthonormal one in general.  If we define $G$ to be the square matrix indexed by  critical $p$-cells  such that
$$
G_{\sk_1, \sk_2} = \nr{g_{\sk_1}, g_{\sk_2}},
$$
however, and 
$$
h_{\sk} = G^{-1/2} g_\sk,
$$
then the $h_\sk$ do form an orthonormal basis.
%$G$ to be the square matrix indexed by the critical $p$-cells of $X$ 
%\begin{align*}
%G_{\sk_1, \sk_2} = \nr{g_{\sk_1}, g_{\sk_2}}  && h_\sk = G^{-1/2} g_\sk.
%\end{align*}
%Evidently $G_{\sk_1, \sk_2} = \dk_{\sk_1, \sk_2} + O(e^{-tc})$ and $h_\sk = \sk + O(e^{-tc})$.  The $h_\sk$ form an orthonormal basis of $W_p(t)$.  
Forman's observation is that the matrix representation of $\partial_t$ with respect to this basis tends to that of the discrete Morse complex.

\begin{theorem}[Forman \cite{FormanWitten98}] 
\label{thm:morsewittenformancomplex}  
Suppose that $f$ is a flat Morse function.   Then for any critical  $\sk \der p$ and  $\tk \der p$ there exists a constant $c> 0$ such that
\[
\nr{\partial_t h_\tk, h_\sk}  = e^{t(f(\sk) - f(\tk))} \left[  \nr{\tilde \partial \tk, \sk } + O(e^{-tc}) \right] ,
\]
where $\tilde \partial$ is as in \eqref{eq:morseformula}.
\end{theorem}

\section{Morse-Witten Theory}
\label{sec:morseresult}

Forman's  Theorem \ref{thm:formanmorsecomplex} has been  extended in several directions.  Kozlov \cite{KozlovDiscrete05} gives a beautiful exposition of this result, illustrating that the discrete Morse complex may be realized as summand  of the initial complex obtained by a sequence of splitting operations  removing ``atomic'' subcomplexes with trivial homology.   Sk\"oldberg develops the most abstract version of which we are aware, which relies only on a direct sum decomposition of a complex of modules, together with some nondegeneracy requirements \cite{SkoeldbergMorse06}.  In this exposition the Morse complex is realized as the image of $\pi = \mathrm{id} - (\phi d + d \phi)$ where $\phi$ is a certain splitting homotopy.

These expositions share two properties in common: first, that the objects treated are differential complexes, and second, that the proofs are of an elementary but rather technical nature.  We believe that these issues are related.  In fact, we submit that the technical formalities found in these treatments owe not to the nature of the Morse construction, but to the  restriction of a general fact about nilpotent operators to the (in this case over-structured) special case of complexes.  

In evidence let us argue that Theorem \ref{thm:formanmorsecomplex} is not special to differentials, but holds for arbitrary 2-nilpotent operators on an object in an abelian category --  even one with no biproduct structure {\em a priori}.   In fact this was already shown in Theorem  \ref{prop:jordansplit}.    

\begin{theorem} \label{thm:abmorse} Let $(C,\partial)$ be any chain complex and $\wk$ any graded coproduct structure on $C$.  If 
$$
f: \wk \to \R
$$ 
is a discrete Morse function on $\wk$, then $\nabla f$ is an acyclic Jordan pairing on $\partial(E,E)$, and the Morse complex of $f$ is Jordan complement of $\nabla f$ in $\partial$.
\end{theorem}
\begin{proof}
One need only compare the discrete Morse complex \eqref{eq:morseformula} with the formula for the matrix representation of $\partial$ with respect to $\wk[[\nabla f]]$, as given by Proposition \ref{prop:jordancomplementmobiusformula}.
\end{proof}

We will show in the following section that Forman's Theorem \ref{thm:morsewittenformancomplex}, moreover, requires neither  that $\partial$ be graded nor that $\partial^2 = 0$.  It is rather a property of arbitrary operators on finite-dimensional real vector spaces.

Algebraic Morse theory has formed the foundation of several intensely active areas of mathematical research over the past two decades, and it is reasonable to ask why, given the high level of mathematical activity that surrounds it, the basic connection between the algebraic Morse complex and  classical matrix algebra, as outlined in Theorem \ref{thm:abmorse} and its proof,  have gone so long unremarked.   The elements of this proof, which is neither technical nor complicated, are M\"obius inversion and Theorem  \ref{prop:jordansplit}.  M\"obius inversion began to enter maturity with a foundational paper by Rota in 1964 \cite{rota1964}, and the essential ingredients of Theorem \ref{prop:jordansplit}, at least in the  case of complex  vector spaces, entered circulation no later than 1956 \cite{ptak1956}.

The strength of the ties that link these results to combinatorial topology, moreover, has been recognized for some time.  The notion of an acyclic relation is fundamental  both to the general theory of M\"obius inversion and to the classical formulation of algebraic Morse theory.   The text of Kozlov  \cite{KozlovCombinatorial08} devotes an entire chapter to the M\"obius transformation immediately ahead of the chapter on algebraic Morse theory -- though, so far as we are aware, no direct line is drawn between the two.  Even before the birth of algebraic Morse theory in its modern form, basic connections between the inversion formula and combinatorial topology were well understood, for example, vis-\`{a}-vis matroids and shellability \cite{white1992}.  

Connections with elementary exchange have been similarly understood.  This understanding is implicit in the various works that characterize  the Morse complex as a direct summand.  In the special case of  gradient fields with cardinality one, the fundamental result of algebraic Morse theory coincides exactly with a property of clearing operations described by Chen and Kerber in \cite{CKPersistent11}.  Precursors to this observation can be found even among the seminal papers on persistent homology computation, e.g.\ \cite{ZCComputing05}.   Exchange, moreover, plays a pivotal role in nearly every work on computational homology, especially so in the pioneering work of Mischaikow and Nanda \cite{MNMorse13}.   Adding to the body of interconnections, A.\ Patel has recently built on the work of Rota, Bender, Goldman, and Leinster  \cite{bender1975,leinster2012,rota1964} to extend the notion of constructible persistence to constructible maps into abelian groups, via M\"obius inversion \cite{patel2016}.

How, then, has the essential simplicity of the  Morse complex (and its relation to  M\"obius inversion, elementary exchange) gone  for so long unremarked?   This question is  ill posed for a several of reasons.  In the first place, while we have made every effort to find a reference on this subject, there necessarily  remains the possibility that one or more have been published (we would be grateful to hear of them).  In the second, while several authors have offered extremely helpful historical insights, the collective editorial decisions that shaped the body of published literature as it exists today are largely unknowable.  

Nevertheless, we regard it plausible that the genesis of Morse theory as a  tool in homological algebra may have influenced  its  circulation.  Morse theory assumed its first definite form in the setting smooth geometry, and entered   algebra by way of combinatorial topology.   The construction of the algebraic Morse complex closely mirrors that of an associated \emph{topological} complex, which is in general more complicated and carries much more information.  It would be unsurprising, given the extensive relations between the algebraic construction and its geometric and topological counterparts, that researchers in this area might infer a dependence relationship between the core  results of algebraic Morse theory and the underlying geometric, topological, or homological structure.

%A second consideration is that canonical form typically associate to \emph{linear} objects -- most often related to $\C$.   It is not often framed in the language of nilpotent maps on semisimple objects, and universally known to be inapplicable to arbitrary nilpotent endomorphisms in an Abelian category.   That the special case of Proposition \ref{prop:jordansplit} for $\C$-linear maps on complex vector spaces might naturally extend to arbitrary Abelian categories -- even on objects that have no biproduct decomposition \emph{a priori} -- might therefore require a shift in thinking.

The idea that this construction might be more general in nature came to us by way of matroid theory, and specifically matroid duality.   The theory of linear matroids relies extensively on the study of Schur complements and their relation to finite families of vectors, and it was through this lens that the  M\"obius transformation  came into view.    This clarification did much to simplify the story, however a direct understanding of the Morse complex as a \emph{Jordan complement} required the idea of a \emph{dual exchange}, for which we find traditional language of vectors and vector space duality ill suited.  It was the matroid theoretic notion of complementarity, that is \emph{matroid duality}, that illustrated the symmetry between primal and dual exchange operations that coalesce into a single Jordan pivot.  In fact, in this matter we consistently find  matroid duality  a more ready than the traditional venues for linear duality, e.g.\ tensors.
%While this may seem counterintuitive at first blush - after all matroids are inherently combinatorial objects, and relate to algebraic objects most often in the context of linear maps and vector spaces.   
%The matroid-theoretic formalization of duality that first made us aware of second half of the equation in elementary exchange -- that is, the role of exchanging \emph{dual} basis vectors, or more generally dual maps in a biproduct -- needed to understand the basic operation of a Jordan pivot.  And it was the theory of linear matroids which first made us aware of the basic connection between M\"obius inversion, algebraic Morse theory, and the Schur complement.   

These observations tell one part of a larger story, wherein combinatorial ideas unearth new truths about algebraic, analytic, topological, and geometric objects by providing recognizable indicators of structure.  We have seen this in the algebraic constructions of discrete Morse theory and will see it in the spectral/analytic constructions of Morse-Forman-Witten theory.  One piece of this story which deserves independent recognition is the Schur complement.   This is an extremely general \emph{categorical} object that has been viewed, historically, as an almost exclusively as a linear construction.  So much so, in fact, that so far as we are aware it has not been remarked anywhere in the literature that the Schur complement \emph{is} in fact a complement, in the categorical sense.   This despite the fact of the fundamental role played by the Schur complement in such diverse fields as probability, geometry, combinatorics, operator algebras, optimization, and matrix analysis.   That entire books have been written on this subject without reference to its categorical (if not historic) foundations speaks to the breadth and depth of opportunities to expand our understanding of these application areas from a categorical perspective.

Let us now examine the algebraic foundations of Forman's Theorem \ref{thm:morsewittenformancomplex}.  Posit an operator $T$ on a finite-dimensional real vector space $V$ with coproduct structure 
$$
\wk \su \Hom(\R, V).
$$  
We write $T^*$ for the linear adjoint to $T$ with respect to this structure, that is, the unique map for which 
$$
\sk^\sharp T^* \tk^\flat = \tk^{\sharp} T \sk^\flat
$$
for all $\sk, \tk \in E$.
  For economy we will sometimes drop the sharp and flat operators from the elements of $E$, writing, for example
$$
\sk T^*T \tk
$$
for $\sk^\sharp T T^*\tk^\flat$.  Context will make the intended meaning clear where this convention is invoked.
A function 
$$
f: E \to \R
$$
is \emph{monotone} if $R = \Supp(T(E,E))$ includes into $\sim _f$,  \emph{acyclic} if
$$
\d f = \{(\sk, \tk) \in R : f(\sk) = f(\tk)\}
$$
is an acyclic pairing on $R$, and \emph{canted} if $(\d f) \der \sharp \cap (\d f ) \der \flat  = \emptyset$.    We will assume a monotone, acyclic, canted $f$, throughout.  The elements of 
$$
\ag = \wk -\left( \d f \der \flat \cup \d f \der \sharp \right)
$$ 
we call \emph{critical cells}.     

For convenience we will sometimes substitute $\sk$ and $\tk$ for $f(\sk)$ and $f(\tk)$, for instance writing
\begin{align*}
\sk \le \tk
&& and &&
e^{-|\sk - \tk|}
\end{align*}
for $f(\sk) \le f(\tk)$ and $e^{-|f(\sk) - f(\tk)|}$, respectively.
Given $t \in \R$, we write $e^{tf}$ for the operator on $V$ such that
\begin{align*}
e^{tf} \sk = e^{tf(\sk)} \sk 
\end{align*}
for $\sk \in E$, and define
$$
T_t = e^{-tf} T e^{tf}.
$$
The associated \emph{Laplacian} is 
$$
\ad_t = T_t T_t^* + T_t^*T_t.
$$
As $\Supp\left (\lim_{t \to + \infty}  T_t \right ) = \d f$, the matrix representation of $\ad_t$ with respect to $E$ tends to an  array in $\R^{E \times E}$ supported on the diagonal of
$$
 \left (\d f\der \flat \cup \d f \der \sharp \right )\times \left (\d f\der \flat \cup \d f \der \sharp \right).
$$
The spectrum of $\ad_t$ therefore separates into a low sector of $|\wk| - 2 |\d f|$ eigenvalues which, for $t$ sufficiently large, lie in any open ball around zero, and a high sector bounded away from zero.   Following the convention of Forman, we write $W(t)$ for the subspace spanned by the eigenvectors in the low sector, $ \pi_t$ for orthogonal projection onto $W(t)$ and
\begin{align*}
g_\sk = \pi_t \sk.
\end{align*}
for each $\sk \in \ag$.  For $t$ sufficiently large,  $\{g_\sk : \sk  \in \ag \}$ forms a basis for $W(t)$, and  if for $\sk_1, \sk_2 \in \ag$ 
\begin{align*}
G_{\sk_1, \sk_2} = \nr{g_{\sk_1}, g_{\sk_2}} && h_{\sk_1} = G^{-1/2} g_{\sk_1} ,
\end{align*}
then $\{h_\sk : \sk  \in \ag \}$ is an \emph{orthonormal} basis for $W(t)$.

If  $g$ is a function $ \R \to \R$ and  $h$ is a function  $\R \to \R^m$, then we write
\begin{align*}
h \in o_e(g)  %\label{eq:specailonotation}
\end{align*}
when there exist real constants  $a < b$ such that $\nn{h}_2 \in O(e^{at})$ and $e^{bt} \in O(g)$.  Similarly we write 
$$
h \in O_e(g)
$$ 
when there exists a constant $c$ so that $\nn{h}_2 \in O(e^{tc})$ and $e^{tc} \in O(g)$.
If $g$ is a vector-valued function $\R \to \R^m$, then we write  $h \in o_e(g) $ if
\begin{align*}
h_i \in o_e(g_i)  && i = 1,\ld, m,
\end{align*}
where  $\pi_i$ is orthogonal projection onto the $i$th axis,  $h_i = \pi_i h$, and $g_i = \pi_i g$.

\renewcommand{\k}[1]{k_{#1}}
\newcommand{\n}{N}

Finally,  write $\k t$ for the idempotent  that projects onto the null space of 
$$
\left ( \times \d f \der \sharp \right ) T_t
$$
along the image of $\oplus \d f \der \flat$,  setting $k = k_0$ for economy.
It is simple to check that
$$k_t = e^{-tf} k e^{tf}$$
for all $t$.
% for $\times_{h \in \d f \der \sharp} h$, and  $e_{k}^t$ for the idempotent operator that projects onto the null space of $d f \der \sharp T$ along $\d f \der \flat$.  

Theorem \ref{thm:morsewittenformancomplex}  of Forman is  a special case of the following.

\begin{theorem} \label{thm:schurformanwitten} Under the stated conventions, 
$$
\nr{T_t h_\tk, h_\sk}  \in e^{\tk - \sk} (\tk T k \sk )  + o_e(e^{\tk - \sk}).
$$
for critical $\sk$ and $\tk$.
\end{theorem}
\newcommand{\alo}{T}

The remainder of this section is devoted to the proof of Theorem \ref{thm:schurformanwitten}.
Our argument will begin and end in a fashion almost identical to that of Forman  \cite{FormanWitten98}.  It will be greatly simplified, however, by observing first that the result holds for general $T$ -- not only for the boundary operators of real chain complexes -- and second that the formula for the limiting array may be expressed in terms of $k$.

By our initial assumptions there exists a positive constant $c$ such that
\begin{align*}
\ad^N_t = 
\begin{cases}
O(e^{-cNt}) & \mathrm{on} \;\;\;\;\; W(t) \\%[-16pt]
1 + O(e^{-ct}) & \mathrm{on} \;\;\;\;\; W(t)^\perp \\
\end{cases}
\end{align*}
hence
\begin{align*}
1-\ad^N_t = 
\begin{cases}
1+O(e^{-cNt}) & \mathrm{on} \;\;\;\;\; W(t) \\%[-16pt]
O(e^{-ct}) & \mathrm{on} \;\;\;\;\; W(t)^\perp \\
\end{cases}
\end{align*}
therefore
\begin{align*}
(1-\ad^N_t)^N = 
\begin{cases}
1+O(e^{-cNt}) & \mathrm{on} \;\;\;\;\; W(t) \\%[-16pt]
O(e^{-cNt}) & \mathrm{on} \;\;\;\;\; W(t)^\perp \\
\end{cases}
\end{align*}
for some positive constant $c$ and every positive integer $N$.  Consequently 
$$
(1 - \ad_t^N)^N = \pi_t + O(e^{-cNt})
$$
in the operator norm.  This  explicit construction provides all the traction we  need to reason about $\pi_t$.

\newcommand{\tms}{e^{-|\tk-\sk|}}
\begin{lemma}  One has
\begin{align}
\tk \alo_t^{*} \alo_t \sk = 
\begin{cases}
[\tk,\sk] \tms + o_e(\tms) & \sk \notin \d f\der \flat,  \tk \in \d f\der \flat,  \tk < \sk, \;\;\; \mathrm{or}  \\%%[-16pt]
& \sk \in \d f\der \flat, \tk \in \d f\der \flat \\
o_e(\tms) & \sk \notin \d f\der \flat, \tk \in \d f\der \flat, \tk \ge \sk, \;\;\; \mathrm{or}  \\%%[-16pt]
& \sk \notin \d f\der \flat, \tk \notin \d f\der \flat
\end{cases}
\label{eq:LtLest}
\end{align}
and
\begin{align}
\tk \alo _t \alo_t^{*} \sk = 
\begin{cases}
[\sk, \tk] \tms + o_e(\tms) & \sk \notin \d f\der\sharp, \tk \in \d f\der\sharp, \tk > \sk, \; \; \mathrm{or} \\%[-16pt]
& \sk \in  \d f\der\sharp, \tk \in \d f\der\sharp \\
o_e(\tms) & \sk \notin \d f\der\sharp, \tk \in \d f\der\sharp, \tk \le \sk, \; \; \mathrm{or} \\%[-16pt]
& \sk \notin \d f\der\sharp, \tk \notin \d f\der\sharp.
\end{cases}
\label{eq:LLtest}
\end{align}
\end{lemma}
\begin{proof}[Proof of \eqref{eq:LtLest}]  In each case $\tk \alo _t ^{*} \alo _t \sk$ is a sum of terms of form $[\ck,\tk][\ck, \sk] e^{\ck - \tk + \ck - \sk}$.    In the upper two cases nonzero terms satisfy $\ck \le \sk$ and $\ck \le \tk$.  Let us swap $\sk$ and $\tk$ as necessary in the second case, so that  $\tk \le \sk$ (we may do so since $\tk \alo_t^{*} \alo_t \sk = \sk \alo_t^{*} \alo \tk$).   Among the terms in our sum the greatest exponent can then be found on $[\tk, \sk] e^{\tk - \sk} =[\tk, \sk] e^{|\tk - \sk|}$, corresponding to $\ck = \tk$.  Thus the first two cases are justified.  In the lower two cases our sum runs over $\ck$ strictly less than $\sk$ \emph{and} $\tk$.  Since $\ck - \tk - \ck- \sk < -|\tk - \sk|$ for all such $\ck$, the lower expression is justified.
\end{proof}

%\begin{lemma}  One has
%\begin{align}
%\tk \alo _t \alo_t^{*} \sk = 
%\begin{cases}
%[\sk, \tk] \tms + o_e(\tms) & \sk \notin \d f\der\sharp, \tk \in \d f\der\sharp, \tk > \sk, \; \; \mathrm{or} \\%[-16pt]
%& \sk \in  \d f\der\sharp, \tk \in \d f\der\sharp \\
%o_e(\tms) & \sk \notin \d f\der\sharp, \tk \in \d f\der\sharp, \tk \le \sk, \; \; \mathrm{or} \\%[-16pt]
%& \sk \notin \d f\der\sharp, \tk \notin \d f\der\sharp.
%\end{cases}
%\label{eq:LLtest}
%\end{align}
%\end{lemma}
\begin{proof}[Proof of \eqref{eq:LLtest}]
As before, $\tk \alo \alo^{*} \sk$ is a sum of terms $[\tk, \ck][\sk, \ck] e^{\tk - \ck + \sk - \ck}$.  In each of the upper two cases, nonzero terms satisfy $\ck \ge \sk$ and $\ck \ge \tk$.  Let us swap $\sk$ and $\tk$ as necessary in the second case, so that  $\tk \ge \sk$ (we may do so since $\tk \alo_t \alo_t^{*} \sk = \sk \alo_t \alo^{*} \tk$).     Among the terms in our sum $[\sk, \tk] e^{\sk - \tk} = [\sk, \tk] e^{|\sk - \tk|}$ has the greatest exponent, corresponding to $\ck = \tk$, hence the first two cases.  All nonzero terms in the lower two cases satisfy $\ck < \sk$ \tm{and} $\ck < \tk$, so the exponent $  \tk - \ck + \sk - \ck $ is strictly lower than $-|\tk  - \sk|$.
\end{proof}

\begin{lemma}  \label{prop:secondestimate}
Suppose that $\sk \in \d f\der \flat$ and $\sk < \sk_0$, and fix any $\tk$.  Then
\begin{align}
\tk \alo_t ^{*} \alo_t \sk \cdot e^{-|\sk - \sk_0|} = 
\begin{cases}
[\tk, \sk] e^{-|\tk - \sk_0|} + o_e( e^{-|\tk - \sk_0|}) & \tk  \le \sk, \tk  \in \d f\der \flat  \\%[-16pt]
o_e(e^{-|\tk - \sk_0|}) & else.
\end{cases}
\end{align}
\end{lemma}

\begin{proof}
In either case one has 
$$
\tk \alo^{*}_t \alo_t \sk = [\tk, \sk] e^{-|\tk - \sk|} + o_e(e^{-|\tk - \sk|})
$$
by  \eqref{eq:LtLest}.  If $\tk \le \sk$ then  $-|\tk - \sk| - |\sk - \sk_0| = -|\tk  - \sk_0|$, and if $\sk < \tk$ then $-|\tk - \sk| - |\sk - \sk_0| < -|\tk  - \sk_0|$.
%The first expression holds for any $\tk \in \d f\der \flat$, by  \eqref{eq:LtLest}.  When $\tk > \sk$ the coefficient $[\tk, \sk]$ vanishes, by hypothesis.
%The condition $\tk \le \sk$ guarantees that $-|\tk - \sk| - |\sk - \sk_0| = -|\tk  - \sk_0|$.  Thus \eqref{eq:LtLest}  justifies the first case.  Likewise, case 3 justifies the lower expression when $\tk \le \sk$.  In the remaining case, $\tk > \sk$, one has $-|\tk - \sk| - |\sk - \sk_0| < -|\tk  - \sk_0|$.
\end{proof}

\begin{lemma}  \label{prop:dualtosecondestimate}
Suppose $\sk \in \d f\der \sharp$ and $\sk_0 < \sk$, and fix any $\tk$.  Then
\begin{align}
\tk \alo_t  \alo_t^{*} \sk \cdot e^{-|\sk - \sk_0|} = 
\begin{cases}
[\tk, \sk] e^{-|\tk - \sk_0|} + o_e( e^{-|\tk - \sk_0|}) & \sk  \le \tk, \tk  \in \d f\der \sharp  \\%[-16pt]
o_e(e^{-|\tk - \sk_0|}) & else.
\end{cases}
\end{align}
\end{lemma}
%\begin{proof}
%In either case one has 
%$$
%\tk \alo_t \alo_t^{*} \sk = [\tk, \sk] e^{-|\tk - \sk|} + o_e(e^{-|\tk - \sk|})
%$$
%by  \eqref{eq:LLtest}.  If $\sk \le \tk$ then  $-|\tk - \sk| - |\sk - \sk_0| = -|\tk  - \sk_0|$, and if $\tk < \sk$ then $-|\tk - \sk| - |\sk - \sk_0| < -|\tk  - \sk_0|$.
%%The first expression holds for any $\tk \in \d f\der \flat$, by  \eqref{eq:LtLest}.  When $\tk > \sk$ the coefficient $[\tk, \sk]$ vanishes, by hypothesis.
%%The condition $\tk \le \sk$ guarantees that $-|\tk - \sk| - |\sk - \sk_0| = -|\tk  - \sk_0|$.  Thus \eqref{eq:LtLest}  justifies the first case.  Likewise, case 3 justifies the lower expression when $\tk \le \sk$.  In the remaining case, $\tk > \sk$, one has $-|\tk - \sk| - |\sk - \sk_0| < -|\tk  - \sk_0|$.
%\end{proof}

%\begin{proof}
%The condition $\sk \le \tk$ guarantees that $-|\tk - \sk| - |\sk - \sk_0| = -|\tk  - \sk_0|$.  Thus, case 2 of Prop XXX (with $\sk$ and $\tk$ swapping places) justifies the upper expression.  Likewise, case 3 justifies the lower expression when $\sk \le \tk$.  In the remaining case, $\tk < \sk$, one has $-|\tk - \sk| - |\sk - \sk_0| < -|\tk  - \sk_0|$.
%\end{proof}

The proof of Lemma \ref{prop:dualtosecondestimate} is entirely analogous to that of Lemma \ref{prop:secondestimate}.  These observations bound scalars of form $\tk \alo_t ^{*} \alo_t \sk$ for $\sk \in \d f\der \flat$, and those of form  $\tk \alo_t  \alo_t^{*} \sk$ for $\sk \in \d f\der \sharp$.  The following control ``cross-terms,'' scalars of form  $\tk \alo_t ^{*} \alo_t \sk$ for $\sk \in \d f\der \sharp$, and those of form $\tk \alo_t  \alo_t^{*} \sk$ for $\sk \in \d f\der \flat$.

\begin{lemma}
If $\sk \in \d f\der \flat$ and $\sk < \sk_0$, then for any $\tk$,
\begin{align}
\tk \alo_t \alo_t^{*} \sk \cdot e^{-|\sk - \sk_0|}  = o_e(e^{-|\tk - \sk_0|}).  \label{eq:crossterm1}
\end{align}
Likewise, if $\sk \in \d f\der \sharp$ and $\sk > \sk_0$, then for any $\tk$,
\begin{align}
\tk \alo_t^* \alo_t \sk \cdot e^{-|\sk - \sk_0|}  = o_e(e^{-|\tk - \sk_0|}).  \label{eq:crossterm2}
\end{align}
\end{lemma}
\begin{proof}  Suppose $\sk \in \d f\der \flat$ and $\sk < \sk_0$.   If $\sk < \tk$ then $-|\sk - \tk|-|\sk - \sk_0| < -|\tk - \sk_0|$, and the desired conclusion follows.  All other cases relevant to \eqref{eq:crossterm1} are addressed by \eqref{eq:LLtest} directly.  Now suppose $\sk \in \d f\der \sharp$ and $\sk_0 < \sk$.  If $\tk < \sk$, then $-|\sk - \tk|-|\sk - \sk_0| < -|\tk - \sk_0|$, and the desired conclusion follows.  All other cases relevant to \eqref{eq:crossterm2} are addressed  by \eqref{eq:LtLest} directly.
\end{proof}

The following is a convenient repackaging of the  preceding remarks.    By abuse of notation, we will write $\ek$ both  for  the diagonal array such that 
$$
\ek(\sk, \sk) = e^{-|\sk - \sk_0|}.
$$
and for the  {tuple}  
$$
\ek(\sk) = e^{-|\sk - \sk_0|}.
$$
By a second overload, we will write   $\flat$   for  the linear operator on the base space of $\alo$  defined 
\begin{align*}
\flat(\sk) = 
\begin{cases}
\sk \der \flat & \sk \in \d f\der \sharp \\
0 & \sk \in E - \d f\der \sharp
\end{cases}
%&&
%\sharp(\sk) = 
%\begin{cases}
%\sk \der \sharp & \sk \in \d f\der \flat \\
%0 & \sk \in E - \d f\der \flat.
%\end{cases}
\end{align*}
and for the restriction of this map to the isomorphism
$$
\spann \left ( \d f\der \sharp \right  ) \longrightarrow \spann \left (\d f\der \flat \right).
$$
We define $\sharp$ similarly.  Evidently $\sharp^* = \flat$, both  as isomorphisms and as operators on the base space.

\begin{lemma} \label{prop:repackage1}  If $\ak$, $\wk$ are constant vectors supported on $\d f\der \flat_{< \sk_0}$ and $\d f\der \sharp_{>\sk_0}$, respectively, then for any $\n \ge 0$ one has
$$
\ad ^\n (\ek \ak + \ek \wk) = \ek(\flat \alo )^\n \ak + \ek (\flat^* \alo ^* )^\n \wk + o_e(\ek),
$$
the first term having support on $\d f\der\flat_{< \sk_0}$,  the second on $\d f\der\sharp_{>\sk_0}$.
\end{lemma}
\begin{proof}
The three preceding lemmas imply that $\flat \alo \ek \ak$ and $\flat^* \alo^* \ek \wk$ have support on $\d f\der\flat _{< \sk_0}$ and $\d f\der\sharp_{> \sk_0}$, respectively, and
\begin{align*}
\alo_t^* \alo_t (\ek \ak)  &=  \ek (\flat \alo \ak) + o_e(\ek) & \alo_t \alo_t^* (\ek \wk)  &=  \ek(\flat^* \alo \wk) + o_e(\ek)\\
 \alo_t^* \alo_t (\ek \wk)  &= o_e(\ek)  &  \alo_t \alo_t^* (\ek \ak)  &= o_e(\ek).  
\end{align*}
The desired conclusion follows by direct calculation.
\end{proof}

The $\ak$ and $\wk$ that will occupy our attention are those that derive from expressions of form $\flat^* \alo \sk_0$ and $\flat^* \alo^* \sk_0$.  Like the preceding observation, Lemma \ref{prop:secondrepackage} is a  convenient repackaging.

\begin{lemma}  \label{prop:secondrepackage}
If $\sk_0$ is critical then $\flat \alo \sk_0$ and $\flat^* \alo^* \sk_0$ have support on $\d f\der \flat _{< \sk_0}$ and $\d f\der \sharp_{> \sk_0}$, respectively, and
\begin{align}
\alo_t^* \alo_t \sk_0 =  \ek \flat \alo \sk_0 + o_e(\ek) &&  \alo_t  \alo_t ^* \sk_0 = \ek  \flat^* \alo ^* \sk_0 + O(\ek).  \label{eq:firstcolation}
\end{align}
\end{lemma}
\begin{proof}  That  $\flat \alo \sk_0$ and $\flat^* \alo^* \sk_0$ have support on $\d f\der \flat _{< \sk_0}$ and $\d f\der \sharp_{> \sk_0}$, respectively, follows directly from our hypothesis on $f$.  The left and righthand identities in \eqref{eq:firstcolation} are simply collations of some cases in  \eqref{eq:LtLest} and \eqref{eq:LtLest}, respectively.
\end{proof}

Lemma \ref{prop:synthesis} is a direct synthesis of the two preceding remarks.

\begin{lemma} \label{prop:synthesis}  If $\sk_0 \in y$ and $n$ is a nonnegative integer, then
\begin{align}
\ad^{\n+1}_t \sk_0 = \ek (\flat \alo )^\n(\flat \alo \sk_0) + \ek (\flat^* \alo ^*)^\n (\flat^* \alo^ * \sk_0) + o_e(\ek),  \label{eq:deltaform}
\end{align}
the first term having support on $x_{< \sk_0}$,  the second on $x^*_{>\sk_0}$.
\end{lemma}

We are now ready to state the main observation.  In preparation, let $\lk_\flat$ and $\lk_\sharp$ denote the inclusion  and projection maps
\begin{align*}
\spann \left ( \d f\der \flat \right)  \lr{} V && and && V \lr{} \spann \left( \d f\der \sharp \right ).
\end{align*}
%respectively, where $V$ is the base space of $\alo$, and  let  $e_k$ denote idempotent projection onto the null space of $\lk_\sharp \alo_t$ along 
%To emphasize dependence on $t$, we will write $e_{k_s}$ for the value of $e_k$ at $t = s$.  
Each $v \in V$ may be uniquely expressed  as the sum of an element in the image of
\begin{align}
\times(E - \d f\der \flat).% \label{eq:compcomponent}
\end{align}
and  one from the image of $ \times \d f\der \flat $.  We refer to the latter as the \emph{flat component} of $v$.

\begin{proposition} \label{prop:mainwittenprop}
For critical $\sk$ and $\tk$, 
$$
\nr{ \alo_t\pi_t \tk,  \pi_t \sk}  \in  e^{\tk - \sk} \tk \alo \k {} \sk + o_e(e^{\tk - \sk}).
$$
\end{proposition}

\begin{proof}
Let $\fk_\flat$ and $\fk_\sharp$ denote the isomorphisms
\begin{align*}
\fk_\flat = \flat \lk_\sharp \alo_t \lk_\flat && and && \fk_\sharp =  \flat^*\lk_\flat^* \alo_t^* \lk_\sharp^*
\end{align*}
respectively.  Under this convention \eqref{eq:deltaform} may be expressed
$$
\ad^{\n}_t(\ek \ak + \ek \wk) = \ek\fk_\flat^\n\ak  + \ek \fk_\sharp^\n \wk + o_e(\ek).
$$
If
$$
u = (1- \ad_t^\n)^\n \sk_0.
$$
%and put
%$
%\fk = \flat \com \lk_\sharp.
%$ 
% Evidently,  $\fk \alo_t \lk_\flat$ is an invertible self map on $\spann( \d f\der \flat)$.
%Noting that $u  \in W \sk_0 + O(e^{-cNt})$, we will first show that $u$  close to $e_{k} \sk_0$ in a certain sense.  The two do not agree in general, since
then  $\lk_\sharp \alo_t u$ does not in general vanish, however we claim that there is a ``small'' time-varying vector $v$ such that $\lk_{\sharp} \alo_t (u-v)$ does vanish, specifically, 
$$
v = \fk \inv_\flat (\flat  \alo_t) u.
$$

\noindent \emph{Claim:}  $v \in o_e(\ek)$.

\noindent \emph{Proof:}  There exist constant vectors $\ak$, $\wk$ supported on $\d f\der \flat_{<\sk_0}$ and $\d f\der \sharp_{>\sk_0}$, respectively, such that $\ad_t u = \ek \ak + \ek \wk + o_e(\ek)$.   Precomposition with $\ad^{N-1}$ yields $O(e^{-cNt})$ on the lefthand side, and $\ek \fk_\flat^{N-1} \ak + \ek \fk_{\sharp}^{N-1} \wk + o_e(\ek)$ on the right.   As $\fk_\flat$ and $\fk_\sharp$ are isomorphisms, it follows that $\ak$ and $\wk$  vanish for sufficiently large $N$.  When they do, $\ad_t u \in o_e(\ek)$.   The flat components of  $\ad_t u$ and $(\flat \alo_t) u$ agree up to an error of $o_e(\ek)$, so the latter is $o_e(\ek)$, also.  Moreover, it is simple to check that $\fk_\flat\inv $ sends $O_e(\ek)$ to $O_e(\ek)$ and, consequently, $o_e(\ek)$ to $o_e(\ek)$. (The operant observation is that  $\fk_\flat\inv $ is triangular, and its $(\tk, \sk)$ entry is proportional to $e^{-|\tk - \sk|}$).  Consequently $\fk_\flat\inv (\flat \alo_t) u \in o_e(\ek)$, which was to be shown.

Now, fix $\tk \in y$, and set $u_{\tk} = (1- \ad_t^N)^N \tk$.  Since
\begin{align*}
\nr{u_\tk, \alo_t u } = \nr{\pi \tk, \alo_t \pi \sk_0} + O(e^{-cNt}) && \nr{u_\tk, \alo_t  v} = O(e^{\tk - \sk}),
\end{align*}
our ultimate objective may be realized by establishing
\begin{align}
\nr{u_\tk, \alo_t (u+v) }  \in e^{\tk - \sk_0} \alo \k{0} \sk_0 + o_e(e^{\tk - \sk}). \label{eq:estimatedinnerproduct}
\end{align}
For this we require one further  observation.

\noindent \emph{Claim:}  $\nr{u_\tk, \alo_t \k{t} \hk} = e^{\tk - \hk} \tk \alo \k{0} \hk + o_e(e^{\tk - \hk})$, for any $\hk \in E- \d f\der \flat$.

\noindent \emph{Proof:} The inner product is a sum of terms that are either proportional to or dominated by $\mk \cdot e^{\mk - \hk - |\mk - \tk|}$, with $\mk$ running over all cells.  We consider the individual contribution of each term in three exhaustive cases: 

\noindent \emph{Case 1:} $\mk = \tk$.  This term contributes $e^{\tk - \ck} \tk \alo k \hk$.  

\noindent \emph{Case 2:}  $\mk \ge \hk$.  
Since $
\k{t} = e^{tf} \k{} e^{-tf},
$
one has   
$$
e^{tf} \alo \k{} e^{-tf} = \alo_t \k{t}.
$$  
As $\alo_t$ tends to zero on $E - \d f\der \flat$ and $\k{t}$ tends to orthogonal projection onto the span of $E - \d f\der \flat$, the product $\alo_t \k{t}$ tends to zero.  
%
%Since $x^*\alo_t \ov x$ tends to 0 as $t \to \infty$ and $x^* \alo_t x$ tends to a permutation matrix, $k_t$ tends to $\ov x$.  Thus $\alo_t e_k$ tends to zero.  
%Since $\alo_t$ tends to zero on $E - \d f\der \flat$ and $e_k$ tends to orthogonal projection onto the span of $\d f\der \flat$, the product $\lk_t e_k$ tends to zero.  
Consequently $\mk \alo_t \k{t} \hk = e^{\mk - \hk} \mk \alo \k{} \hk$ tends to zero.  When $\mk \ge \hk$, this is only possible if $\mk \alo \k{} \hk$ vanishes.  The contributions of all such $\mk$ are therefore vacuous. 

\noindent \emph{Case 3:} $\mk < \hk$.  If $\mk < \tk$ then the exponent $\mk - \hk - |\mk-\tk|$ is strictly lower than $-|\hk - \tk|$.   If $\tk < \mk$ then $\mk - \hk - |\mk-\tk| = |\hk - \tk|$, and we may consider three subcases: (a) $\mk \in \d f\der \sharp$.  By definition of $k$, $\mk \alo \k{} \hk$ vanishes.  (b) $\mk \notin  \d f\der \flat \cup \d f\der \sharp$.  The $\mk$-component of $u_\tk$ is strictly dominated by $e^{-|\mk - \tk|}$, so the contribution is strictly dominated by $e^{-|\hk - \tk|}$.  (c) $\mk \in \d f\der \flat$.  Since by assumption $ \tk < \mk$,  the $\mk$-component of $u_\tk$ is  strictly dominated by $e^{-|\mk - \tk|}$, hence the contribution is strictly dominated by  $e^{-|\mk - \tk|}$.

The stated claim follows.
 
Since $u + \uk = \k{t} u$, one has $\k{t} (u + \uk) = u+ \uk$, and therefore
$$
\alo_t(u+v) = \alo_t \k{t} (u+v).
$$
Recalling that $\k{t}$ annihilates  $\d f\der \flat$, we may express the righthand side as a sum of terms $\alo_t \k t \hk (u+v)$ with $\hk$ running over $E - \d f\der \flat$.  The second claim asserts that, up to negligible error, the inner product of $u_\tk$ with any such term is 
$$
e^{\tk - \hk }  [\hk(u+v)]  (\tk \alo \k{} \hk).
$$
If $\hk = \sk_0$ then $\hk (u + v) = 1 + o_e(1)$.  This term contributes $e^{\tk - \sk_0} \tk \alo \k{} \sk_0 + o_e(e^{\tk - \sk_0})$ to \eqref{eq:estimatedinnerproduct}.  
If $\hk \neq \sk$ and either $\hk \notin (\d f\der \flat \cup \d f\der \sharp)$ or $\hk  \in  \d f\der \sharp_{\le \sk_0}$, then $\hk(u+v) \in o_e(e^{-|\hk - \sk_0|})$, so its contribution is $o_e(e^{\tk - \sk_0})$.  If $\hk \in \d f\der \sharp_{> \sk_0}$ then  $e^{\tk - \sk_0}$ strictly dominates $e^{\tk - \hk}$.  As $\hk(u+v)$ is bounded, $e^{\tk - \sk_0}$ strictly dominates this contribution as well.  In summary, we have established \eqref{eq:estimatedinnerproduct}, which was to be shown.
\end{proof}

One may rephrase Proposition \ref{prop:mainwittenprop} as the statement that
$$
\nr{T_t g_\tk,  g_\sk} \in e^{\tk - \sk} \tk T \k{} \sk + o_e(e^{\tk - \sk})
$$
for critical $\sk$ and $\tk$.  We would like the same to hold for $\nr{T_th_\tk,  h_\sk}$.  To check that it does, let us first bound $G$.

\begin{proposition}
For critical $\sk_0$ and $\sk_1$, 
\begin{align}
G_{\sk_0, \sk_1} = \nr{g_{\sk_0}, g_{\sk_1}} = 
\begin{cases}
o_e(e^{-|\sk_0- \sk_1|}) & \sk_0 \neq \sk_1\\%[-16pt]
1 + o_e(1) & \sk_0 = \sk_1.
\end{cases} \notag
\end{align}
\end{proposition}
\begin{proof}
For $N$ sufficiently large,  the projection  $\pi \sk_0 = (1-\ad_t^N)^N \sk_0 + O(e^{-cNt})$ may be expressed
$
\sk_0 + \ek \ak + \ek\wk + o_e(\ek)
$
for some $\ak, \wk$ supported on $\d f \der \flat_{< \sk_0}$ and $\d f \der \sharp_{> \sk_0}$, respectively.  A similar expression may be derived for $\pi_{\sk_1}$.   The desired conclusion follows.
\end{proof}

As an immediate consequence, one has
\begin{align}
(G^{-1/2})_{\sk_0, \sk_1}  = 
\begin{cases}
o_e(e^{-|\sk_0- \sk_1|}) & \sk_0 \neq \sk_1\\%[-16pt]
1 + o_e(1) & \sk_0 = \sk_1.
\end{cases}  \notag
\end{align}
for critical $\sk_0, \sk_1$, whence
\begin{align*}
\nr{ T_t  h _\tk, h_\sk} & = \sum_{(\sk_1, \tk_1) \in \ag} (G^{-1/2})_{\tk \tk_1} \nr{ T_t g_{\tk_1}, g_{\sk_1}} (G^{-1/2})_{\sk_1\sk} \\
& = (G^{-1/2})_{\tk \tk} \nr{T_t g_\tk, g_\sk} (G^{-1/2})_{\sk \sk} + \sum_{\ag  - \{(\sk, \tk)\}}  (G^{-1/2})_{\tk \tk_1} \nr{ T_t g_{\tk_1}, g_{\sk_1}} (G^{-1/2})_{\sk_1\sk}.
\end{align*}
The first term is 
$$
e^{\tk - \sk} \tk T \k{} \sk + o_e(e^{\tk - \sk})
$$
while the second is a sum over $(\sk_1, \tk_1) \in \ag - \{(\sk, \tk)\}$ of terms in
$$
O_e(e^{-|\tk - \tk_1|}) O_e (e^{\tk - \sk}) O_e(e^{-|\sk - \sk_1|}).
$$
Each of these is $o_e(e^{\tk - \sk})$, so
$$
\nr{T_t h_\tk, h_\sk}  \in e^{\tk - \sk} \tk T \k{} \sk + o_e(e^{\tk - \sk})
$$
as desired.

\chapter{Abelian Matroids}

This chapter introduces a natural extension to the idea of a linear matroid representation: a \emph{(semisimple) abelian}  representation.  Several of the prime movers in linear representation theory --  deletion, contraction, and duality -- have natural analogs for abelian representations.  These objects are simpler and more general than their linear counterparts;  the relation between primal and dual representations, for example, is little more than a restatement of the (categorical) exchange lemma.   Duality is broadly recognized as a primary source of depth and structure in the study of matroids, and the light cast on this construction by the categorical approach speaks to real potential for interaction between the fields.

For the reader unfamiliar with the language of category theory, we invoke the conventions laid out at the opening to Chapter \ref{ch:exchangeformulae}.  The terms \emph{object} and \emph{morphism} may be replaced  by $\field$\emph{-linear vector space} and $\field$\emph{-linear map}, respectively.  The term \emph{map} is occasionally used in place of  morphism.   A \emph{monomorphism} is an injection and an \emph{epimorphism} is a surjection.   An \emph{endomorphism} on $W$ is a morphism $W \to W$, and an \emph{automorphism} is an invertible endomorphism.  With these substitutions in place, the phrases \emph{in a preadditive category} and \emph{in an abelian category} may be stricken altogether.

\section{Linear Matroids}

Let us recall the elements of linear representation theory outlined in \S\ref{sec:basisexchange}.  

\subsubsection{Linear representations}

A $\field$-\emph{linear representation} of a matroid $\mc = (E, \ic)$ is a function
$$
r: E \to W
$$
such that 
\begin{enumerate}
\item $W$ is a $\field$-linear vector space, and 
\item $S \su E$ is independent in $\mc$ iff $r|_S$ is a linearly independent indexed family in $W$. 
\end{enumerate}

\begin{remark} Independence for $r|_S$ implies independence for $r(S)$, but the converse may fail.   For example, suppose $\mc$ is the matroid on ground set $E = \{ 1, 2\}$ for which every subset of cardinality strictly less than two is independent.  Then the constant map $r: E \to R$ sending each element to 1 is a linear representation.  However $r |_E$ is not linearly independent as an indexed family in $\R$, while $r(E) = \{1 \}$ is.
\end{remark}

 To every linear representation $r$ and every $S \su E$ correspond a canonical \emph{restriction} and a non-canonical \emph{contraction}  operation (Lemma \ref{lem:linrepminors}).     We refer to $r|_S$ as the restriction of $r$ to $S$, and to $q \com \left(r|_{E-S} \right)$ as the \emph{contraction} of $r$ by $S$.  We do not define an operation to produce a matrix representation of  $\mc^*$, given $r$, either canonical or otherwise.

 \begin{lemma} \label{lem:linrepminors} If $r: E \to W$ is a linear representation of $\mc$, then
 \begin{align*}
 r|_{S}  && and && q \com \left(r|_{E-S} \right)
 \end{align*}
 represent $\mc|_S$ and $\mc/S$, respectively, where $q$ is any morphism such that
 $$
\K(q) =  \mathrm{\emph{span}}(S).
 $$
 \end{lemma}
 
 \subsubsection{Matrix representations}

A $\field$-linear \emph{matrix representation} of $\mc$ is an array $M \in \field^{I \times E}$ such that 
$$
r_M: E \to \field^I
$$
is a linear representation, where 
$$
r_M(e)(i) = M(i,e)
$$
for all $i \in I$ and  $e \in E$.

For every subset $S \su E$ the  restriction $M|_{I \times S}$ yields a canonical representation of $\mc|_E$.   There is no analogous operation to produce a canonical representation of $\mc/S$, however there are many non-canonical operations. To illustrate, fix subsets $\ak \su I$ and $\bk \su S$ such that $M(\ak, \bk)$ is invertible and $|\ak| = |\bk| = \rk(S)$.  For convenience identify each $i \in I$ with the standard unit vector $\chi_i \in \field^I$,  and let $T$ be any endomorphism on $\field^I$ such that
\begin{align*}
T r_M(\bk) = \ak && and && T(I- \ak) = I - \ak.
\end{align*}
If $U = q T$ is the postcomposition of $T$ with the deletion operator $q:\field^I \to \field^{I-\ak}$, then the kerenel of $U$ is the linear span of $r_M(S)$, hence  
\begin{align}
s = (U \com  r_M)|_{E - S}  \label{eq:s}
\end{align}
is a linear representation of $\mc/S$.  The array $N \in \field^{{(I - \ak)} \times {(E- S)}}$ defined
$$
N(i,e) = s(e)(i)
$$
is then a matrix representation $\mc/S$.  We do not define an operation to produce a matrix representation of  $\mc^*$, given $M$, either canonical or otherwise.

\subsubsection{Based Representations}

A \emph{based} representation is a pair $(r,B)$, where  $B \in \bc(M)$ and
$$
r: E \to \field^B
$$
satisfies
$$
r(b) = \chi_b
$$
for all $b \in B$.  Based representations inherit the canonical restriction operation $r \mapsto r|_S$ and the non-canonical contraction operation $r \mapsto s$, where $s$ is the representation defined by \eqref{eq:s} in the special case where  $M$ satisfies $r = r_M$.

Based representations have, in addition, a canonical dual operation $(r,B) \mapsto (r^*, E - B)$, where
$$
r: E \to \field^{E - B}
$$
satisfies
$$
r^*(e)(b) = r(b)(e)
$$
for all $e \in (E - B)$ and all $b \in B$.  That $r^*$  is a {bona fide}  representation for $\mc^*$ is a nontrivial fact of representation theory.

\subsubsection{Standard Representations}
 
 An array $M \in \field^{B \times (E - B)}$ is a \emph{$B$-standard} matrix representation of $\mc$ if the pair $(r,B)$ is a based representation, where $r: E \to \field^B$ is the function defined by
 \begin{align*}
 r(b) = \chi_b && r(e)(b) = M(b,e)
 \end{align*}
 for all $b \in B$ and all $e \in (E - B)$.
 
It is worth noting that standard representations are not matrix representations, since their column indices do not run over all of $E$.  Every standard representation uniquely determines a matrix representation, however, which may be characterized either as the unique array $N$ such that $r = r_N$, or more concretely by
$$
N = \left [\;\; \dk^B \;\; | \;\; M \;\;  \right],
$$
where $\dk^B$ is the Dirac delta on $B \times B$.

\begin{lemma} \label{lem:matroidlindual} If \emph{$M \in \field^{B \times (E - B)}$} is a standard representation for $\mc$, then 
$$
M ^* \in \mathrm{\emph{$\field$}}^{(E - B) \times B}
$$
is a standard representation for $\mc^*$.
\end{lemma}

\section{Covariant Matroids}

Let us fix a small abelian category $\cc$.    
Recall that an object in $\cc$  is \emph{simple} if it has exactly one proper subobject -- namely 0.   An object $W$ is \emph{semisimple} if it is isomorphic to a coproduct $\oplus V$, where $V$ is an indexed family of simple objects.  As a special case, the Jordan-H\"older theorem for abelian categories states that  such decompositions are essentially unique up to permutation.

\begin{theorem}[Jordan-H\"older]  If $V = (V_i)_{i = 1}^m$ and $W = (W_j)_{j \in J}$ are indexed families of simple objects in $\cc$ and
$$
\oplus V \cong  \oplus W,
$$
then there exists a bijection $\fk: \{1, \ld, m\} \to J$ such that
$$
V_p \cong W_{\fk(p)}
$$
for all $p \in \{1, \ld, m\}$.
\end{theorem}

In light of the Jordan-H\"older theorem, we say  any object isomorphic to the coproduct of $m$  simple objects has \emph{length $m$}.

\newcommand{\Sim}{\mathrm{Sim}}

Let $Y$ denote the class of  simple objects in $\cc$, and for each $W$ define
\begin{align*}
\Sim(*, W) = \bigcup_{y \in Y} \Hom(y,W)  && \Sim(W, *) = \bigcup_{y \in Y} \Hom(W,y).
\end{align*}
If $W$ has finite length, then we may define $\ic^\flat$ to be the class of all subfamilies $\lk \su \Sim(*, W)$ such that $\oplus \lk$ is a monomorphism.  Dually,  $\ic^\sharp$ is the class of all subfamilies $\uk \su \Sim(W,*)$ such that $\times \uk$ is an epimorphism.  

\begin{proposition}
The pairs
\begin{align*}
\left ( \Sim(*, W), \;\ic^\flat \right ) && and && \left ( \Sim(W,*), \;\ic^\sharp \right ) 
\end{align*}
are matroid independence systems.
\end{proposition}
\begin{proof}
It is evident that $\ic^\flat$ and $\ic^\sharp$ are closed under inclusion.  The Steinitz Exchange property follows from the Jordan-H\"older theorem, and from the fact that subobjects of semisimple objects are semisimple.
\end{proof}

We call submatroids of $\left ( \Sim(*, W), \;\ic^\flat \right ) $  \emph{covariant}, and submatroids of matroids of $\left ( \Sim(W,*), \;\ic^\sharp \right ) $ \emph{contravariant}.

\subsubsection{Abelian representations}

A \emph{covariant  representation} of a matroid $\mc = (E, \ic)$ is a function 
\begin{align*}
r: E \to \Sim(*,W)
\end{align*}
such that %($W$ is a semisimple object of finite length and)
\begin{align*}
\ic = \{I \su E: r|_I  \in \ic^\flat\}.
\end{align*}
A  \emph{contravariant  representation} is a  function
\begin{align*}
r: E \to \Sim(W,*)
\end{align*}
such that %$W$ is a semisimple object of finite length and
\begin{align*}
\ic = \{I \su E : r|_I  \in \ic^\sharp \}.
\end{align*}
By a slight abuse of terms, we will say that $r$ is \emph{semisimple} if $W$ is a semisimple object of finite length.

 To every abelian representation $r$ and every $S \su E$ corresponds a canonical \emph{restriction} operation $r \mapsto r|_S$.  While there exists no canonical \emph{contraction}, there do exist many non-canonical {contractions}  when $r$ is semisimple (Lemma \ref{lem:abrepminors}).     %We refer to $r|_S$ as the restriction of $r$ to $S$, and to $q \com \left(r|_{E-S} \right)$ as the \emph{contraction} of $r$ by $S$.  
 We do not define an operation to produce a matrix representation of  $\mc^*$ from a give representation  $r$, either canonical or otherwise.

 \begin{lemma} \label{lem:abrepminors} If $r$ is a semisimple covariant representation of $\mc$ and  $e$ is an idempotent  such that
 $$
\K(e) =  \I(\oplus r|_S), 
 $$
then
 \begin{align*}
 e \com \left(r|_{E-S} \right)
 \end{align*}
is a  semisimple  covariant representation of $\mc/S$.  Dually, if $r$ is a semisimple contravariant representation and $e$ is an idempotent such that 
 $$
\I(e) =  \K(\times r|_S),
 $$
then
  \begin{align*}
\left(r|_{E-S} \right) \com e
 \end{align*}
 is a semisimple contravariant representation of $\mc/S$.

 \end{lemma}
 
 \subsubsection{Morphic Representations}

Posit a morphism 
$$
T: W \to W\op.
$$
We say that a product structure $\lk\op$ on $W\op$ is semisimple if $\lk\op \su \Sim(W,\;*\;)$, and  a coproduct structure $\lk$ on $W$ is semisimple if $\lk \su \Sim(\;* \;, W )$.  

A  \emph{covariant morphic representation} is a pair $(T,\lk)$, where $\lk: E \to \Sim(\; * \;, W)$ is a semisimple coproduct structure and
$$
T \lk = (T\lk_e)_{e \in E}
$$
is a covariant  representation.  Dually, a  \emph{contravariant morphic representation} is a pair $(T,\lk\op)$,  where $\lk\op$ is a semisimple product structure on $W\op$ and 
$$
\lk\op T  = (\lk\op_e T)_{e \in E}
$$
is a contravariant  representation.  

The restriction and contraction operations described for linear matrix representations have natural analogs for abelian representations.  Restriction for a covariant representation may be achieved by replacing  $\lk$ with $\lk|_S$ and $W$ with $\I(\lk|_S)$, for example.  Contraction may be realized by first identifying a pair $(\ak, \bk)$ such that $\bk \su \lk$ and $T(\ak, \bk)$ is invertible and $|\ak| = |\bk| = \rk(S)$,  then substituting $T$ with the restriction to $\I(\lk|_{E - S})$ of $Te$, where $e$ is projection onto the kernel of $(\times \ak ) T$ along $\I(\oplus \bk)$.  

\subsubsection{Standard representations}

Let $B$ be a basis in $\mc$.  A \emph{covariant $B$-standard representation} of $\mc$  is a triple $(T,\lk, \uk)$, where
$$
T: W \to W\op,
$$
$T\lk : B \to \Sim(\; * \; , W\op)$ and $\uk : (E-B) \to \Sim(\; * \; , W)$ are coproduct structures, and $Tr$ is a covariant representation of $\mc$, where
\begin{align*}
r|_B = \lk && r|_{E - B} = \uk.
\end{align*}

Dually, a \emph{contravariant $B$-standard representation} of $\mc$ is a triple $(T, \lk\op, \uk\op)$, where $\lk\op T : B \to \Sim(  W\op, \; * \;)$ and $\uk\op : (E-B) \to \Sim( W,\; * \;)$ are product structures and $r\op T$ is a covariant representation of $\mc$, where
\begin{align*}
r|_B = \lk\op && r|_{E - B} = \uk\op.
\end{align*}
(Note that by a slight abuse of notation, we write $\lk\op T$ and  $r\op T$ for the functions $e \mapsto \lk\op_e T$ and  $e\mapsto r\op_e T$).

\begin{remark}  In the special case where $T \lk$ and $\uk$ are the canonical coproduct structures on $\field^B$ and $\field^{E-B}$, respectively, then the rule
$$
(T, \lk, \uk) \mapsto  T(\lk, \uk)
$$
determines a 1-1 correspondence between covariant $B$-standard representations and $\field$-linear $B$-standard matrix representations.  A similar statement holds for the corresponding dual structures.  Therefore Theorem \ref{thm:abmatroiddual} implies Lemma \ref{lem:matroidlindual}.
\end{remark}

The simplicity of the proof of Theorem \ref{thm:abmatroiddual} should be contrasted with the technical arguments found in most introductory treatments. 

\begin{theorem} \label{thm:abmatroiddual} If 
$$
(T, \lk, \uk)
$$ is a $B$-standard covariant representation of $\mc$, then 
\begin{align*}
(T,\lk^\sharp, \uk^\sharp) % \label{eq:startingdata}
\end{align*}
is an $(E-B)$-standard contravariant representation of $\mc^*$.   Dually, if $(T, \lk, \uk)$ is a $B$-standard contravariant representation of $\mc$, then 
$$
(T,\lk^\flat, \uk^\flat)
$$
is an $(E-B)$ standard covariant representation of $\mc^*$.
\end{theorem}

\begin{proof}  Let $C$ be any basis of $\mc$, and let
\begin{align*}
\ak\op = B - C && \ak = C - B.
\end{align*}
Submatrix  $T(\ak\op, \ak)$ is invertible, so the exchange lemma  provides that 
$$
%(B-C)^\sharp  T \cup ((E -B)- (C-B))^\sharp = 
(B-C)^\sharp T \cup ((E-B)-C)^\sharp
$$
is a product structure.  Thus $E-C$ is independent in the matroid represented by $(T, \lk^\sharp, \uk^\sharp)$.  This argument is easily reversed, and the desired conclusion follows.
\end{proof}

\bibliographystyle{acm}
\bibliography{EATbib.bib}{}
% An alphabetical listing of all references must be used.
%\end{thebibliography}

\end{document}